%% file: optim_maint_sched_v2.tex
\documentclass[a4paper,11pt]{article}
\usepackage[left=2.1cm, right=2.1cm, top=2.5cm, bottom=2.5cm,includefoot,footskip=30pt]{geometry}
\input{preambule}

\begin{document}

\maketitle

\begin{abstract}
  Optimizing maintenance scheduling is a major issue to improve the performance
  of hydropower plants. We study a system of several physical components
    of the same family: either a set of turbines, a set of transformers or a set
    of generators. The components share a common stock of spare parts and
  experience random failures that occur according to known failure
  distributions. We seek a deterministic preventive maintenance strategy that
  minimizes an expected cost depending on maintenance and forced outages of the
  system. The Auxiliary Problem Principle is used to decompose the original
  large-scale optimization problem into a sequence of independent subproblems of
  smaller dimension while ensuring their coordination. Each subproblem consists
  in optimizing the maintenance on a single
  component. Decomposition-coordination techniques are based on
    variational techniques but the maintenance optimization problem is a
    mixed-integer problem. Therefore, we relax the dynamics and the cost
    functions of the system. The resulting algorithm iteratively solves the
    subproblems on the relaxed system with a blackbox method and coordinates the
    components. Relaxation parameters have an important influence on the
  optimization and must be appropriately chosen. An admissible maintenance strategy is then derived from the resolution of the relaxed problem. We apply the decomposition
  algorithm on a system with 80 components. It outperforms the reference
  blackbox method applied directly on the original problem.
\end{abstract}

\input{introduction}

\input{description}

\input{app_decomp}

% \input{toy_appli}

\input{aux_pb_indus}

\input{numerical}

\input{conclusion}

\appendix

\input{appendix}

\bibliographystyle{siamplain}      % mathematics and physical sciences
\bibliography{these}   % name your BibTeX data base

\end{document}

%% file: preambule.tex
\usepackage[utf8]{inputenc}
\usepackage{amsfonts}
\usepackage{amsthm}
\usepackage{amsmath}
\usepackage{amssymb}
\usepackage{graphicx}

\usepackage{epstopdf}
\usepackage{algpseudocode}

\usepackage{tikz}
\usepackage[framemethod=TikZ]{mdframed}

\usepackage{tabularx}
\usepackage{booktabs}
\usepackage[ntheorem]{empheq}
\usepackage{dsfont}
\usepackage{eurosym}
\usepackage[mathscr]{euscript}
\usepackage{enumitem}
\usepackage{algorithm}
\numberwithin{equation}{section}

\numberwithin{equation}{section}

\theoremstyle{plain}
\newtheorem{theorem}{Theorem}[section]
\newtheorem{lemma}[theorem]{Lemma}
\newtheorem{proposition}[theorem]{Proposition}

\theoremstyle{definition}
\newtheorem{definition}[theorem]{Definition}

\theoremstyle{remark}
\newtheorem{remark}[theorem]{Remark}
\newtheorem{example}[theorem]{Example}

\usetikzlibrary{shapes,shadows,calc, decorations.pathmorphing, backgrounds, fit, decorations.pathreplacing, positioning, arrows, matrix}

\tikzstyle{unit} = [rectangle, draw, line width=1pt, minimum height=1.2cm]
\tikzstyle{comp} = [rectangle, draw=blue!70!black, fill=blue!50, line width=1pt, minimum width=3cm, inner sep=1mm, align=center]
\tikzstyle{corrm} = [ellipse, draw=red!70, fill=red!50, line width=1pt,  minimum width=0.5cm, inner sep=0pt]
\tikzstyle{pm} = [ellipse, draw=green!70!black, fill=green!50, line width=1pt, minimum width=0.5cm, inner sep=0pt]
\tikzstyle{fo} = [ellipse, draw=red, fill=red!90!blue, line width=1pt, minimum width=0.5cm, inner sep=0pt]
\tikzstyle{event} = [rounded rectangle, line width=1pt, draw, minimum width=2cm, minimum height=0.8cm]
\tikzstyle{part_ord} = [rectangle, line width=1pt, draw, minimum height=0.5cm]
\tikzstyle{brokenstate} = [rectangle, fill=blue!50, draw=red!60, font=\small, align=center, draw, line width=1pt, inner sep=1mm, minimum width=4cm, fill opacity=0.5, text opacity=1]
\tikzstyle{healthystate} = [rectangle, fill=blue!50, draw=green!70!black, font=\small, align=center, draw, line width=1pt, inner sep=1mm, minimum width=4cm, fill opacity=0.5, text opacity=1]
\tikzstyle{control} = [ellipse, fill=green!50, draw=green!70!black, font=\small, align=center, draw, line width=1pt, inner sep=2pt, fill opacity=0.5, text opacity=1]
\tikzstyle{spare} = [rectangle, draw=red!70, fill=red!50, font=\small, align=center, draw, line width=1pt, inner sep=1mm, fill opacity=0.5, text opacity=1]

\tikzstyle{system} = [ellipse, draw=blue!60!, fill=blue!50, line width=1pt, minimum width=3cm, minimum height=1.5cm, inner sep=1mm, align=center]
\tikzstyle{cycle_elem} = [rectangle, rounded corners=1mm, fill=blue!70!green!50, line width=1pt, minimum width=1.6cm, minimum height=0.65cm, inner sep=1mm, outer sep=0, align=center]

\usepackage{amsopn}

\input{commandes}

\newcommand{\dis}[2]{\mathrm{dist}\np{#1,#2}}

\newcommand{\indic}[2]{\findi{#1}(#2)}
\newcommand{\indicrel}[2]{\findi{#1}\rel(#2)}
\newcommand{\indicrelder}[2]{\findi{#1}^{'\alpha}(#2)}

\newcommand{\expec}[1]{\espe(#1)}

\newcommand{\bbRp}{\mathbb{R}_{+}}
\newcommand{\bbRpe}{\mathbb{R}_{+}^{*}}

\newcommand{\reg}[2]{\va{E}_{#1,#2}}
\newcommand{\age}[2]{\va{A}_{#1,#2}}
\newcommand{\spr}[2]{\va{P}_{#1,#2}}

\newcommand{\regall}{\va{E}}
\newcommand{\ageall}{\va{A}}

\newcommand{\sprnofail}{\delta}
\newcommand{\delay}{D}

\newcommand{\stateallcomp}{\va{X}}
\newcommand{\stateallcompbar}{\bar{\stateallcomp}}
\newcommand{\stateallcompopt}{\stateallcomp\opt}
\newcommand{\statecomponent}[1]{\stateallcomp_{#1}}

\newcommand{\statecompopt}[1]{\stateallcompopt_{#1}}
\newcommand{\statecompbar}[1]{\stateallcompbar_{#1}}
\newcommand{\statecomptime}[2]{\stateallcomp_{#1,#2}}
\newcommand{\statecomptimeopt}[2]{\stateallcompopt_{#1,#2}}
\newcommand{\statecomptimebar}[2]{\stateallcompbar_{#1,#2}}

\newcommand{\stoalltime}{\va{S}}
\newcommand{\stoalltimebar}{\bar{\stoalltime}}
\newcommand{\stoalltimeopt}{\stoalltime\opt}
\newcommand{\sto}[1]{\stoalltime_{#1}}
\newcommand{\stoopt}[1]{\stoalltimeopt_{#1}}
\newcommand{\stobar}[1]{\bar{\stoalltime}_{#1}}

\newcommand{\noiseallcomp}{\va{W}}
\newcommand{\noise}[2]{\noiseallcomp_{#1,#2}}
\newcommand{\noisecomp}[1]{\noiseallcomp_{#1}}

\newcommand{\ctrlallcomp}{u}
\newcommand{\ctrlallcompbar}{\bar{u}}
\newcommand{\ctrlallcompopt}{u\opt}
\newcommand{\ctrlcomp}[1]{\ctrlallcomp_{#1}}
\newcommand{\ctrlcompbar}[1]{\ctrlallcompbar_{#1}}
\newcommand{\ctrlcompopt}[1]{\ctrlallcompopt_{#1}}
\newcommand{\ctrlcomptime}[2]{\ctrlallcomp_{#1,#2}}

\newcommand{\ctrlcomptimeopt}[2]{\ctrlallcompopt_{#1,#2}}

\newcommand{\multallcomp}{\va{\Lambda}}
\newcommand{\multallcompbar}{\bar{\multallcomp}}
\newcommand{\multallcompopt}{\multallcomp\opt}
\newcommand{\multcomp}[1]{\multallcomp_{#1}}
\newcommand{\multcompopt}[1]{\multallcompopt_{#1}}
\newcommand{\multcompbar}[1]{\multallcompbar_{#1}}
\newcommand{\multcomptime}[2]{\multallcomp_{#1,#2}}
\newcommand{\multcomptimeopt}[2]{\multallcompopt_{#1,#2}}
\newcommand{\multcomptimebar}[2]{\multallcompbar_{#1,#2}}

\newcommand{\components}{\mathbb{I}}
\newcommand{\timesteps}{\mathbb{T}}
\newcommand{\timestepsnoend}{\mathbb{T}_{-1}}

\newcommand{\spares}{\{0, \ldots, s\}}

\newcommand{\cpm}{C^{P}}
\newcommand{\ccm}{C^{C}}
\newcommand{\cfo}{C^{F}}
\newcommand{\costpm}{j^{P}}
\newcommand{\costcm}{j^{C}}
\newcommand{\costfo}{j^{F}}

\newcommand{\costmrel}{j\rel}
\newcommand{\costforel}{j^{F, \alpha}}

\newcommand{\costadd}{J_{\Sigma}}
\newcommand{\costaddi}{J_{\Sigma, i}}
\newcommand{\costdiff}{J_{\Delta}}

\newcommand{\dyn}{\Theta}
\newcommand{\dynrel}{\Theta\rel}
\newcommand{\dyncomp}[1]{\Theta_{#1}}
\newcommand{\dyncomprel}[1]{\Theta_{#1}\rel}
\newcommand{\dyncomptime}[2]{\Theta_{#1,#2}}

\newcommand{\auxdyn}{\Phi}
\newcommand{\auxdynrel}{\Phi\rel}
\newcommand{\auxdyncomp}[1]{\Phi_{#1}}
\newcommand{\auxdyncomprel}[1]{\Phi_{#1}\rel}
\newcommand{\auxdyncomptime}[2]{\Phi_{#1,#2}}

\newcommand{\dersto}[1]{\nabla_{\sto{#1}}}

\newcommand{\lag}{L}
\newcommand{\rel}{^{\alpha}}

\newcommand{\disc}[1]{\eta_{#1}}

\newcommand{\sucht}{\mathrm{s.t.}\ }

\newcommand{\primspace}{\espacef{X} \times \espacef{S} \times \espacea{U}}
\newcommand{\primdecomp}[1]{\espacef{X}_{#1} \times \espacea{U}_{#1}}

\newcommand{\optimspace}{(\stateallcomp,\stoalltime,\ctrlallcomp) \in \primspace}

\newcommand{\prevopt}{\stateallcompbar, \stoalltimebar, \ctrlallcompbar}

\algblock{ParFor}{EndParFor}
\algnewcommand\algorithmicparfor{\textbf{for all}}
\algnewcommand\algorithmicpardo{\textbf{do in parallel:}}
\algnewcommand\algorithmicendparfor{\textbf{end\ for}}
\algrenewtext{ParFor}[1]{\algorithmicparfor\ #1\ \algorithmicpardo}
\algrenewtext{EndParFor}{\algorithmicendparfor}

\makeatletter
    \newcommand*{\algrule}[1][\algorithmicindent]{\makebox[#1][l]{\hspace*{.5em}\thealgruleextra\vrule height \thealgruleheight depth \thealgruledepth}}%
\newcommand*{\thealgruleextra}{}
\newcommand*{\thealgruleheight}{.75\baselineskip}
\newcommand*{\thealgruledepth}{.25\baselineskip}

\newcount\ALG@printindent@tempcnta
\def\ALG@printindent{%
    \ifnum \theALG@nested>0% is there anything to print
        \ifx\ALG@text\ALG@x@notext% is this an end group without any text?
        \else
            \unskip
            \addvspace{-1pt}% FUDGE to make the rules line up
            \ALG@printindent@tempcnta=1
            \loop
                \algrule[\csname ALG@ind@\the\ALG@printindent@tempcnta\endcsname]%
                \advance \ALG@printindent@tempcnta 1
            \ifnum \ALG@printindent@tempcnta<\numexpr\theALG@nested+1\relax% can't do <=, so add one to RHS and use < instead
            \repeat
        \fi
    \fi
    }%
\usepackage{etoolbox}
\patchcmd{\ALG@doentity}{\noindent\hskip\ALG@tlm}{\ALG@printindent}{}{\errmessage{failed to patch}}
\makeatother

\newbox\statebox
\newcommand{\myState}[1]{%
    \setbox\statebox=\vbox{#1}%
    \edef\thealgruleheight{\dimexpr \the\ht\statebox+1pt\relax}%
    \edef\thealgruledepth{\dimexpr \the\dp\statebox+1pt\relax}%
    \ifdim\thealgruleheight<.75\baselineskip
        \def\thealgruleheight{\dimexpr .75\baselineskip+1pt\relax}%
    \fi
    \ifdim\thealgruledepth<.25\baselineskip
        \def\thealgruledepth{\dimexpr .25\baselineskip+1pt\relax}%
    \fi
    \State #1%
    \def\thealgruleheight{\dimexpr .75\baselineskip+1pt\relax}%
    \def\thealgruledepth{\dimexpr .25\baselineskip+1pt\relax}%
}
\newlength\myindent % define a new length \myindent
\usepackage{hyperref}
\ifpdf
\hypersetup{
  colorlinks=false,
  frenchlinks=false,
  pdfborder={0 0 0},
  naturalnames=false,
  hypertexnames=false,
  breaklinks,
  pdftitle={A Decomposition Method by Interaction Prediction for the Optimization of Maintenance Scheduling},
  pdfauthor={T. Bittar, P. Carpentier, J-Ph. Chancelier, and J. Lonchampt}
}
\fi

\title{A Decomposition Method by Interaction Prediction for the Optimization of Maintenance Scheduling
}

\author{T. Bittar\thanks{CERMICS, \'{E}cole des Ponts ParisTech, 6 et 8 avenue Blaise Pascal, 77455 Marne la Vallée Cedex 2, France}  \textsuperscript{,$\ddagger$}
\and P. Carpentier\thanks{Unité de Mathématiques Appliquées, ENSTA Paris, 828 Boulevard des Maréchaux, 91762 Palaiseau Cedex, France
  }
\and J-Ph. Chancelier\footnotemark[1]
\and J. Lonchampt\thanks{EDF R\&D PRISME, 6 quai Watier, 78400 Chatou, France} }
\date{}

%% file: commandes.tex
\def\mathscr{\EuScript}

\newcommand{\bbN}{\mathbb{N}}                               % Entiers naturels
\newcommand{\bbR}{\mathbb{R}}                               % Nombres r\'{e}els
\newcommand{\norm}[1]{\left\|#1\right\|}                    % Norme
\newcommand{\sqnorm}[1]{\left\|#1\right\|^{2}}              % Norme au carr\'{e}

\newcommand{\transpos}{^{\top}}                               % Transposition
\newcommand{\dual}{^{\star}}                                % Dualit\'{e}
\newcommand{\opt}{^{\sharp}}                                % Optimalit\'{e}
\newcommand{\ad}{^{\mathrm{ad}}}                            % Admissibilit\'{e}
\newcommand{\compl}{^{\mathrm{c}}}                          % Compl\'{e}mentarit\'{e}

\newcommand{\np}[1]{(#1)}                                   % Parenth\`{e}se normal
\newcommand{\bp}[1]{\big(#1\big)}                           % Parenth\`{e}se big
\newcommand{\Bp}[1]{\Big(#1\Big)}                           % Parenth\`{e}se Big
\newcommand{\na}[1]{\{#1\}}                                 % Accolade normal
\newcommand{\ba}[1]{\big\{#1\big\}}                         % Accolade big
\newcommand{\espace}[1]{\mathbb{#1}}                        % Espace de Hilbert
\newcommand{\partie}[1]{#1}                                 % Sous-ensemble
\newcommand{\findi}[1]{\mathbf{1}_{#1}}                     % Fonction indicatrice
\newcommand{\espacea}[1]{\mathbb{#1}}                       % Espace d'arriv\'{e}e
\newcommand{\espacef}[1]{\mathcal{#1}}                      % Espace fonctionnel
\newcommand{\tribu}[1]{\mathscr{#1}}                        % Tribu
\newcommand{\omeg}{\Omega}                                  % espace du triplet
\newcommand{\trib}{\tribu{F}}                               % tribu  du triplet
\newcommand{\prbt}{\mathbb{P}}                              % proba  du triplet
\newcommand{\espe}{\mathbb{E}}                              % Symbole esp\'{e}rance
\makeatletter
\def\va@a{\boldsymbol{\va@arg^{\textstyle\text{\unboldmath$\scriptstyle\va@expo$}}_{\textstyle\text{\unboldmath$\scriptstyle\va@index$}}}}
\def\va#1{\def\va@expo{}\def\va@index{}\def\va@arg{\uppercase{#1}}%
  \@ifnextchar^{\va@h}{\@ifnextchar_\va@u\va@a}}
\def\va@h^#1{\def\va@expo{#1}\@ifnextchar_\va@hu\va@a}
\def\va@u_#1{\def\va@index{#1}\@ifnextchar^\va@uh\va@a}
\def\va@hu_#1{\def\va@index{#1}\va@a}
\def\va@uh^#1{\def\va@expo{#1}\va@a}
\makeatother

\newcommand{\Bigdelim}[1]{\Bp{#1}}                          % Taille ``Big''
\newcommand{\vardelim}[1]{\left(#1\right)}                  % Taille ``variable''

\newcommand{\Besp}[2][]{\espe_{#1}\Bigdelim{#2}}            % Esp\'{e}rance Big
\newcommand{\proscal}[2]{\left\langle#1\:,#2\right\rangle}  % Produit scalaire

\newcommand{\nabs}[1]{|#1|}                                 % Valeur absolue
\newcommand{\lsc}{\text{l.s.c.}}                            % ``l.s.c.''
\newcommand{\as}{\text{a.s.}}                               % ``a.s.''
\newcommand{\Pas}{\text{$\prbt$-}\as}                       % ``P-a.s.''
\def\eqsepv{\; , \enspace}                                  % Virgule dans une \'{e}quation
\def\eqfinv{\; ,}                                           % Virgule en fin d'\'{e}quation
\def\eqfinp{\; .}                                           % Point en fin d'\'{e}quation
\def\eqfinpv{\; ;}                                          % Point-virgule en fin d'\'{e}quation

\newcommand{\Uad}{\partie{U}\ad}

\newcommand{\finpreuvesymb}{$\Box$}%symbole fin preuve
\newcommand{\finremarksymb}{$\Diamond$}%symbole fin remarque
\newcommand{\finexemplesymb}{$\triangle$}%symbole fin exemple
\newcommand{\finpreuve}{\ \hspace*{\fill}\finpreuvesymb}
\newcommand{\finremark}{\ \hspace*{\fill}\finremarksymb}
\newcommand{\finexemple}{\ \hspace*{\fill}\finexemplesymb}
\newcommand{\skipfinpreuve}{\renewcommand{\finpreuve}{}}
\newcommand{\skipfinremark}{\renewcommand{\finremark}{}}

\makeatletter
\def\endproof{\finpreuve\@endtheorem}
\def\endremark{\finremark\@endtheorem}
\def\endexample{\finexemple\@endtheorem}
\makeatother

%% file: introduction.tex
% !TEX root = optim_maint_sched_v2.tex

\section{Introduction}

In industry, maintenance aims at improving the availability of physical assets and therefore impacts the overall performance of a system. There exists two main kinds of maintenance: corrective and preventive. Corrective maintenance (CM) is performed in reaction to a breakdown. Preventive maintenance (PM) is a planned operation that consists in repairing or replacing a component before a failure. Maintenance policies have an important economic impact and are therefore studied in various areas such as the electricity sector~\cite{froger_maintenance_2016}, the manufacturing industry~\cite{ding_maintenance_2015} or civil engineering~\cite{sanchez-silva_maintenance_2016}. In the electricity sector, maintenance optimization plays a major role in ensuring a reliable and competitive electricity production. 

In this work, we consider components of hydroelectric power plants that can be either turbines, transformers or generators. The system of interest gathers only components of the same family sharing a common stock of spare parts. This means that we consider either a system with turbines, a system with transformers or a system with generators. The time horizon is $40$ years. Over time, components experience random failures that occur according to known failure distributions. Thus, the dynamics of the system is stochastic. A preventive strategy consists in choosing the dates of replacement for each component of the system. The goal is to find a preventive strategy that minimizes an expected cost depending on maintenance and on the occurrences of forced outages of the system. The numerical experiments should involve systems constituted of up to $80$ components in order to model the most demanding industrial case. This leads to optimization problems in high dimension that are numerically challenging.

The framework of this paper is close to~\cite{patriksson_stochastic_2015} where a stochastic opportunistic replacement problem with a maintenance decision at each time step is studied. 
% The focus of~\cite{patriksson_stochastic_2015} is to take advantage of the failure of one component to perform PMs on other components.
Here, the focus is to take into account the couplings that are induced by a common stock of spare parts and the cost due to the forced outages of the system. Some features make our problem singular:
\begin{enumerate}
    \item We consider cutting-edge components for which there is a non-negligible delay between the order of a spare part and its arrival in the stock (several months, or even years).
    \item Operational constraints impose to only look for deterministic maintenance strategies. This means that the dates of PM are chosen at the beginning of the time horizon with only a statistical knowledge of the future dates of failure: this is referred as an open-loop strategy. This differs from condition-based maintenance~\cite{olde_keizer_condition-based_2017,rahmati_novel_2018} where maintenance decisions are taken given the degradation state of the components. Indeed, as the decisions depend on online observations, a condition-based maintenance strategy is stochastic.
    \item Many studies consider periodic~\cite{roux_development_2008,sarker_optimization_2000} or age-based~\cite{chen_optimizing_2006,rezg_joint_2004} maintenance policies. Such strategies are only defined with one decision variable per component: either the periodicity of maintenance or the age at which a component is replaced. In this paper, more general strategies are considered as we can decide whether or not to perform a PM at each time step for each component. Suppose that there are $T$ time steps and $n$ components, then our maintenance strategy is defined by $nT$ decision variables instead of $n$ variables for time-based or age-based strategies. The effort is justified as we consider a system on a long-term horizon where the costs incurred by forced outages are of the order of millions of euros. Then, even a minor improvement in the maintenance strategy generates important savings.
\end{enumerate}
% We give a brief overview of the literature on the optimization methods for multi-component maintenance scheduling problems. As noted in~\cite{dekker_applications_1996}, there is a great diversity in the modeling of maintenance problems because the objective and constraints vary from one study to another. Nevertheless, 

Reviews on optimal maintenance scheduling~\cite{alrabghi_state_2015,cho_survey_1991,nicolai_optimal_2008} give a summary of the optimization techniques used for maintenance problems in the literature. They can be split in two main categories: mathematical programming and heuristic methods~\cite{froger_maintenance_2016}. Heuristic methods can easily deal with non-linear objectives and contraints that arise when modeling complex industrial systems. They include genetic algorithms~\cite{almakhlafi_benchmarks_2012}, particle swarm optimization~\cite{suresh_hybrid_2013} and simulated annealing~\cite{fattahi_new_2014}. However, these methods are known to be subject of the curse of dimensionality and cannot be used for our large-scale optimization problem. For high-dimensional problems, a frontal resolution is impracticable, and resorting to decomposition methods is relevant~\cite{froger_maintenance_2016}. In~\cite{grigoriev_modeling_2006}, mixed integer programming is used to model a periodic maintenance optimization problem. A linear relaxation of the problem is then solved using a column generation technique. Column generation consists in iteratively solving a constrained master problem on a subset of variables and an unconstrained subproblem on the whole set of variables. The subproblem indicates which variables are to be considered in the master problem. In~\cite{lusby_solution_2013}, a shutdown planning for the refueling of nuclear power plants is designed with a Benders decomposition coupled with various heuristics. Benders decomposition relies on a fixed splitting of the variables and consists in iteratively solving a master problem on the first subset of variables with initially few constraints and a subproblem on the remaining variables while the first variables are fixed to the value given by the master. The subproblems are easy to solve and generate additional constraints (Benders cuts) for the master problem that is solved again. Column generation and Benders decomposition are efficient when the problem exhibits some particular structure, but which does not appear in our study.

% Previous works use a linear relaxation to apply decomposition techniques such as Benders~\cite{lusby_solution_2013} or Dantzig-Wolfe decomposition~\cite{grigoriev_modeling_2006}.
% of $10$ machines on a period of $18$ time steps. 
% The largest instance of the problem contains in the order of $10^{8}$ variables. 
% The state space of the largest instance consists of around $5,5 \times 10^{5}$ states. 
% However, when the problem size grows, the resolution becomes intractable with these methods. 

% For large-scale systems, heuristic methods are used most of the time as it is easier to deal with non-linearities than with mixed integer programming~\cite{alrabghi_simulation_2013}. They are also easy to couple with simulation-based models. Heuristic techniques includes genetic algorithms~\cite{almakhlafi_benchmarks_2012}, simulated annealing~\cite{fattahi_new_2014} and particle swarm optimization~\cite{suresh_hybrid_2013}. However, for large instances of the problem, generating acceptable solutions may be very slow due to the size of the solution space.

% For the considered industrial system, operational constraints impose to choose the maintenance strategy at the beginning of the time horizon without the possibility to change it afterwards. Hence, we seek a deterministic maintenance strategy for a stochastic system. 

This is why we investigate decomposition-coordination methods~\cite{carpentier_decomposition-coordination_2017}. Originated from the work of~\cite{arrow_decentralization_1960,lasdon_multi-level_1965,mesarovic_theory_1970,takahara_multilevel_1964}, these decomposition schemes are based on variational techniques and 
consist in the iterative resolution of auxiliary problems whose solutions converge to the solution of the original problem. The auxiliary problems are designed so that they are decomposable into independent subproblems of smaller size.
% consist in splitting the original large-scale optimization problem into several independent subproblems of smaller size that can be solved efficiently. The subproblems are coordinated to ensure that the concatenation of solutions leads to an optimal solution of the original problem.
Different types of decomposition-coordination schemes can be implemented, by prices, by quantities or by prediction. They have been unified within the Auxiliary Problem Principle~\cite{cohen_optimization_1978}. The types of decomposition differ in the way the auxiliary problems are designed and interpreted. In price decomposition, we aim at finding the saddle-point of the Lagrangian of the original problem. At each iteration, the auxiliary problem is the inner minimization of the Lagrangian, followed by an update of the multiplier using a projected gradient step. Price decomposition is the most commonly used decomposition-coordination method, and is applied in~\cite{nowak_stochastic_2000} for a power scheduling problem, in~\cite{diabat_lagrangian_2013} for supply chain management or in~\cite{kaihara_proactive_2010} for maintenance scheduling. The decomposition by prediction introduced by Mesarovic~\cite{mesarovic_theory_1970} is often used when explicit constraints on the system dynamics are considered, which corresponds to the framework of this paper. The resolution uses a fixed-point algorithm that is more efficient than the gradient-based algorithm used in price decomposition. This scheme has been used for the optimal control of robot manipulators~\cite{sadati_optimal_2006} or to solve an optimal power flow problem~\cite{nogales_decomposition_2003}. However, to our knowledge, it has never been applied to maintenance scheduling problems, which makes the originality of this work.

The industrial maintenance scheduling problem is modelled as a non-linear mixed integer program. We use a continuous relaxation of the system and decompose the large-scale optimization problem into several subproblems that consist in optimizing the maintenance on a single component. The decomposition algorithm iteratively solves the subproblems with the blackbox algorithm MADS~\cite{audet_mesh_2006} and coordinates the components in order to find an efficient maintenance strategy at the scale of the whole system. Then, we use the solution of the relaxed problem to design an admissible maintenance strategy for the original mixed integer problem. The aim of the paper is to show that the decomposition method can efficiently tackle maintenance problems with a large number of components, therefore it is applied on a system with $80$ components. The relaxation parameters have an important influence on the optimization and must be appropriately chosen. On the test case, the decomposition method outperforms the blackbox algorithm applied directly on the original problem.

The paper is organized as follows: in Section~\ref{sec:description}, we describe the industrial system and formulate the maintenance optimization problem. The Auxiliary Problem Principle and the decomposition by prediction are introduced in~Section~\ref{sec:app}. The application of the decomposition method to the maintenance optimization problem and the continuous relaxation of the system are presented in~Section~\ref{sec:app_syst}. Section~\ref{sec:num_res} contains numerical results showing the efficiency of the method in high dimension. Finally, in~Section~\ref{sec:ccl}, we conclude and give directions for future research.

%%% Local Variables:
%%% mode: latex
%%% TeX-master: "optim_maint_sched_v2"
%%% End:

%% file: description.tex
% !TEX root = optim_maint_sched_v2.tex

\section{System modeling and maintenance optimization problem}
\label{sec:description}

In this section, we describe the model of the studied industrial system and formulate the maintenance optimization problem. 
For any vector $v = (v_{1}, \ldots, v_{n})$, we denote the first $k$ components of $v$ by:
\begin{equation}
    v_{1:k} = (v_{1}, \ldots, v_{k}) \eqfinp
\end{equation}
The notation $\proscal{\cdot}{\cdot}$ represents the inner product in a Hilbert space and $\norm{\cdot}$ is the induced norm. Random variables are defined on a given probability space $(\omeg, \trib, \prbt)$ and are denoted with capital bold letters. For a set $\mathcal{A} \subset \bbR^{p}$, we denote by $\findi{\mathcal{A}}$ the indicator function of the set $\mathcal{A}$, \emph{i.e.} for $x \in \bbR^{p}$:
\begin{equation}
    \indic{\mathcal{A}}{x} = 
    \left\{
    \begin{aligned}
        1 &\text{ if } x \in \mathcal{A} \eqfinv\\
        0 &\text{ if } x \notin \mathcal{A} \eqfinp\\
    \end{aligned}
    \right.
\end{equation}

% \begin{definition}
%     Let $p \in \bbN$, $\mathcal{A} \subset \bbR^{p}$ and $x \in \bbR^{p}$. The indicator function of the set $\mathcal{A}$ is
%     \begin{equation}
%         \indic{\mathcal{A}}{x} = 
%         \left\{
%         \begin{aligned}
%             1 &\text{ if } x \in \mathcal{A} \eqfinv\\
%             0 &\text{ if } x \notin \mathcal{A} \eqfinp\\
%         \end{aligned}
%         \right.
%     \end{equation}
% \end{definition}
% Random variables are used in the description of the system. Hence, we introduce the following notation for the space of random variables.
% \begin{definition}
%     Let $(\omeg, \trib, \prbt)$ be a probability space and $(\espacea{Y}, \mathscr{Y})$ be a measurable space. The set of measurable functions from $(\omeg, \trib, \prbt)$ to $(\espacea{Y}, \tribu{Y})$ is denoted by
%     \(
%       \espacef{Y} = \espaceva{\omeg, \trib, \prbt}{\espacea{Y}}
%     \), where we omit the $\sigma$-algebra $\tribu{Y}$ in the notation.
% \end{definition}
% Any random variable $\va{Y}: \omeg \to \espacea{Y}$ is an element of $\espacef{Y}$.
\subsection{Description of the system}
\label{subsec:gen_syst}

This work is motivated by an industrial maintenance scheduling problem for components from hydroelectric power plants. In this paper, a \emph{system} gathers several components of the same family: we consider either a set of turbines, a set of transformers or a set of generators. The system also comprises a common stock of spare parts. All components belong to the same production unit and they must all be in their nominal state for the unit to produce electricity (series system). The components can either be healthy or broken, no degradation state is considered. Over time, failures occur according to known failure distributions. To replace a failed component, we perform a CM using a part from the common stock. As soon as a failure occurs, we order a new spare part. We consider complex cutting-edge components that are critical for the electricity production and for which the manufacturing and the qualification processes require several months or even years. Therefore the time of replenishment of the stock is important. If several components fail at close moments in time, it may happen that no spare part is available. In this case, some components cannot be replaced and no electricity is produced: the system is in \emph{forced outage}. We can also perform PMs to replace healthy components in order to prevent a future failure. For a PM, we order a spare part so that it arrives just on time for the maintenance operation, therefore a PM does not use a part from the stock. A PM must be planned several years in advance as it requires the planning of the budget, the anticipation of the manufacturing of the part and involves highly specialized teams that operate all around the country and must be booked several years in advance. Hence, the dates of PMs are chosen at the initial time of the horizon based only on the statistical knowledge of the future dates of failures. We assume that PMs and CMs are performed in negligible time and that they are perfect operations meaning that a component is \emph{as good as new} after a maintenance. As PMs are planned operations, they are cheaper than unpredictable CMs which require the unplanned requisition of highly specialized teams and induce a modification of their working schedules. PM and CM costs take into account all the costs generated by the maintenance operation (labor cost, cost of the new part, logistic cost, set up cost). Forced outage costs are much larger than that of a PM or a CM. No holding cost is considered for the stock of spare parts.

We denote by $n \in \bbN^{\star}$ the number of physical components in the system. The horizon of study is denoted by $T \in \bbN^{\star}$. In the sequel, $i \in \components = \{1, \ldots, n\}$ denotes a component index, $t \in \timesteps = \{0, \ldots, T\}$ denotes a time step and we use the notation $\timestepsnoend$ for the set $\na{0, \ldots, T-1}$. A sketch of the system with $n=2$ components is represented in Fig.~\ref{fig:toy_system}. The variables $\sto{t}$, $\statecomptime{1}{t}$ and $\statecomptime{2}{t}$ describe respectively the state of the stock and of the components. They are defined in~\S\ref{subsubsec:sto_comp_char}. The variable $\ctrlcomptime{i}{t}$ represents the PM decision and is defined precisely in~\S\ref{subsubsec:model_pm}. Finally $\costpm_{i,t}$, $\costcm_{i,t}$ and $\costfo_{i,t}$ represent the cost generated by a PM, a CM and a forced outage respectively, they are defined in~\S\ref{subsec:cost}. An example of the dynamics of a system with two components and a stock with initially one spare part is given in Fig.~\ref{fig:ex_dyn}.
% A CM consists in the replacement of a component after a failure. A PM is a planned replacement of a component before a failure. 
% In practice, the number $n$ of components in the industrial system can range between $2$ and $80$ and the time horizon is $40$ years. As we consider components that should be reliable and fail on average $3$ or $4$ times on a horizon of $40$ years, the initial number of spare parts is relatively low compared to the number of components, of the order of $1/5$ of the number of components.

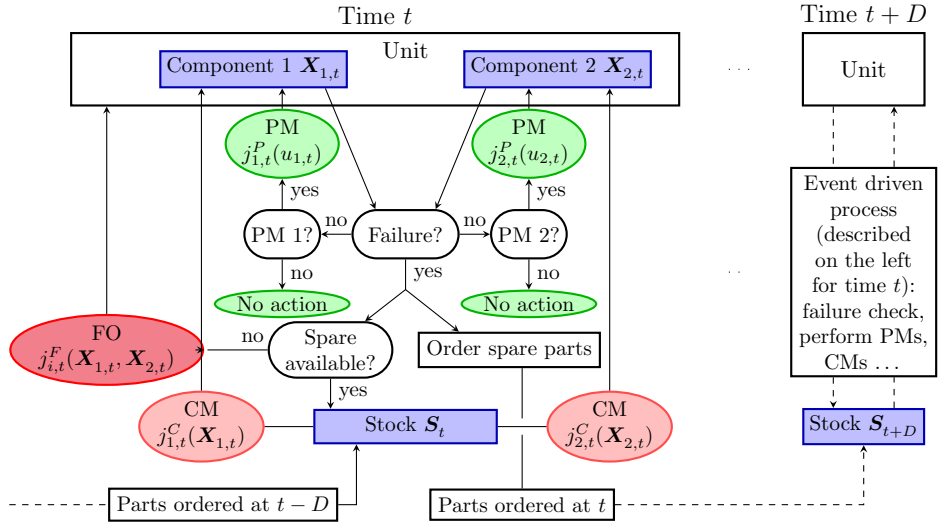
\begin{figure}[htbp]
    \centering
    \scalebox{0.8}{
    \begin{tikzpicture}
        \begin{scope}[every node/.style={font=\small, align=center, fill opacity=0.5, text opacity=1, draw opacity=1}]
            \node (unit) [unit, minimum width=10cm, label=north:{\large Time $t$}] {};
            \node (comp1) [comp, xshift=-2cm] at (unit) {Component 1 $\statecomptime{1}{t}$};
            \node (comp2) [comp, xshift=3cm] at (unit) {Component 2 $\statecomptime{2}{t}$};
            \node (unit_label) [rectangle, font=\normalsize] at ($(comp1.north)!0.5!(comp2.north)$) {Unit};
            \node (fail) [event, anchor=north, yshift=-2cm] at ($(comp1.south)!0.5!(comp2.south)$) {Failure?};
            \node (avail) [event, anchor=east, yshift=-1.5cm, xshift=-0.2cm] at (fail.south) {Spare\\ available?};
            \node (part_ord) [part_ord, anchor=west, yshift=-1.5cm, xshift=0.2cm] at (fail.south) {Order spare parts};
            \node (pm1_check) [event, anchor=east, xshift=-0.5cm, minimum width=1.5cm] at (fail.west) {PM 1?};
            \node (pm2_check) [event, anchor=west, xshift=0.5cm, minimum width=1.5cm] at (fail.east) {PM 2?};
            \node (stock) [comp, anchor=north] at ([yshift=-0.5cm]avail.south -| fail.south) {Stock $\sto{t}$};
            \node (no_act_1) [pm, anchor=north, yshift=-0.5cm, inner sep=1pt] at (pm1_check.south) {No action};
            \node (no_act_2) [pm, anchor=north, yshift=-0.5cm, inner sep=1pt] at (pm2_check.south) {No action};

            \node (new_unit) [unit, xshift=2cm, anchor=west, font=\normalsize, minimum width=2cm, label=north:{\large Time $t+D$}] at (unit.east) {Unit};
            \node (new_stock) [comp, minimum width=0pt] at (new_unit |- stock) {Stock $\sto{t+D}$};
            \node (box_unit) [unit, minimum width=2cm, minimum height=3cm, yshift=-1cm, anchor=north] at (new_unit.south){Event driven \\ process \\ (described \\ on the left \\ for time $t$):\\ failure check, \\ perform PMs, \\ CMs \dots};

            \node (pm1) [pm, anchor=south, yshift=0.5cm] at (pm1_check.north) {PM \\$\costpm_{1,t}(\ctrlcomptime{1}{t})$};
            \node (pm2) [pm, anchor=south, yshift=0.5cm] at (pm2_check.north) {PM \\$\costpm_{2,t}(\ctrlcomptime{2}{t})$};
            \node (cm1) [corrm, anchor=east, xshift=-0.8cm] at (stock.west) {CM \\$\costcm_{1,t}(\statecomptime{1}{t})$};
            \node (cm2) [corrm, anchor=west, xshift=0.8cm] at (stock.east) {CM \\$\costcm_{2,t}(\statecomptime{2}{t})$};
            \node (fo) [fo, anchor=west, xshift=-1cm] at (avail.west -| unit.west) {FO \\ $\costfo_{i,t}(\statecomptime{1}{t}, \statecomptime{2}{t})$};

            \node (prev_part_order) [part_ord, yshift=-1cm, xshift=-3cm] at (stock.south) {Parts ordered at $t-D$};
            \node (tmp1) at ($(stock.east)!0.5!(cm2.west)$){};
            \node (new_part_order) [part_ord, yshift=-1cm] at (tmp1 |- stock.south) {Parts ordered at $t$};
        \end{scope}
        % \node (legend) [rectangle, xshift=0.5cm, anchor=west, align=left] at (comp2.east) {\textcolor{red!70}{CM: Corrective Maintenance} 
        % %(following a failure)
        % \\ \textcolor{green!70!black}{PM: Preventive Maintenance} %(planned)
        % };
        \begin{scope}[line width=0.5pt, >=stealth, every node/.style={font=\small}]
            %   \draw[->] (comp1.300) to [out=270, in=180] (stock.west);
            %   \draw[->] (comp2.240) to [out=270, in=0] (stock.east);
            \draw[->] (comp1.345) -- (fail.140);
            \draw[->] (comp2.195) -- (fail.40);
            \draw[->] (fail.south) -- node[anchor=west]{yes} ([yshift=-0.5cm]fail.south) -- (avail.40);
            \draw[->] ([yshift=-0.5cm]fail.south) -- (part_ord.160);
            \draw[->] (avail.south) -- node[anchor=west]{yes} (stock.north -| avail.south); 
            \draw[->] (fail.west) -- node[anchor=south]{no} (pm1_check.east);
            \draw[->] (fail.east) -- node[anchor=south]{no} (pm2_check.west);
            \draw[->] (pm1_check.north) -- node[anchor=west]{yes} (pm1.south);
            \draw[->] (pm2_check.north) -- node[anchor=west]{yes} (pm2.south);
            \draw[->] (pm1_check.south) -- node[anchor=west]{no} (no_act_1.north);
            \draw[->] (pm2_check.south) -- node[anchor=west]{no} (no_act_2.north);

            \draw[->] let \p1 = (comp1.south), \p2 = (pm1.north) in (pm1.north) -- (\x2, \y1);
            \draw[->] let \p3 = (comp2.south), \p4 = (pm2.north) in (pm2.north) -- (\x4, \y3);
            \draw[->] let \p5 = (comp1.south), \p6 = (cm1.north) in (cm1.north) -- (\x6, \y5);
            \draw[->] let \p7 = (comp2.south), \p8 = (cm2.north) in (cm2.north) -- (\x8, \y7);
            \draw[->] let \p9 = (unit.south), \p0 = (fo.north) in (fo.north) -- (\x0, \y9);
            \draw[->] ($(avail.west -| cm1.north) + (-0.1,0)$) -- (fo.east);
            \draw[->] (prev_part_order.east) -| (stock.200);
            \draw[->, dashed] (new_part_order.east) -| (new_stock.south);

            \draw[->, dashed]($(box_unit.south) + (-0.5,0)$) -- ($(new_stock.north) + (-0.5,0)$);
            \draw[->, dashed] ($(box_unit.north) + (0.5,0)$) -- ($(new_unit.south) + (0.5,0)$);
        \end{scope}
        \begin{scope}[line width=0.5pt, every node/.style={font=\small}]
            \draw[-] (stock.west) -- (cm1.east);
            \draw[-] (stock.east) -- (cm2.west);
            \draw[-] (avail.west) -- node[anchor=south, near start]{no} ($(avail.west -| cm1.north) + (0.1,0)$);
            \draw[dashed] ($(unit.west |- prev_part_order.west) + (-1, 0)$) -- (prev_part_order) ;
            \draw[-] (tmp1.south) -- (new_part_order.north -| tmp1);
            \draw[-] (part_ord.south -| tmp1) -- (tmp1.north);

            \draw[dashed] ($(new_unit.south) + (-0.5,0)$) -- ($(box_unit.north) + (-0.5,0)$);
            \draw[dashed] ($(new_stock.north) + (0.5,0)$) -- ($(box_unit.south) + (0.5,0)$);

            \draw[loosely dotted] ($(unit.east |- box_unit.west) + (0.8,0)$) -- ([xshift=-0.8cm]box_unit.west);
            \draw[loosely dotted] ([xshift=0.8cm]unit.east) -- ([xshift=-0.8cm]new_unit.west);
        \end{scope}
    %   \begin{scope}[on background layer]
    %       \node (sys) [fit=(comp1) (compedge1) (compedge2) (comp2) (stock) (pm1) (pm2) (cm1) (cm2)] {};
    %   \end{scope}
    %   \draw [decorate, decoration={brace, amplitude=25pt, raise=0.5cm}, line width=1pt] (sys.north west) -- node (system) [line width=1pt, anchor=south, yshift=1.5cm, font=\normalsize, inner sep=2mm] {Unit} (sys.north east);
    \end{tikzpicture}
    }
    \caption{System of two components from a single unit sharing the same stock of spare parts. The right part of the figure is a copy of the system at time $t+D$. It emphasizes the coupling between time $t$ and $t+D$ through the supply in spare parts.}
    \label{fig:toy_system}
\end{figure}

\begin{figure}[htbp]
    \centering
    \includegraphics[width=0.6\textwidth]{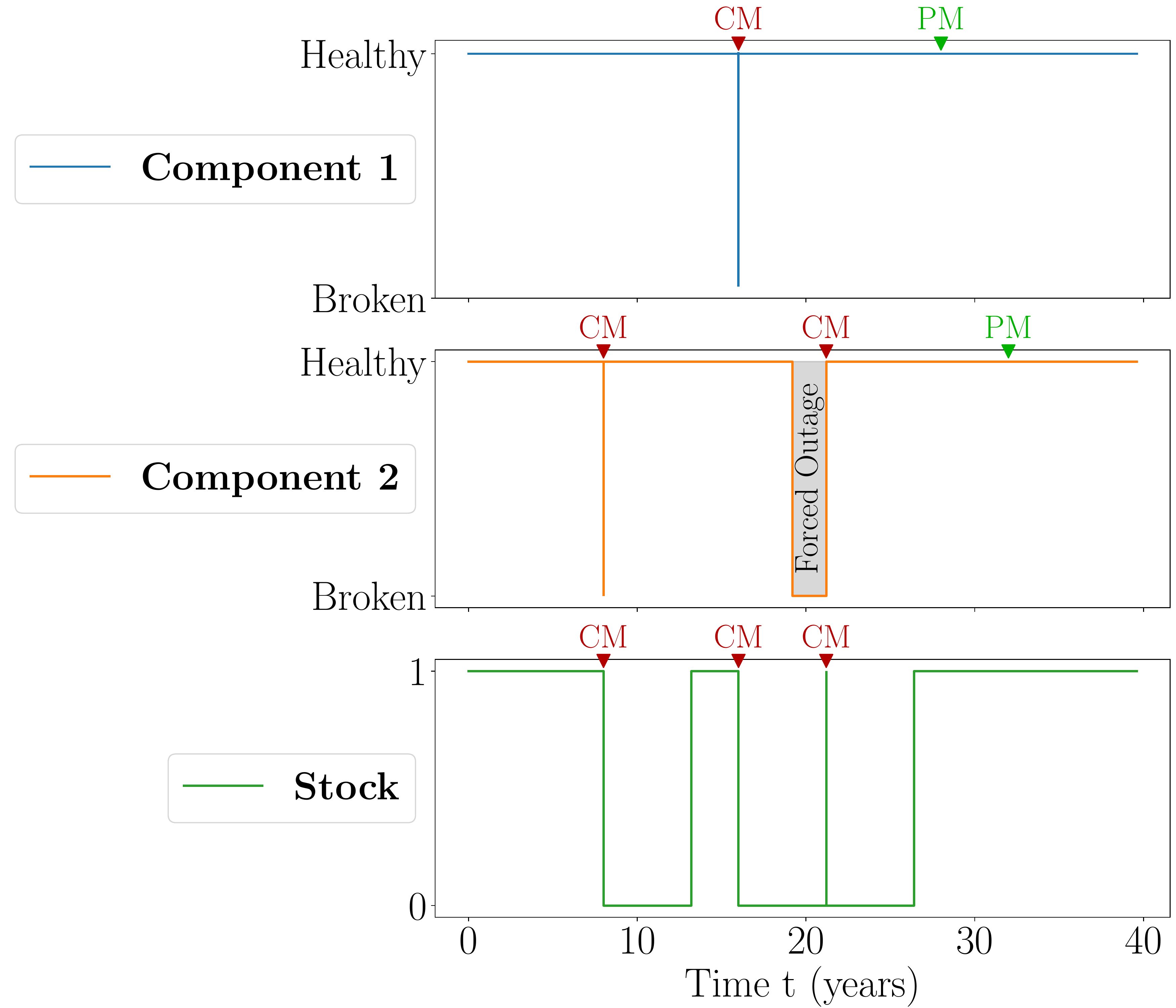}
    \caption{Illustration of the dynamics of a system with two components. When a failure occurs, if a spare part is available, a CM is performed immediately. The stock is then empty until the arrival of a new part (ordered at the time of the previous failure). If a failure occurs and the stock is empty, there is a forced outage until the arrival of a spare part. Some PMs can be performed to avoid failures, they do not use parts from the stock as they can be scheduled in advance.}
    \label{fig:ex_dyn}
\end{figure}

The model intends to represent the following real situations:
\begin{enumerate}
\item In a single production unit, we consider a system with a unique family of components from the plant (either turbines, transformers or generators). In this case, any failure of a component
  induces a failure of the system (series system). In a real situation, the
  number of components in this case can be up to $10$.
\item We can also consider a system with components across several production
  units, still with one common stock. This time, a failure of a given component
  will only induce a failure of its own production unit and not of the whole
  system. Our methodology remains applicable to this system with the only
  difference that there is one forced outage cost for each unit. In this case
  the number of components can be up to $80$.
\item Finally, the developed methodology is also easily applicable to a system
  with several families of components and one separate stock of spares for each
  family. The components can be either in a single production unit or
  distributed across several units.
\end{enumerate}
The primary goal of our work is to show that we are able to tackle maintenance scheduling problems with a large number of components (up to $80$). This is why, we design a model that is a proxy for the aforementioned cases. In Section~\ref{sec:num_res}, we explain how our model can be used to represent each of the three real situations described above.

\subsubsection{Characterization of the stock and the components}
\label{subsubsec:sto_comp_char}

The stock over time is characterized by the sequence of random variables
\begin{equation}
  \stoalltime = (\sto{0}, \ldots, \sto{T}) \in \espacef{S}\eqfinv
  \label{eq:adm_s}
\end{equation}
where $\sto{t}: \omeg \rightarrow \spares$ is the random variable representing the number of available spare parts at time $t$ and $\mathcal{S}$ is the set of all random variables taking values in $\spares^{T+1}$. The parameter $s \in \bbN$ is the maximum number of spare parts. The value of the initial stock is set to $\sto{0} = s$. The replenishment delay for the parts, that is, the time from order to delivery of a part, is known and denoted by $\delay \in \bbN$.

% In practice the time of supply of spare parts can be up to $2$ years. 
At time $t$, component $i$ is characterized by random variables representing:
\begin{itemize}
    \item its regime
        \begin{equation}
            \reg{i}{t} = \left\{
                \begin{aligned}
                    0 & \text{ if the component is broken},\\
                    1 & \text{ if the component is healthy}.
                \end{aligned}
                \right.
        \end{equation}
    A component has only two regimes: in the healthy regime, it runs in its nominal operating point. In the broken regime, it stops working completely. Initially all components are healthy \emph{i.e.} $\va{E}_{i,0} = 1$ for all $i \in \components$.

    \item its age (if healthy) or the time for which it has failed (if broken) denoted by 
    the real-valued random variable 
    $\age{i}{t}$.
    Initially the components are new \emph{i.e.} $\va{A}_{i,0} = 0$ for all $i \in \components$.
    \item the time elapsed since its last $\delay$ failures
        \begin{equation}
            \spr{i}{t} = (\spr{i}{t}^{1}, \ldots, \spr{i}{t}^{\delay}) \eqfinv
            % \in [\{-1\} \cup \bbRp]^{\delay}
        \end{equation}
    where $\delay$ is the number of time steps for the supply of spare parts. For $d \in \{1, \ldots, \delay\}, \ \spr{i}{t}^{d}$ is the number of time steps elapsed since the $d$-th undiscarded failure of component $i$. $\spr{i}{t}^{d}$ takes a default value $\sprnofail$ if the component has failed fewer than $d$ times. Hence, $\spr{i}{t}^{d}$ takes values in $\na{\sprnofail} \cup \bbRp$ and $\spr{i}{0} = (\sprnofail, \ldots, \sprnofail)$. The random vector $\spr{i}{t}$ is useful to compute the dates of replenishment of the stock. 
    % A spare part is ordered each time a component fails, so when the time since a previous failure is equal to the time of supply $D\Delta t$ of the part, the stock is replenished. 
    It is enough to store at most the dates of the last $D$ failures to describe the supply of the stock. More details are given in~\S\ref{subsubsec:dyn_comp}. 
    % If more than $D$ failures have occurred then the dates of the oldest ones can simply be discarded as the corresponding part has already arrived in the stock.
    
    % If $\spr{i}{t} = (t_{1}, \ldots, t_{D})$, then $t_{1} > \ldots > t_{D} \geq 0$ implying that $t_{1} \geq D-1$. At the next time step, the part ordered from the failure at $t_{1}$ has arrived and if a new failure occurs we can discard $t_{1}$ to make room for the new date of failure.
\end{itemize}
The characteristics of component $i$ at time $t$ are gathered in:
\begin{equation}
    \statecomptime{i}{t} = (\reg{i}{t}, \age{i}{t}, \spr{i}{t}) \in \mathcal{X}_{i,t} \eqfinv
    % \text{ taking values in }
    % \espacea{X}_{i,t} = \{0,1\} \times \bbR \times \{\na{\sprnofail} \cup \bbRp\}^{\delay}
\end{equation}
where $\mathcal{X}_{i,t}$ is the set of all random variables defined on $\omeg$ taking values in $\{0,1\} \times \bbRp \times \left(\na{\sprnofail} \cup \bbRp\right)^{\delay}$. The state of the system is then described at $t$ by $(\statecomptime{1}{t}, \ldots, \statecomptime{n}{t}, \sto{t})$. Finally, to describe the components over the whole study period we introduce:
\begin{equation}
    \stateallcomp = (\statecomponent{1}, \ldots, \statecomponent{n}) = ((\statecomptime{1}{0}, \ldots, \statecomptime{1}{T}), \ldots, (\statecomptime{n}{0}, \ldots, \statecomptime{n}{T})) \in \espacef{X} \eqfinv
\end{equation}
where $\espacef{X} = \prod_{i=1}^{n} \espacef{X}_{i}$ and $\espacef{X}_{i} = \prod_{t=0}^{T} \espacef{X}_{i,t}$, for all $i\in \components$. In order to emphasize that $\stateallcomp$ depends on all the components of the system, we sometimes use the notation $\statecomponent{1:n}$ instead of $\stateallcomp$.

% $\statecomponent{i}$ characterizes component $i$ on the whole study period.
% In order to characterize all the components of the system at a given time $t$ we usu
% \begin{equation}
%     \statetime{t} = (\statecomptime{1}{t}, \ldots, \statecomptime{n}{t})
% \end{equation}
% and for the characterization of component $i$ through the whole study period, we use
% \begin{equation}
%     \statetime{i} = (\statecomptime{i}{0}, \ldots, \statecomptime{i,T})
% \end{equation}

\subsubsection{Preventive maintenance strategy}
\label{subsubsec:model_pm}

A PM consists in repairing a component although it is in the healthy regime. The dates of PM can be different for each component. They define the preventive maintenance strategy of the system. Operational constraints impose to look for deterministic strategies. This means that the dates of PM are chosen without any knowledge on the state of the system after the beginning of the time horizon and cannot be changed during the study. 
The maintenance strategy is defined by a vector
\begin{equation}
    \ctrlallcomp = (\ctrlcomp{1}, \ldots, \ctrlcomp{n}) = 
        ((\ctrlcomptime{1}{0}, \ldots, \ctrlcomptime{1}{T-1}), \ldots, (\ctrlcomptime{n}{0}, \ldots, \ctrlcomptime{n}{T-1}))\in \espacea{U} = [0,1]^{nT} \eqfinv
    \label{eq:adm_u}
\end{equation}
where $\ctrlcomptime{i}{t}$ characterizes the PM for component $i$ at time $t$. More precisely, we set a threshold $0 < \nu < 1$: a control $\ctrlcomptime{i}{t} \geq \nu$ corresponds to a rejuvenation of the component proportional to $\ctrlcomptime{i}{t}$ and a value $\ctrlcomptime{i}{t} < \nu$ corresponds to not performing a maintenance. We consider that the duration of the maintenance operation is negligible. Moreover, a PM does not use parts from the stock as it is planned in advance and the parts are ordered so that they arrive just on time for the maintenance. Note that the modeling of the PM strategy uses a continuous decision variable $u$, this choice is justified in~\S\ref{subsec:mads}. 

\subsubsection{Failures of the components}
\label{subsubsec:model_fail}

In our study, the distribution of the time to failure for component $i$ is a known Weibull distribution
% of parameters $(\beta_{i}, \lambda_{i})$ denoted by $\mathrm{Weib}(\beta_{i}, \lambda_{i})$ and 
with cumulative distribution function denoted by $F_{i}$. If component $i$ is healthy at time $t$ and has age $a \geq 0$ , its probability of failure at time $t + \Delta t$ is given by:
\begin{equation}
        p_{i}(a) = \frac{F_{i}(a + \Delta t) - F_{i}(a)}{1 - F_{i}(a)} \eqfinp
\end{equation}
% \begin{proposition}
%     Let $i \in \components$. Assume that component $i$ has age $a \geq 0$ at time~$t$.
%     % , that is $\reg{i}{t} = 1$ and $\age{i}{t} = a$. 
%     Let $F_{i}$ be the cumulative distribution function of the failure distribution of component $i$. Then, the probability of failure of component $i$ at time~$t+\Delta t$ conditionally to the component being healthy at $t$ is given by
%     \begin{equation}
%         \begin{aligned}
%             p_{i}(a) 
%             % &= \mathbb{P}(\reg{i}{t+1} = 0 | \reg{i}{t}=1, \age{i}{t}=a)\\
%             &= \frac{F_{i}(a + \Delta t) - F_{i}(a)}{1 - F_{i}(a)} \eqfinv
%         \end{aligned}
%     \end{equation}
% \end{proposition}
% The proof of this result can be found in~\cite[Section 2.2.1]{meeker_statistical_2014}.
% \begin{proof}
%     Let $\va{A}$ be the random variable that represents the age of the component just before the failure. As we assume that the component has age $a$ at time $t$, the failure of the component at time step $t+1$ corresponds to the event $\va{A} \in [a, a + \Delta t[$ conditionally to $\va{A} \geq a$. By definition of the failure distribution
%     \begin{equation}
%         \mathbb{P}(\va{A} \in [a, a + \Delta t[ \ | \va{A} \geq a ) = \frac{F_{i}(a + \Delta t) - F_{i}(a)}{1 - F_{i}(a)}
%     \end{equation}
%     \skipfinpreuve
% \end{proof}
We introduce the random sequence:
\begin{equation}
    \noiseallcomp = (\noisecomp{1}, \ldots, \noisecomp{n}) = ((\noise{1}{1}, \ldots, \noise{1}{T}), \ldots, (\noise{n}{1}, \ldots, \noise{n}{T})) \in \espacef{W} \eqfinv
\end{equation}
where $\espacef{W}$ is the set of all random variables defined on $\omeg$ taking values in $[0, 1]^{nT}$. The random process $\noiseallcomp$ is an exogenous noise that affects the dynamics of the regime $\regall$ and the age $\ageall$.
% We denote by $(\trib_{t})_{t \in \timestepsnobegin$ the filtration associated with the process $\noiseallcomp$. This is the $\sigma$-algebra generated by the past noise:
% \begin{equation}
%     \trib_{t} = \sigma(\noise{1:n}{0:t})
% \end{equation}
We assume that all $\noise{i}{t}$ are independent random variables and follow a uniform distribution on $[0,1]$. At time step $t$, component~$i$ has age $\age{i}{t}$. If $\noise{i}{t+1} < p_{i}(\age{i}{t})$, component $i$ fails at $t+1$, otherwise no failure occurs.

% Hence $\noise{i}{t+1}$ follows a Bernoulli distribution
% \begin{equation}
%     \noise{i}{t+1} \sim \mathcal{B}(p_{i}(\age{i}{t}))
% \end{equation}
% and has the following effect:
% \begin{itemize}
%     \item If the component is healthy \emph{i.e.} $\reg{i}{t} = 1$ and has age $\age{i}{t}$:
%     \begin{itemize}
%         \item If no PM is performed \emph{i.e.} $\ctrlcomptime{i}{t} = 0$:
%         \begin{itemize}
%             \item If $\noise{i}{t+1} = 1$, the component fails hence $\reg{i}{t+1} = 0$ and $\va{A}_{i, t+1} = 0$.
%             \item If $\noise{i}{t+1} = 0$, the component stays healthy \emph{i.e.} $\reg{i}{t+1} = 1$ and his age is $\va{A}_{i, t+1} = \age{i}{t} + \Delta t$
%         \end{itemize}
%         \item If a PM is performed, $\noise{i}{t+1}$ has no effect, this case has been described in~\S\ref{subsubsec:model_pm}.
%     \end{itemize}
%     \item If the component is broken \emph{i.e.} $\reg{i}{t} = 0$, then $\noise{i}{t+1}$ has no effect, this case is covered next.
% \end{itemize}

\subsection{Dynamics of the system}
\label{subsec:syst_dyn}

Now, we describe the dynamics of the system, that is, we explain how the variables characterizing the system evolve between two time steps.

\subsubsection{Dynamics of a component}
\label{subsubsec:dyn_comp}

% Definition of the maintenance strategy
Let $i \in \components$ and $t \in \timestepsnoend$. The dynamics of component $i$ between $t$ and $t+1$ is described as follows.
% Dynamics of a healthy component, possibility of failure
\begin{enumerate}
    \item If component $i$ is healthy \emph{i.e.} $\reg{i}{t} = 1$:
    \begin{enumerate}
        \item If $\ctrlcomptime{i}{t} \geq \nu$, then a PM is performed. After a PM, component $i$ stays healthy and is rejuvenated so that:
            \begin{equation}
                (\reg{i}{t+1}, \age{i}{t+1}) = (1, (1 - \ctrlcomptime{i}{t}) \age{i}{t} + 1) \eqfinp
            \end{equation}
        The PM is performed instantaneously at time $t$, so after the maintenance (still at time $t$) its age is $(1-\ctrlcomptime{i}{t})\age{i}{t}$. Therefore, the age is $(1 - \ctrlcomptime{i}{t}) \age{i}{t} + 1$ at time $t+1$. Note that $\ctrlcomptime{i}{t} = 1$ makes the component as good as new: in this case we have $\age{i}{t+1} = 1$.
        \item If $\ctrlcomptime{i}{t} < \nu$, then no PM is performed. Component $i$ fails with probability $p_{i}(\age{i}{t})$:
        \begin{align}
            (\reg{i}{t+1}, \age{i}{t+1}) = 
            \left\{
                \begin{aligned}
                    &(0, 0) && \text{ if } \noise{i}{t+1} < p_{i}(\age{i}{t}) \eqfinv\\
                    &(1, \age{i}{t} + 1) && \text{ otherwise} \eqfinp
                \end{aligned}
            \right.
        \end{align}
    \end{enumerate}
    \item If component $i$ is broken \emph{i.e.} $\reg{i}{t} = 0$:
    \begin{enumerate}
        \item If a spare is available in the stock, a CM is performed to replace the component. We assume that the CM is an identical replacement, which implies that the component becomes as good as new. We get:
        \begin{equation}
                (\reg{i}{t+1}, \age{i}{t+1}) = (1, 1) \eqfinp
        \end{equation}
        A CM is performed instantaneously at time $t$, so the age of the component is $0$ after the CM, therefore it has age $1$ at time $t+1$.
        \item If no spare part is available, the defective component stays in the broken regime:
        \begin{equation}
                (\reg{i}{t+1}, \age{i}{t+1}) = (0, \age{i}{t} + 1) \eqfinp
        \end{equation}
        As all components belong to the same power plant, when at least one component is broken, the unit is shut down until the arrival of a spare part and the execution of the CM. Such a situation is a \emph{forced outage}. During the shut down no electricity is produced. 
    \end{enumerate}
\end{enumerate}

We have to express formally that a spare part is available for the replacement of component $i$. At time $t$, suppose that the stock has $\sto{t} = r$ parts and that $m$ components are broken. If $r \geq m$, then all components can be replaced immediately. When $r < m$, we must choose which components to replace. Our modeling choice is to replace the broken components following the order of their index: if $i_{1} \leq \ldots \leq i_{r} \leq \ldots \leq i_{m}$ are the indices of the broken components, we replace only the components with index $i_{1}, \ldots, i_{r}$, the others stay in the broken regime and wait for new available parts. Using this choice, the availability of a spare part for component $i$ corresponds to the condition:
\begin{equation}
    \sto{t} \geq \sum_{j=1}^{i} \indic{\na{0}}{\reg{j}{t}} \eqfinp
    \label{eq:avail_comp}
\end{equation}
The right hand side of~\textup {(\ref{eq:avail_comp})} simply counts the number of broken components with index smaller or equal than $i$.

% \begin{remark}
%     With our modeling choice, we could feel that a component with a high index is more likely to wait in a failed state than a component with a smaller index that would be replaced first. We could have chosen to replace the components in the order of which they have failed to avoid this bias. However with this latter modeling the condition of availability of a part for component $i$ is more complicated than~\textup {(\ref{eq:avail_comp})}, namely
%     \begin{equation}
%         \sto{t} \geq \sum_{j=1}^{n} (1 - \reg{j}{t}) \indic{\bbRp}{\age{i}{t}{j}{t} - \age{i}{t}}
%     \end{equation}
%     . Moreover if at least one component stays in the broken regime, a cost of forced outage is incurred. This cost is identical no matter which component is broken. Hence the choice of which component to repair first has no influence on the cost. For the simplicity of~\textup {(\ref{eq:avail_comp})}, we stick to the replacement strategy by smaller index first. 
% \end{remark}

To completely describe the dynamics of a component, we have to specify the dynamics of the vector $\spr{i}{t}$. It has been introduced in~\S\ref{subsubsec:sto_comp_char} to store the dates of failures of the component and compute the dates for the replenishment of the stock.
\begin{itemize}
    \item If $\spr{i}{t} = (t_{1}, \ldots, t_{d}, \sprnofail, \ldots, \sprnofail)$ with $t_{1}, \ldots, t_{d} \geq 0$, meaning that component $i$ has undergone $d < \delay$ failures so far, then:
    \begin{subequations}
        \begin{empheq}[left={\spr{i}{t+1} =\empheqlbrace}]{align}
            &(t_{1}+1, \ldots, t_{d}+1, 0, \sprnofail, \ldots, \sprnofail) && \text{if failure at $t+1$}\eqfinv
            \label{eq:nb_fail_fail}\\
            &(t_{1}+1, \ldots, t_{d}+1, \sprnofail, \sprnofail, \ldots, \sprnofail) && \text{otherwise}\eqfinp
            \label{eq:nb_fail_nofail}
        \end{empheq}
        \label{eq:nb_fail}
    \end{subequations}
    \item If $\spr{i}{t} = (t_{1}, \ldots, t_{\delay})$ with $t_{1}, \ldots, t_{\delay} \geq 0$, meaning that component $i$ has undergone at least $\delay$ failures so far, then:
    \begin{subequations}
        \begin{empheq}[left={\spr{i}{t+1} =\empheqlbrace}]{align}
            &(t_{2}+1, \ldots, t_{\delay}+1, 0) && \text{if failure at $t+1$}\eqfinv
            \label{eq:nb_fail_leqk_fail}\\
            &(t_{1}+1, \ldots, t_{\delay}+1) && \text{otherwise}\eqfinp
            \label{eq:nb_fail_leqk_nofail}
        \end{empheq}
        \label{eq:nb_fail_leqk}
    \end{subequations}
    In~\textup {(\ref{eq:nb_fail_leqk_fail})}, note that $t_{1}$ is discarded. As $\spr{i}{t} = (t_{1}, \ldots, t_{\delay})$ and $t_{1} > \ldots > t_{\delay} \geq 0$, we get that $t_{1} \geq \delay-1$. At time step $t+1$, the part ordered from the failure at $t_{1}$ has arrived. Then, storing $t_{1}$ is not useful anymore. So if a failure occurs at $t+1$, we can discard $t_{1}$ to make room for the new date of failure. This proves that it is enough to have $\spr{i}{t}$ of size $\delay$ to compute the replenishment of the stock as stated in~\S\ref{subsubsec:sto_comp_char}. Note that the dates are not discarded if there is no failure (see~\eqref{eq:nb_fail_leqk_nofail}), so it is possible to have $t_{d} > \delay$ for some $d \in \{1, \ldots, \delay\}$. Such variables have no influence on the dynamics of the system.
\end{itemize}
\begin{figure}[htbp]
    \centering \scalebox{0.8}{
    \begin{tikzpicture}
        \node[healthystate](xt1){$(1, \age{i}{t})$ \\ Healthy, age $\age{i}{t}$};
        \node[brokenstate, yshift=-3.5cm, anchor=north](xt0) at (xt1.south) {$(0, \age{i}{t})$ \\ Broken for a time $\age{i}{t}$};
    
        \node[control, xshift=2cm, yshift=1.1cm, anchor=west](pm) at (xt1.east) {$\ctrlcomptime{i}{t} \geq \nu$ \\ PM};
        \node[control, xshift=2cm, yshift=-1.1cm, anchor=west](nopm) at (xt1.east) {$\ctrlcomptime{i}{t} < \nu$ \\No PM};
    
        \node[healthystate, xshift=3cm, anchor=west](repair) at (pm.east) {$(1, (1 - \ctrlcomptime{i}{t})\age{i}{t} + 1)$ \\ Rejuvenation};
        \node[healthystate, xshift=3cm, yshift=0.7cm, anchor=west](ageing) at (nopm.east) {$(1, \age{i}{t} + 1)$ \\ Ageing};
        \node[brokenstate, xshift=3cm, yshift=-0.7cm, anchor=west](fail) at (nopm.east) {$(0, 0)$ \\ Failure};
    
        \node[healthystate, yshift=0.75cm](replace) at (xt0 -| ageing) {$(1, 1)$ \\ Replacement};
        \node[brokenstate, yshift=-0.75cm](broken) at (xt0 -| ageing) {$(0, \age{i}{t} + 1)$ \\ Stays broken};
    
        \node[spare] (spare) at (xt0 -| nopm) {$\sto{t} \geq \sum_{j=1}^{i} \indic{\na{0}}{\reg{j}{t}}$ ? \\ Enough spare parts ?};
    
        \node[yshift=1cm, minimum width=4cm, anchor=south](desct) at (xt1 |- repair){$(\reg{i}{t}, \age{i}{t})$};
        \node[yshift=1cm, minimum width=4cm, anchor=south](desct+1) at (repair){$(\reg{i}{t+1}, \age{i}{t+1})$};
        \draw(desct.south west) -- (desct+1.south east);
    
        \begin{scope}[line width=0.5pt, >=stealth, ->, every node/.style={font=\small, inner sep = 1pt, fill=white}]
            \draw (xt1.east) -- (pm.west);
            \draw (xt1.east) -- (nopm.west);
            \draw (pm.east) -- (repair.west);
            \draw (nopm.east) -- node {$p_{i}(\age{i}{t})$}(fail.west);
            \draw (nopm.east) -- node {$1 - p_{i}(\age{i}{t})$}(ageing.west);
    
            \draw (xt0.east) -- (spare.west);
            \draw (spare.east) -- node {Yes} (replace.west);
            \draw (spare.east) -- node {No} (broken.west);
        \end{scope}
    \end{tikzpicture}
    }
    \caption{Dynamics of component $i$.}
    \label{fig:comp_dyn}
\end{figure}
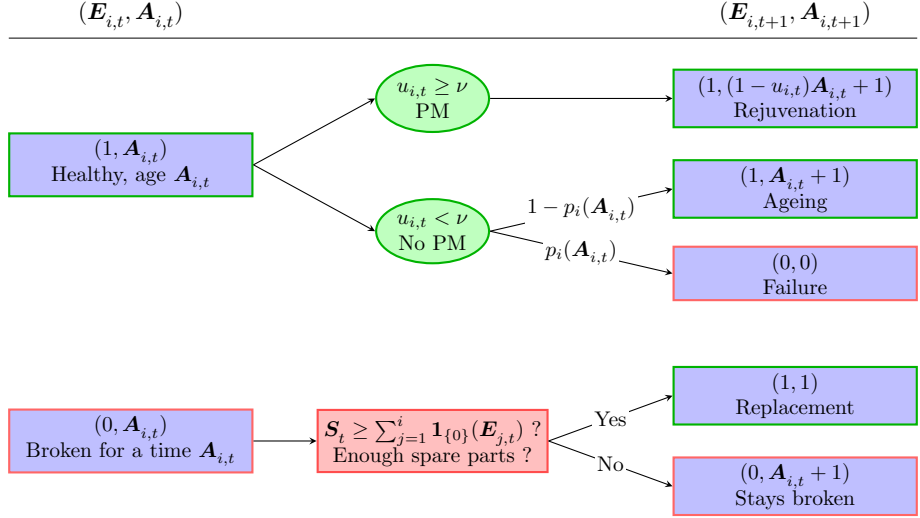
% Time is discretized with time steps $\Delta t$, taking $D = \left\lceil d/\Delta t \right\rceil$ allows to cover the worst case. This corresponds to the case where the component has failed at time $t_{1}$ and at all other time steps between $t_{1}$ and $t_{1} + d/\Delta t$ (time of return of the first part in the stock). In this case we need to be able to store the dates of the last $\left\lceil d/\Delta t \right\rceil$ failures to compute the future dates of the stock replenishment. We set $P_{i,0} = (-1, \ldots, -1)$.
% which is not useful anymore because the corresponding ordered part has already arrived in the stock (by choice of $\delay$).
% In this case when a failure occurs we discard the information about the $\delay^{\text{th}}$ last failure. Then we insert the value corresponding to the last failure at the end of the vector.
%so the failures in $\spr{i}{t}$ are always sorted from the oldest to the newest. 
% Note that this time can be greater than the return time $d_{i}$ of a part in the stock if there are time steps when the component does not use a part (either it stays healthy or it is broken but the stock is empty). 
In Fig.~\ref{fig:comp_dyn}, we summarize the dynamics of component $i$ from $t$ to $t+1$. Recall that we have $\statecomptime{i}{t} = (\reg{i}{t}, \age{i}{t}, \spr{i}{t})$. We write the \emph{dynamics of component $i$} on the whole time horizon as:
\begin{equation}
    \dyncomp{i}(\statecomponent{1:i}, \stoalltime, \ctrlcomp{i}, \noisecomp{i}) = 0 \eqfinv
\end{equation}
where $\dyncomp{i}=\na{\dyncomp{i,t}}_{t\in\timesteps}$ is such that:
\begin{align}
    &
      \left\{
      \begin{aligned}
        \dyncomptime{i}{t+1}(\statecomponent{1:i}, \stoalltime, \ctrlcomp{i}, \noisecomp{i})
        &= \statecomponent{i,t+1} - f_{i}\np{\statecomponent{1:i,t}, \sto{t}, \ctrlcomptime{i}{t}, \noise{i}{t+1}}, \quad t \in \timestepsnoend \eqfinv\\
        \dyncomptime{i}{0}(\statecomponent{1:i}, \stoalltime, \ctrlcomp{i}, \noisecomp{i})
        &= \statecomponent{i,0} - x_{i}\eqsepv
      \end{aligned}
          \right. \label{eq:dyn_comp_i}
\end{align}
with $x_{i} = (1, 0, \sprnofail, \ldots, \sprnofail)\transpos$ and $f_{i}$ represents the dynamics we just described for component $i$. An explicit expression of $f_{i}$ is given in~Appendix~\ref{sec:rew_syst_dyn}.

Note that there is a coupling between the dynamics of component $i$ and the stock. There is also a coupling with components $j < i$. This is due to the choice~\textup {(\ref{eq:avail_comp})} of replacing the broken components with the smallest indices first if there are not enough spare parts.

\subsubsection{Dynamics of the stock}
\label{subsubsec:dyn_sto}

% Stock dynamics
For the stock, the initial number of spare parts is $\sto{0} = s$. As PMs can be anticipated, we consider that the required spares are ordered so that they arrive just on time for the scheduled maintenance. Therefore, they do not appear in the dynamics of the stock. A part is used for each CM and a new part is ordered only after the failure of a component. The number of time steps for the supply of a part is $\delay$. Hence, the part ordered after the $d$-th undiscarded failure of component $i$ arrives in the stock at $t+1$ if $\spr{i}{t+1}^{d} = \delay$. This is equivalent to $\spr{i}{t}^{d} = \delay-1$. On the other hand, the number of broken components is 
$\sum_{i=1}^{n} \indic{\na{0}}{\reg{i}{t}}$
% $\sum_{i=1}^{n} (1-\reg{i}{t})$
and we replace as many of them as possible given the current level of stock $\sto{t}$. Thus, we have:
\begin{equation}
    \sto{t+1} = \sto{t} + \sum_{i=1}^{n} \sum_{d=1}^{\delay} \indic{\na{\delay-1}}{\spr{i}{t}^{d}} - \min\left\{\sto{t}, \ \sum_{i=1}^{n} \indic{\na{0}}{\reg{i}{t}} \right\}, \quad t \in \timestepsnoend \eqfinp
    \label{eq:stock}
\end{equation}
We write the \emph{dynamics of the stock} in compact form as:
\begin{equation}
    \dyncomp{\stoalltime}(\statecomponent{1:n}, \stoalltime) = 0 \eqfinv
\end{equation}
where $\dyncomp{\stoalltime}= \na{\dyncomptime{\stoalltime}{t}}_{t\in \timesteps}$ is such that:
\begin{align}
    &
      \left\{
      \begin{aligned}
        \dyncomptime{\stoalltime}{t+1}(\statecomponent{1:n}, \stoalltime)
        &= \sto{t+1} - f_{\stoalltime}(\statecomptime{1:n}{t}, \sto{t}), \quad t \in \timestepsnoend \eqfinv\\
        \dyncomptime{\stoalltime}{0}(\statecomponent{1:n}, \stoalltime)
        &= \sto{0} - s\eqsepv
      \end{aligned}
          \right. \nonumber\label{eq:dyn_stock}
\end{align}
with $f_{S}$ corresponding to the right-hand side of~\textup {(\ref{eq:stock})}. Note that $\sto{t+1}$ depends on on the current level of stock $\sto{t}$ but also on $\statecomptime{i}{t}$ for all $i \in \components$. The stock is coupling all the components of the system.

Finally, the \emph{dynamics of the whole system} is summarized by the almost sure equality constraint
\(\dyn(\stateallcomp, \stoalltime, \ctrlallcomp, \noiseallcomp) = 0\),
where we have $\dyn : \primspace \times \espacef{W} \rightarrow \espacef{L}$, with
$\dyn = \ba{\na{\dyncomp{i}}_{i\in\components},\dyncomp{\stoalltime}}$ 
% and $q = (n(D+2) + 1)(T+1)$
and $\espacef{L} = \vardelim{\prod_{i=1}^{n} \espacef{L}_{i}}\times \espacef{L}_{\va{S}}$
where $\espacef{L}_{i}$ is the set of random variables with range in $\bbR^{(\delay+2)(T+1)}$ and $\espacef{L}_{\va{S}}$ is the set of random variables  with range in $\bbR^{(T+1)}$.

We have now completely described the dynamics of the system. In the next part we specify the costs associated to the system. 

\subsection{Costs generated by the system}
\label{subsec:cost}

The costs generated by the system are due to PMs, CMs and forced outages of the unit. In practice as PMs are scheduled in advance, they are cheaper than unpredictable CMs. A forced outage of the unit induces a loss of production. It is characterized by a yearly cost which is higher than that of a PM or a CM. We consider a discount rate $\tau$ meaning that a cost $c$ occurring at time $t$ will be valued $\disc{t} c$ with the discount factor
\(
\disc{t} := \frac{1}{(1+\tau)^{t}}.
\)
We introduce the following notations:
\begin{itemize}
    \item $\costpm_{i,t}(\ctrlcomptime{i}{t})$ is the PM cost incurred at time $t$ for component $i$. Let $\cpm_{i}$ be the cost of a PM operation on component $i$. We set:
    \begin{equation}
        \costpm_{i,t}(\ctrlcomptime{i}{t}) = \disc{t} \cpm_{i}\ctrlcomptime{i}{t}^{2} \eqfinp
        \label{eq:pm_cost}
    \end{equation}
    We use a quadratic cost as it is strongly convex and should favor numerical convergence. In particular, in the case where $0 < \ctrlcomptime{i}{t} < \nu$, which models a situation where no PM is performed, we have $\costpm_{i,t}(\ctrlcomptime{i}{t}) > 0$.\footnote{The fact that $\costpm_{i,t}(\ctrlcomptime{i}{t}) > 0$ when $0 < \ctrlcomptime{i}{t} < \nu$ is favorable from a numerical point of view. For $0 < \ctrlcomptime{i}{t} < \nu$, we always have $\costpm_{i,t}(\ctrlcomptime{i}{t}) > \costpm_{i,t}(0)$ while no PM is performed. Hence, the system dynamics is the same with $\ctrlcomptime{i}{t} = 0$ and $0 < \ctrlcomptime{i}{t} < \nu$ while the cost is higher for $0 < \ctrlcomptime{i}{t} < \nu$. Therefore, the control $\ctrlcomptime{i}{t} = 0$ is always better to $0 < \ctrlcomptime{i}{t} < \nu$. This feature will allow us to clearly distinguish the steps where a PM is performed from the others.}
    \item $\costcm_{i,t}(\statecomptime{i}{t})$ is the CM cost. It is due at the time of the failure of a component, even if there is no spare part to perform the operation immediately. Hence it only occurs when $(\reg{i}{t}, \age{i}{t}) = (0,0)$. Let $\ccm_{i}$ be the cost of a CM operation on component $i$. We have:
    \begin{equation}
        \costcm_{i,t}(\statecomptime{i}{t}) = \disc{t} \ccm_{i}\indic{\na{0}}{\reg{i}{t}}\indic{\na{0}}{\age{i}{t}} \eqfinp
        \label{eq:cm_cost}
    \end{equation}
    \item $\costfo_{t}(\statecomptime{1:n}{t})$ is the forced outage cost. As all components belong to the same production unit, a forced outage occurs when at least one component is in a failed state and the CM has not been performed immediately because of a lack of spare part. 
    % However this cost is not cumulative when more than one component are broken. Once the unit is shut down, it makes no difference if one or more components are broken. 
    Let $\cfo$ be the forced outage cost per time unit. We have:
    \begin{equation}
        \costfo_{t}(\statecomptime{1:n}{t}) = \disc{t} \cfo \min\left\{1,\ \sum_{i=1}^{n} \indic{\na{0}}{\reg{i}{t}}\indic{\bbRpe}{\age{i}{t}}\right\} \eqfinp
        \label{eq:fo_cost}
    \end{equation}
\end{itemize}
In order to consider the previous costs over the whole study period we introduce:
\begin{itemize}
  \item the \emph{total maintenance cost} (preventive and corrective) generated by component $i \in \components$ on the studied period: % on whole study period:
    \begin{equation}
        j_{i}(\statecomponent{i}, \ctrlcomp{i}) = \sum_{t=0}^{T-1} \costpm_{i,t}(\ctrlcomptime{i}{t}) + \sum_{t=0}^{T} \costcm_{i,t}(\statecomptime{i}{t})
      \eqsepv
      \label{eq:costm_compact}
    \end{equation}
  \item the \emph{total forced outage cost} generated by the system during the studied period: % on the whole study period is
  \begin{equation}
    \costfo(\statecomponent{1:n}) = \sum_{t=0}^{T} \costfo_{t}(\statecomptime{1:n}{t})\eqsepv
    \label{eq:costfo_compact}
  \end{equation}
\end{itemize}

\subsection{Formulation of the maintenance optimization problem}
\label{subsec:form_opt_pb}

The dynamics of the system is stochastic as it depends on the failure of the components, modelled by the random vector $\noiseallcomp$. The cost function is then stochastic as well. The objective is to find the deterministic maintenance strategy $u \in \espacea{U}$ that minimizes the expected cost generated by the system over all failure scenarios. Hence, the industrial optimal maintenance scheduling problem is formulated as follows:
\begin{equation}
  \begin{aligned}
    \min_{\optimspace} &\ \Besp{\sum_{i=1}^{n} j_{i}(\statecomponent{i}, \ctrlcomp{i}) + \costfo(\statecomponent{1:n})} \\
    \sucht &\ \dyn(\stateallcomp, \stoalltime, \ctrlallcomp, \noiseallcomp) = 0 \eqfinv \quad \Pas \eqfinp
  \end{aligned}
  \label{eq:optim_pb_no_time}
\end{equation}

The term $\expec{\sum_{i=1}^{n} j_{i}(\statecomponent{i}, \ctrlcomp{i})} = \sum_{i=1}^{n} \expec{j_{i}(\statecomponent{i}, \ctrlcomp{i})}$ is additive with respect to the components whereas $\expec{\costfo(\statecomponent{1:n})}$ induces a non-additive coupling between the components. In the theoretical part 
on the Auxiliary Problem Principle (Section~\ref{sec:app}), we will see that these two terms are treated in a different way for the design of a decomposition-coordination algorithm.

\subsection{Complexity of the problem and limits of a direct approach}
\label{subsec:mads}

For maintenance scheduling problems that involve systems that are similar to the one considered in this paper, the algorithm MADS (Mesh Adaptive Direct Search)~\cite{audet_mesh_2006} can be used. In particular, MADS has been successfully applied for periodic maintenance problems in~\cite{aoudjit_planification_2010} and for a system with $4$ components and monthly maintenance decisions in~\cite{alarie_optimization_2019}. However, the problem considered here with $80$ components and general maintenance strategies is more challenging than the two previous ones.

The industrial maintenance problem~\eqref{eq:optim_pb_no_time} is a non-smooth non-convex optimization problem. The number of components $n$ in the system can be up to $80$ and $T$ is $40$ years so $\espacea{U}$ has dimension up to $nT = 3200$ as we have one maintenance decision each year for each component. In \S\ref{subsec:syst_dyn}, the dynamics is modelled with the vector $\dyn$, taking values in $\mathcal{L} = \bbR^{(n(D+2) + 1)(T+1)}$. Hence, there are $(n(D+2) + 1)(T+1)$ constraints in Problem~\eqref{eq:optim_pb_no_time} \emph{i.e.} $13161$ contraints for the systems studied in Section~\ref{sec:num_res} with a delay $D=2$ years for the supply in spare parts. However, any maintenance strategy $u \in \espacea{U}$ is feasible, as we can always choose whether or not to do a PM for each component at each time step. The dynamics $\dyn$ then allows to compute the value of $\stateallcomp$ and $\stoalltime$ given the strategy $u$ and the occurrences of failures modelled by $\noiseallcomp$.

As the problem is non-smooth, non-convex and in high dimension, it is challenging for any frontal resolution approach, including the algorithm MADS. MADS is a blackbox algorithm, meaning that evaluation points are chosen iteratively without using the gradients of the objective function, which may not even be defined. This feature is particularly appealing as the cost function is not differentiable. The algorithm has been initially designed for continuous optimization, therefore it uses the modeling of~\S\ref{subsec:gen_syst} and the cost functions described in~\S\ref{subsec:cost}. In particular, the PM strategies are modelled with a continuous decision variable. When the number of components is not too large, say $n<10$, the maintenance problem~\eqref{eq:optim_pb_no_time} is solved efficiently by MADS: the algorithm converges in around a minute.
% In practice, the objective function is costly to evaluate as the expectation is estimated using Monte-Carlo simulations.
When the number of components is large, MADS needs more iterations to explore the high-dimensional space of solutions. As the volume of the search space grows exponentially with the number of components, MADS may not be able to find an effective maintenance strategy in reasonable time (less than a day) for the most demanding case with $80$ components.
% \todo[inline]{Résultats du cas 5 dates}
% \begin{itemize}
%     \item The admissible set $\espacea{U} = \{0,1\}^{n(T+1)}$ is large especially when the number of components $n$ in the system becomes important: recall that $n$ can be up to $80$. If we solve the problem directly, we face the so-called curse of dimensionality.
%     % \item The expected cost $J$ is estimated using Monte-Carlo simulations. However, given that the forced outage cost is rare but very costly, a large number of Monte-Carlo simulations are needed to estimate $J(\ctrlallcomp) = \espe(j(\ctrlallcomp,\stateallcomp))$ with acceptable precision. This results in high computation time for each estimation of $J(\ctrlallcomp)$.
%     \item The maintenance strategies are defined on a discrete set so the gradients of $J$ are not defined, we are in the context of blackbox optimization. The algorithm Mesh Adaptive Direct Search (MADS)~\cite{audet_mesh_2006} is adapted to this context and can solve expensive blackbox optimization problems as long as the admissible set is not too large. When the number of components in the system is large, MADS is not able to solve the problem in reasonable time. By reasonable we mean that the problem should be solved in a time of the order of a day.
% \end{itemize}

To overcome the difficulty of MADS when dealing with large systems, we implement a decomposition of Problem~\eqref{eq:optim_pb_no_time} component by component. 

% In the decomposition method, we use the same modeling for the system as in MADS. In particular, we use the continuous modeling of PM strategies given in~\S\ref{subsubsec:model_pm}.
% This turns the original problem into the resolution of small subproblems on each component. Each lower-dimensional subproblem is solved by MADS followed by a coordination step of the local solutions. Iterating the process of decomposition-coordination leads to a solution of the original problem.

%%% Local Variables:
%%% mode: latex
%%% TeX-master: "optim_maint_sched_v2"
%%% End:

%% file: app_decomp.tex
% !TEX root = optim_maint_sched_v2.tex

% \section{A decomposition-coordination scheme for maintenance optimization}

% \subsection{Introduction to decomposition-coordination methods}
\section{The Auxiliary Problem Principle for decomposition}
\label{sec:app}

When performing optimization on a large scale system, that is, a system which is described by a large number of variables or constraints, the computation is often highly expensive either in time or in memory requirement (or both). The idea of decomposition is to formulate subproblems involving only smaller subsystems of the original large system. Each subproblem is easier to solve than the global optimization problem and provides its ``local'' solution. Then the goal of coordination is to ensure that gathering the local solutions leads to a global solution. Decomposition-coordination methods usually result in an iterative process alternating an optimization step on the subsystems and a coordination step that updates the subproblems.

The main advantage of decomposition is that the resolution of the small subproblems is faster than the original problem. More than that, the computational complexity of an optimization problem is often superlinear or even exponential in the size of the problem. Hence, the sum of the computational efforts required for the resolution of all subproblems will be lower than for the global problem, even if the resolution of the subproblems must be carried out multiple times. %This will be the case as the coordination scheme uses the solutions of the subproblems at each iteration.
Another feature of decomposition methods is that they are naturally adapted to parallelization as each subproblem is independent. This leads to a reduction of computation time.

In this section, we introduce the general framework of the Auxiliary Problem Principle (APP) that allows us to formulate a decomposition-coordination scheme for the optimization problem~\eqref{eq:optim_pb_no_time}.
The APP has first been introduced in~\cite{cohen_optimization_1978} as a unified framework for decomposition methods but also for other classical iterative algorithms. This principle casts the resolution of an optimization problem into the resolution of a sequence of auxiliary problems whose solutions converge to the solution of the original problem. Appropriate choices for the auxiliary problems lead to decomposition-coordination schemes. 

Based on~\cite{carpentier_decomposition-coordination_2017}, we present the main ideas of the APP. Consider the following problem, which we call the master problem:
\begin{equation}
    \min_{u \in \Uad} \costadd(u) + \costdiff(u) \ \text{ such that } \ \dyn(u) \in -C \eqfinv
    \label{eq:app_pb_cst}
\end{equation}
where:
\begin{itemize}
    \item $\Uad$ is a non-empty, closed, convex subset of a Hilbert space $\espacea{U} = \espace{U}_{1} \times \ldots \times \espacea{U}_{N}$ and is decomposable as % \todo{Conflit avec le $\espacea{U}$ et $J$ de la partie précédente?} 
    $\Uad = \Uad_{1} \times \cdots \times \Uad_{N}$
    where for all $i \in \components,\ \Uad_{i} \subset \espacea{U}_{i}$ is a closed convex set;
    \item $C$ is a pointed closed convex cone of a Hilbert space $\mathscr{C} = \mathscr{C}_{1} \times \cdots \times \mathscr{C}_{N}$ and is decomposable as $C= {C}_{1} \times \cdots \times {C}_{N}$
    where for all $i \in \components,\ C_{i} \subset \mathscr{C}_{i}$ is a closed convex cone;
    \item $\costadd: \espacea{U} \rightarrow \bbR$ and $\costdiff: \espacea{U} \rightarrow \bbR$ are convex and lower semi-continuous (\lsc). The function $\costadd + \costdiff$ is coercive on $\Uad$. Moreover, $\costdiff$ is differentiable and $\costadd$ is additive with respect to the decomposition of the admissible space, so that we have for $u = (u_{1}, \ldots, u_{N}) \in \Uad$ with $u_{i} \in \Uad_{i}$ for all $i \in \components$:
    \begin{equation}
      \costadd(u) = \sum_{i=1}^{N} \costaddi(u_{i})\eqfinpv
    \end{equation}
    where $\costaddi : \espacea{U}_{i} \rightarrow \bbR$.
    % \todo{Cohérence avec les notations Hilbert et espaces fonctionnels?}
    \item $ \dyn: \espacea{U} \rightarrow \mathscr{C}$ is differentiable and $C$-convex, meaning that:
    \begin{equation}
        \forall u, v \in \espacea{U}, \forall \rho \in [0,1], \rho \dyn(u) + (1-\rho) \dyn(v) - \dyn(\rho u + (1-\rho)v) \in C \eqfinp
    \end{equation}
    We write $\dyn(u) = (\dyncomp{1}(u), \ldots, \dyncomp{N}(u))$ with $\dyncomp{i} : \espacea{U} \rightarrow \mathscr{C}_{i}$ for $i \in \components$.
\end{itemize}
The goal of the APP is to turn the resolution of the master problem on $\Uad$ into the resolution of subproblems on the sets $\Uad_{i}$. In the master problem, there is a non-additive coupling in the cost due to $\costdiff$ and a coupling in the constraint $\dyn$. 
 
%%%%%%%%%%%%%%% Seulement pour le manuscrit de thèse %%%%%%%%%%%%%%%%%%%%%%%%%%
% and choosing $\varepsilon = 1$ in~\cref{eq:aux_fnl} we are looking for a saddle point of
% \begin{multline}
%         G_{v,q
%         %, (r,s), Q, \varepsilon
%         }(u,p)  = K(u) + \costadd(u) + \proscal{\costdiff'(v) - K'(v)}{u} + \proscal{q}{(\dyn'(v) - \auxdyn'(v)) \cdot u}\\
%         + \ \proscal{p}{\auxdyn(u) + \dyn(v) - \auxdyn(v)}
%     \label{eq:aux_fnl_cst}
% \end{multline}
% We can interpret $G_{v,q}(u,p)$ as the Lagrangian of a minimization problem under constraints. Doing so and using the idea of a fixed-point method suggested by~\cref{lem:fun_app_cst}, the resulting algorithm is the following.
%%%%%%%%%%%%%%%%%%%%%%%%%%%%%%%%%%%%%%%%%%%%%%%%%%%%%%%%%%%%%%%%%%%%%%%%%%%%%%%%
We present the APP for the special case of the decomposition by prediction. The construction of a decomposable auxiliary problem relies on the following points:
\begin{enumerate}
    \item The decomposition of the admissible space $\Uad = \Uad_{1} \times \cdots \times \Uad_{N}$ defines the subspace on which each subproblem is solved.
    \item  The decomposition of the cone $C= {C}_{1} \times \cdots \times {C}_{N}$ specifies which part of the constraint is assigned to each subproblem.
\end{enumerate}

Let $\bar{u} = (\bar{u}_{1}, \ldots, \bar{u}_{N}) \in \espacea{U}$ and $\bar{\lambda} = (\bar{\lambda}_{1}, \ldots, \bar{\lambda}_{N}) \in C\dual$ the dual cone of $C$. We recall that:
\begin{equation}
    C\dual = \na{\lambda \in \mathscr{C}\dual, \proscal{\lambda}{\mu} \geq 0 \text{ for all } \mu \in \mathscr{C}} \eqfinv
\end{equation}
with $\mathscr{C}\dual$ being the dual space of $\mathscr{C}$.
Let $K$ be an auxiliary cost and $\auxdyn$ be an auxiliary dynamics satisfying the following properties:
\begin{itemize}
    \item $K: \espacea{U} \rightarrow \bbR$ is convex, \lsc, differentiable and additive: $\displaystyle K(u) = \sum_{i=1}^{N} K_{i} (u_{i})$;
    \item $\auxdyn : \espacea{U} \rightarrow \mathscr{C}$ is differentiable and block-diagonal:
        $\auxdyn(u) = \bp{\auxdyncomp{1}(u_{1}), \ldots, \auxdyncomp{N}(u_{N})}$.
\end{itemize}
The \emph{auxiliary problem} for the master problem~\textup {(\ref{eq:app_pb_cst})} arising from the choice of $K$ and $\auxdyn$ is given by:
\begin{equation}
    \begin{aligned}
        \label{eq:cst_ap}
        \min_{u \in \Uad} & K(u) + \costadd(u)
        + \proscal{\nabla\costdiff(\bar{u}) - \nabla K(\bar{u})}{u} + \proscal{\bar{\lambda}}{\bp{\dyn'(\bar{u}) - \auxdyn'(\bar{u})} \cdot u}\\
        &\sucht  \auxdyn(u) - \auxdyn(\bar{u}) + \dyn(\bar{u})\in -C \eqfinp
    \end{aligned}
\end{equation}

Choosing $K$ additive and $\auxdyn$ block-diagonal ensures that Problem~\textup {(\ref{eq:cst_ap})} decomposes in $N$ independent subproblems
that can be solved in parallel. For $i \in \components$, the $i$-th subproblem is given by:
\begin{equation}
    \begin{aligned}
        \min_{u_{i} \in \Uad_{i}} &
        \begin{multlined}[t]
            K_{i}(u_{i}) + \costaddi(u_{i}) + \proscal{\nabla_{u_{i}}\costdiff(\bar{u}) - \nabla K_{i}(\bar{u}_{i})}{u_{i}} \\
            - \proscal{\bar{\lambda}_{i}}{\auxdyncomp{i}'\np{\bar{u}_{i}} \cdot u_{i}}
            + \sum_{j=1}^{N}\proscal{\bar{\lambda}_{j}}{ \partial_{u_{i}}\dyncomp{j}\np{\bar{u}} \cdot u_{i}}
        \end{multlined}\\
    &\sucht  \auxdyncomp{i}(u_{i}) - \auxdyncomp{i}(\bar{u}_{i}) + \dyncomp{i}(\bar{u})\in -C_{i} \eqfinp
        \label{eq:cst_ap_sub_pb}
    \end{aligned}
\end{equation}
This subproblem only depends on $u_{i} \in \Uad_{i}$ and inherits only the $i$-th component of the constraint.

% \begin{remark}
%     The construction of the auxiliary problem~\textup {(\ref{eq:cst_ap})} is also valid in the case where $\costdiff$ is only subdifferential.
% \end{remark}
%
% \begin{remark}
%     In~\cref{eq:cst_ap_sub_pb} the $i$-th subproblem only bears a part of the constraints of the master problem. In a way, we have allocated some constraints to each subproblem. This is the main idea of the interaction prediction principle introduced in~\cite{mesarovic_theory_1970}. In the master problem, when the costs are not additive or the constraints are not block-diagonal, the objective function of the subproblems is modified in order to include linearized costs that account for the interactions between subproblems. The interaction costs are the two inner products that appear in the objective function of~\textup {(\ref{eq:cst_ap})}.
%     \label{rem:inter_pred}
% \end{remark}

\begin{example}
   Let $\bar{u} \in \espacea{U}$. A canonical choice for the additive auxiliary cost function $K$ is:
    \begin{equation}
        K(u) = \sum_{i=1}^{N} K_{i}(u_{i}) \ \text{ with }\ K_{i}(u_{i}) = 
        \costdiff(\bar{u}_{1:i-1}, u_{i}, \bar{u}_{i+1:n}) \eqfinp
        % \costdiff(\bar{u}_{1}, \cdots, \bar{u}_{i-1}, u_{i}, \bar{u}_{i+1}, \cdots, \bar{u}_{N})
    \end{equation}
    where $u_{i:j} = (u_{i}, \ldots, u_{j})$ for $i \leq j$ and the convention that $u_{i:j}$ is empty if $j > i$. Similarly, a canonical choice for the block-diagonal auxiliary dynamics $\auxdyn$ is:
    \begin{equation}
        \auxdyn(u) = (\auxdyncomp{1}(u_{1}), \ldots, \auxdyncomp{N}(u_{N})) \ \text{ with } \ 
        \auxdyncomp{i}(u_{i}) = \dyncomp{i}(\bar{u}_{1:i-1}, u_{i}, \bar{u}_{i+1:n}) \eqfinp
        % \auxdyncomp{i}(u_{i}) = \dyncomp{i}(\bar{u}_{1}, \cdots, \bar{u}_{i-1}, u_{i}, \bar{u}_{i+1}, \cdots, \bar{u}_{N})
    \end{equation}
    The general idea is to construct the $i$-th term of the auxiliary function from the original function where only the $i$-th component is allowed to vary.
    \label{ex:canon_tech}
\end{example}
The following statement is the fundamental lemma for the theory of the APP.
\begin{lemma}
    Let $u\opt$ be a solution of the auxiliary problem~\textup {(\ref{eq:cst_ap})} and $\lambda\opt$ be an optimal multiplier for its constraint. If $(u\opt, \lambda\opt) = (\bar{u},\bar{\lambda})$, then $u\opt$ is a solution of the master problem~\textup {(\ref{eq:app_pb_cst})} and $\lambda\opt$ is an optimal multiplier for its constraint.
    \label{lem:fun_app_cst}
\end{lemma}
The proof consists in checking that if $(u\opt, \lambda\opt) = (\bar{u},\bar{\lambda})$, then
$(u\opt, \lambda\opt)$ solves the variational inequalities that are satisfied by an optimal solution
and an optimal multiplier of the master problem~\textup {(\ref{eq:app_pb_cst})}. More details can be found in~\cite[Section 4]{carpentier_decomposition-coordination_2017}. Lemma~\ref{lem:fun_app_cst} suggests
to use the APP fixed-point Algorithm~\ref{alg:fix_pt_cst}.

\begin{algorithm}[H]
    \begin{algorithmic}[1]
        \myState{Let $(\bar{u}, \bar{\lambda}) = (u^{0}, \lambda^{0}) \in \Uad \times C\dual$ and set $k=0$.}
        \myState{At iteration $k+1$:
        \begin{itemize}
            \item Solve the auxiliary problem~\eqref{eq:cst_ap}. Let $u^{k+1}$ be a solution and $\lambda^{k+1}$ be an optimal multiplier for its constraint.
            \item Set $(\bar{u}, \bar{\lambda}) = (u^{k+1}, \lambda^{k+1})$.
        \end{itemize}
        } \label{lig:ap_cst}
        \myState{If the maximum number of iterations is reached or $\norm{u^{k+1} - u^{k}}$ + $\norm{\lambda^{k+1} - \lambda^{k}}$ is \emph{sufficiently small} then stop, else $k \gets k+1$ and go back to step~\ref{lig:ap_cst}.}
    \end{algorithmic}
    \caption{APP fixed-point algorithm}
    \label{alg:fix_pt_cst}
\end{algorithm}
A convergence result for~Algorithm~\ref{alg:fix_pt_cst}, is given in~\cite{cohen_auxiliary_1980}.

\begin{theorem}{\cite[Theorem 5.1 with $\epsilon=1$]{cohen_auxiliary_1980}}
    Assume that:
    \begin{itemize}
        \item The admissible space $\Uad$ is equal to the whole space $\espacea{U}$,
        \item The constraints are equality constraints, that is $C = \na{0}$,
        \item $
          K(u) = \frac{1}{2}\proscal{u}{\mathsf{K}u}, \ \costadd(u) = 0, \ \costdiff(u) = \frac{1}{2}\proscal{u}{\mathsf{J}u} + \proscal{\mathsf{j}}{u}
          $
          where $\mathsf{K}$ and $\mathsf{J}$ are linear self-adjoint strongly monotone and Lipschitz continuous operators,
          $\mathsf{j}$ is a vector in $\espacea{U}$, and $2\mathsf{K} - \mathsf{J}$ is assumed to be strongly monotone,
        \item $
            \auxdyn(u) = \mathsf{O}u, \ \dyn(u) = \mathsf{T}u + \mathsf{t}
        $
        where the operators $\mathsf{O}$ and $\mathsf{T}$ are linear and surjective and $\mathsf{t}$ is a vector in $\mathscr{C}$,
        \item (Geometric condition) The operator
        $
            2\vardelim{\mathsf{T}\mathsf{J}^{-1}\mathsf{O}\transpos + \mathsf{O}\mathsf{J}^{-1}\mathsf{T}\transpos} - \mathsf{T}\mathsf{J}^{-1}(2\mathsf{K} + \mathsf{J})\mathsf{J}^{-1}\mathsf{T}\transpos
        $
        is strongly monotone.
    \end{itemize}
    Then, 
    % in~\cref{eq:ap_cst_k} , 
    % the parameter 
    % $\varepsilon$ can be chosen small enough and 
    % $\gamma$ can be chosen large enough so that the operator
    % $
    %     2\mathsf{K} - \varepsilon\mathsf{J}
    % $
    % and
    % $
    %     2\gamma(\mathsf{T}\mathsf{J}^{-1}\mathsf{O}\transpos + \mathsf{O}\mathsf{J}\transpos\mathsf{T}\transpos) - \mathsf{T}\mathsf{J}^{-1}(2\mathsf{K} + \mathsf{J})\mathsf{J}^{-1}\mathsf{T}\transpos
    % $
    % is strongly monotone in which case 
    the sequence $\ba{\np{u^{k}, \lambda^{k}}}_{k \in \bbN}$ generated by the APP fixed-point algorithm converges strongly to the
    unique optimal solution $(u\opt, \lambda\opt)$ of the original problem~\eqref{eq:app_pb_cst}.
    \label{thm:conv_fix_pt}
\end{theorem}

% \Cref{thm:conv_fix_pt} is a version of~\cite[Theorem 5.1]{cohen_auxiliary_1980} in the particular case where the parameter $\varepsilon$ from~\cite[Theorem 5.1]{cohen_auxiliary_1980} is chosen equal to $1$. We assume that $K$ is designed to be strongly convex enough so that this choice is possible.

In this section, we have studied a master optimization problem with couplings coming both from the part $\costdiff$ of the cost function and the constraint $\dyn$. The APP allows to turn the resolution of the master problem into the iterative resolution of an auxiliary problem that is decomposable into independent subproblems of smaller size. In the next section, we present the application of the APP to solve the maintenance optimization problem described in~Section~\ref{sec:description}.

%%% Local Variables:
%%% mode: latex
%%% TeX-master: "optim_maint_sched_v2"
%%% End:

%% file: aux_pb_indus.tex
% !TEX root = optim_maint_sched_v2.tex

\section{Application of the APP for the industrial maintenance problem}
\label{sec:app_syst}

% In order to apply the APP to solve Problem~\eqref{eq:optim_pb_no_time}, also referred as original problem in the sequel, we decompose the admissible space and then construct an auxiliary problem adapted to the proposed decomposition. 
% The difficulty for the industrial case comes from the non-smoothness of the dynamics and the cost function. The system needs to be relaxed in order to apply the fixed-point algorithm. This relaxation is the object of Section~\S\ref{subsec:relax_indus}.
The application of the APP to Problem~\eqref{eq:optim_pb_no_time} is done as follows:
\begin{enumerate}
    \item We specify a decomposition of the admissible space and of the space of contraints~(\S\ref{subsec:decomp_space}).
    \item We construct an auxiliary problem that is adapted to the proposed decomposition using a suitable auxiliary cost and auxiliary dynamics~(\S\ref{subsec:constr_aux_pb}).
    \item The industrial system involves integer variables whereas the APP is based on variational techniques. Therefore, we proceed to a continuous relaxation of the system in order to compute the gradients of the dynamics that appear in the auxiliary problem~(\S\ref{subsec:relax_indus}).
    \item We give the explicit formulation of the independent subproblems that arise from the decomposition of the auxiliary problem~(\S\ref{subsec:decomp_ap}).
    \item We see that the subproblem on the stock is easy to solve numerically. We take advantage of this feature to design an efficient implementation of the APP fixed-point algorithm mixing a parallel and sequential strategy for the resolution of the subproblems~(\S\ref{subsec:fix_pt_alg}).
\end{enumerate}
The APP fixed-point algorithm presented in~\S\ref{subsec:fix_pt_alg} is then applied to numerical test cases in Section~\ref{sec:num_res}.

\subsection{Decomposition of the space by component}
\label{subsec:decomp_space}

We start by specifying the decomposition of the admissible space and of the cone of constraints that is considered for the design of the decomposition by component.

Considering the physical nature of the industrial system composed of $n$ components and a stock,
we choose to decompose the problem in $n+1$ subproblems and call this decomposition a \emph{decomposition by component}.
More precisely, for $i\in\components$, the $i$-th subproblem
is called \emph{subproblem on component $i$} since it is solved on $\primdecomp{i}$ and only involves the dynamics of component $i$. The $(n+1)$-th subproblem is called
\emph{subproblem on the stock} since it is solved on $\espacef{S}$ and only involves the dynamics of the stock. This means that the admissible space 
$\primspace$ of Problem~\eqref{eq:optim_pb_no_time} is decomposed as a product of a $n+1$ subspaces:
\begin{equation}
    \primspace = (\primdecomp{1}) \times \ldots \times (\primdecomp{n}) \times \espacef{S}\eqfinv
    \label{eq:decomp_comp}
\end{equation}
where, for $(\stateallcomp, \stoalltime, \ctrlallcomp) = ((\statecomponent{1}, \ldots, \statecomponent{n}), \stoalltime, (\ctrlcomp{1}, \ldots, \ctrlcomp{n}))\in \primspace$, we have
\begin{equation}
  (\statecomponent{i}, \ctrlcomp{i}) \in \primdecomp{i} \text{ for all }  i \in \components
  \text{ and } \stoalltime \in \espacef{S}\eqfinp
\end{equation}
The constraints in Problem~\eqref{eq:optim_pb_no_time}, that is
$\dyn(\stateallcomp, \stoalltime, \ctrlallcomp, \noiseallcomp) \in -C$ with $C = \na{0}_{\espacef{L}}$ is decomposed through the following cone decomposition:
\begin{equation}
    % C\dual = \bbR^{q} = \bbR^{(D+2)(T+1)} \times \ldots \times \bbR^{(D+2)(T+1)} \times \bbR^{T+1} = C_{1}\dual \times \ldots C_{n}\dual \times C_{\stoalltime}\dual
    C = \na{0}_{\espacef{L}} = \na{0}_{\espacef{L}_{1}} \times \ldots \times \na{0}_{\espacef{L}_{n}} \times \na{0}_{\espacef{L}_{\va{S}}} = C_{1} \times \ldots C_{n} \times C_{\stoalltime} \eqfinv
    \label{eq:decomp_cone}
  \end{equation}
  where for $i \in \components, \ \na{0}_{\espacef{L}_{i}}, \na{0}_{\espacef{L}_{\stoalltime}}$ and $\na{0}_{\espacef{L}}$ denote the null function of $\espacef{L}_{i}, \espacef{L}_{\stoalltime}$ and $\espacef{L}$ respectively. We recall that $\dyncomp{i}$ takes values in $\espacef{L}_{i}$ and $\dyncomp{\stoalltime}$ takes values in $\espacef{L}_{\va{S}}$.

\subsection{Construction of an auxiliary problem}
\label{subsec:constr_aux_pb}

In this section, we choose the auxiliary functions that lead to the construction of a decomposable auxiliary problem.

Problem~\eqref{eq:optim_pb_no_time} is not directly decomposable by component because of couplings, highlighted in~Section~\ref{sec:description}, that we recall now.
\begin{itemize}
\item The expected maintenance cost $\expec{\sum_{i=1}^{n} j_{i}(\statecomponent{i}, \ctrlcomp{i})}$ is additive with respect to the decomposition by component and can be identified with $\costadd$ in~Section~\ref{sec:app}.
\item The forced outage cost $\costfo$ induces a non-additive coupling between the components. The expected forced outage cost $\expec{\costfo(\statecomponent{_{1:n}})}$ can be identified with $\costdiff$ in~Section~\ref{sec:app}.
\item The dynamics $\dyncomp{i}$ of component $i$ induces a coupling with the stock and all components with index $j<i$. The stock dynamics, $\dyncomp{\stoalltime}$, is coupling the stock with all components.
\end{itemize}
In order to obtain a decomposition of Problem~\eqref{eq:optim_pb_no_time} by component, we use the canonical technique from~Example~\ref{ex:canon_tech}. We define an additive mapping $K$ and a block diagonal mapping $\auxdyn$
so that the resulting auxiliary problem is decomposable.
We also choose to augment the auxiliary cost $K$ with a strongly convex term in order to ease the numerical convergence of the method.
Let $\stateallcompbar = (\statecompbar{1}, \ldots, \statecompbar{n}) \in \espacef{X}$, $\stoalltimebar \in \espacef{S}$, $\ctrlallcompbar \in \espacea{U}$, $\multallcompbar \in C\dual$ and $\gamma_{x}, \gamma_{s}, \gamma_{u} > 0$. We consider:
\begin{itemize}
\item An additive auxiliary cost function $K: \primspace \rightarrow \bbR$:
  \begin{equation}
    K(\stateallcomp, \stoalltime, \ctrlallcomp) = \sum_{i=1}^{n} K_{i} (\statecomponent{i}, \ctrlcomp{i}) + K_{\stoalltime}(\stoalltime)\eqfinv
  \end{equation}
  with:
  \begin{equation}
    % \left\{
    \begin{aligned}
      K_{i} (\statecomponent{i}, \ctrlcomp{i}) &= \Besp{\costfo(\statecompbar{1:i-1}, \statecomponent{i}, \statecompbar{i+1:n}) + \frac{\gamma_{x}}{2} \sqnorm{\statecomponent{i}} + \frac{\gamma_{u}}{2} \sqnorm{\ctrlcomp{i}}}, \quad i \in \components \eqfinv\\
      K_{\stoalltime}(\stoalltime) &= \Besp{\frac{\gamma_{s}}{2} \sqnorm{\stoalltime}}\eqfinv
    \end{aligned}
    % \right.
    \end{equation}
    \label{def:aux_cost}
    \item A block-diagonal auxiliary dynamics mapping $\auxdyn : \primspace \times \espacef{W} \rightarrow \espacef{L}$:
    \begin{equation}
      \auxdyn(\stateallcomp, \stoalltime, \ctrlallcomp, \noiseallcomp) = (\auxdyncomp{1}(\statecomponent{1}, \ctrlcomp{1}, \noisecomp{1}), \ldots, \auxdyncomp{n}(\statecomponent{n}, \ctrlcomp{n}, \noisecomp{n}), \auxdyncomp{\stoalltime}(\stoalltime))\eqfinv
    \end{equation}
    with:
    \begin{equation}
        % \left\{
        \begin{aligned}
            \auxdyncomp{i}(\statecomponent{i}, \ctrlcomp{i}, \noisecomp{i}) &= \dyncomp{i}(\statecompbar{1:i-1}, \statecomponent{i}, \stoalltimebar, \ctrlcomp{i}, \noisecomp{i}), \quad i \in \components \eqfinv \\
            \auxdyncomp{\stoalltime}(\stoalltime) &= \dyncomp{\stoalltime}(\statecompbar{1:n}, \stoalltime)\eqfinv
        \end{aligned}
        % \right.
        \label{eq:aux_dyn}
    \end{equation}
    with $\statecompbar{1:0}$ being by convention an empty vector.
\end{itemize}

We can now write an auxiliary problem for~\eqref{eq:optim_pb_no_time}. Assume that the dynamics $\dyn$ is differentiable. In this case, $\auxdyn$ is differentiable and the auxiliary problem writes:
\begin{equation}
    \begin{aligned}
        \min_{\optimspace}
        &\ \begin{multlined}[t]
            \espe \bigg(\sum_{i=1}^{n} \left(j_{i}(\statecomponent{i}, \ctrlcomp{i}) + \costfo(\statecompbar{1:i-1}, \statecomponent{i}, \statecompbar{i+1:n})\right)\\
            + \frac{\gamma_{x}}{2} \sqnorm{ \stateallcomp - \stateallcompbar } + \frac{\gamma_{s}}{2} \sqnorm{ \stoalltime - \stoalltimebar } + \frac{\gamma_{u}}{2} \sqnorm{\ctrlallcomp- \ctrlallcompbar }\\
            + \proscal{ \multallcompbar }{ (\dyn'(\prevopt, \noiseallcomp) - \auxdyn'(\prevopt, \noiseallcomp))\cdot (\stateallcomp, \stoalltime, \ctrlallcomp) } \bigg)
        \end{multlined}\\
        \sucht &\ \auxdyn(\stateallcomp, \stoalltime, \ctrlallcomp, \noiseallcomp) = 0\eqfinp
    \end{aligned}
    \label{eq:aux_pb_indus}
\end{equation}

% \begin{remark} Note that, in practice, at the $(k+1)$-th iteration of the fixed-point algorithm,
%   $(\prevopt) \in \primspace$ is set to the optimal solution of the auxiliary problem found at the $k$-th iteration
%   and $\multallcompbar$ is set to the optimal multiplier associated to the constraint $\auxdyn = 0$ found at the $k$-th iteration.
%   Thus, the auxiliary problem changes at each iteration of the algorithm.
% \end{remark}

By construction, the auxiliary problem~\textup {(\ref{eq:aux_pb_indus})} is decomposable with respect to the decompositions~\textup {(\ref{eq:decomp_comp})} and \textup {(\ref{eq:decomp_cone})}. 
% A detailed calculation to retrieve the auxiliary problem~\textup {(\ref{eq:aux_pb_indus})} from~\textup {(\ref{eq:cst_ap})} is given in~\cref{susbsec:app_aux_pb}.

\subsection{Relaxation of the system}
\label{subsec:relax_indus}

% In order to solve the original problem, we consider the auxiliary problem~\textup {(\ref{eq:aux_pb_indus})} and use the fixed-point~Algorithm~\ref{alg:fix_pt_cst}.

% In the auxiliary problem, the multiplier $\multallcompbar$ is an element of the dual cone $C\dual$ where $C = \na{0}_{\espacef{L}}$ and $\espacef{L} = \espaceva{\omeg, \trib, \prbt}{\bbR^{q}}$ with $q = (n(\delay + 2) + 1)(T + 1)$. In fact, the dual space of $\espacef{L}$ does not contain any nonzero continuous functional, hence $C\dual$ only contains the zero functional. In practice, we can consider that $\Theta$ takes values in the infinite dimensional Hilbert space $L^{2}(\omeg, \trib, \prbt; \bbR^{q})$. In this case, the dual cone is $C\dual = L^{2}(\omeg, \trib, \prbt; \bbR^{q})$. Note that~Theorem~\ref{thm:conv_fix_pt} is applicable in this infinite dimensional setting as constraint qualification follows from the assumption that there are only linear equality constraints. However, in the maintenance problem~\eqref{eq:optim_pb_no_time} the dynamics $\Theta$ is not linear but in the numerical experiments we only consider a finite number of failure scenarios $\omega \in \Omega$, see~\textup {(\ref{eq:optim_pb_saa})}. Hence, the optimization problem that is solved numerically has a finite number of constraints that are qualified.

In the industrial case, the assumptions of~Theorem~\ref{thm:conv_fix_pt} are far from being satisfied: the forced outage cost $\costfo$ is not quadratic, nor even convex,
the dynamics $\dyn$ is not linear, and it is hard to check if the geometric condition is satisfied.
% Despite all that, the decomposition method gives good numerical results.
Moreover, the APP relies on variational techniques and requires the mappings $\dyn$ and $\auxdyn$ to be differentiable
as the derivatives $\dyn'$ and $\auxdyn'$ appear in Problem~\textup {(\ref{eq:aux_pb_indus})}. However, the dynamics $\dyn$ in Problem~\eqref{eq:optim_pb_no_time} involves integer variables so $\dyn'$ is not defined. To overcome this difficulty, we propose a continuous relaxation of the system with relaxed cost and dynamics that are differentiable almost everywhere. It is possible to use a differentiable relaxation of the system but this requires more implementation efforts. As we are far from the conditions of convergence of the algorithm, nothing ensures that a differentiable relaxation would give better results than the simple one that is defined below. 

% The APP remains valid in the case where the cost function and the dynamics are subdifferentiable as long as the auxiliary cost and dynamics are differentiable. 

% % \todo[inline]{Le PPA fonctionne dans le cas $j$ et $\dyn$ sous-différentiable mais à condition que $K$ et $\auxdyn$ soient différentiables, faut-il le préciser ici, on n'a pour l'instant utilisé le PPA que dans le cas $j$, $\dyn$ différentiable? Dans notre cas $K$ et $\auxdyn$ ne sont que sous-différentiable}

\subsubsection{State variable relaxation}
\label{subsubsec:relax_state}

Let $i \in \components$ and $t \in \timesteps$. Recall from~\S\ref{subsec:gen_syst} that:
% % \todo[inline]{Abus de notation}
\begin{itemize}
    \item $\statecomptime{i}{t} = (\reg{i}{t}, \age{i}{t}, \spr{i}{t})$ takes values in $\left\{0,1\right\} \times \bbR \times (\{\sprnofail\} \cup \bbRp)^{D}$,
    \item $\sto{t}$ takes values in $\bbN$.
\end{itemize}
% The variables divide into two types. Some of them represent physical quantities that are intrinsically integers, these are being relaxed:
% \begin{itemize}
%     \item the regime of the component $\reg{i}{t}$, either healthy or broken
%     \item the number of available spare parts $\sto{t}$
% \end{itemize}
% On the other hand some variables represent a time or a duration. They are in essence continuous variables that are not affected by the relaxation. However they also have discrete values in practice due to time discretization for the numerical implementation but this is not an issue here. The variables in question are:
% \begin{itemize}
%     \item the age $\age{i}{t}$ of the component
%     \item the vector of $\spr{i}{t}$ times since last failures. 
% \end{itemize}
We relax the integrity constraint on $\reg{i}{t}$ and $\sto{t}$, so we allow:
\begin{itemize}
    \item $\statecomptime{i}{t} = (\reg{i}{t}, \age{i}{t}, \spr{i}{t})$ to take values in $[0,1] \times \bbR \times (\{\sprnofail\} \cup \bbRp)^{D}$,
    \item $\sto{t}$ to take values in $\bbR$.
\end{itemize} 

\begin{remark}
    We lose the physical interpretation of the relaxed variables.
    \begin{itemize}
        \item If $0 < \reg{i}{t} < 1$, we could think that component $i$ is in a degraded regime where the closer $\reg{i}{t}$ is to $1$ the healthier it is. This interpretation is however not exact as the health of a component is only characterized by its age $\age{i}{t}$. The probability of failure of a component is indeed only a function of $\age{i}{t}$.
        \item A value $\sto{t} \in \bbR$ means that there can be a non-integer number of parts in the stock. \finremark
    \end{itemize}
    \skipfinremark
\end{remark}

\subsubsection{Relaxation of the dynamics}
\label{subsubsec:relax_dyn}

The dynamics of the original system has been described in~\S\ref{subsec:syst_dyn} and an explicit expression is given in~Appendix~\ref{sec:rew_syst_dyn}. This expression involves indicator functions $\findi{\mathcal{A}}$ for some set $\mathcal{A}$. The dynamics is then non-continuous. Replacing the original indicator function $\findi{\mathcal{A}}$ with a continuous relaxed version 
% $\findi{\mathcal{A}}^{rel}$ 
allows to define a relaxed dynamics for the system.

% \begin{definition}
%     Let $p \in \bbN$, $\mathcal{A} \subset \bbR^{p}$ and $x \in \bbR^{p}$. The indicator function of the set $\mathcal{A}$ is
%     \begin{equation}
%         \indic{\mathcal{A}}{x} = 
%         \left\{
%         \begin{aligned}
%             1 &\text{ if } x \in \mathcal{A} \\
%             0 &\text{ if } x \notin \mathcal{A} \\
%         \end{aligned}
%         \right.
%     \end{equation}
% \end{definition}

% \begin{definition}
%     Let $p \in \bbN$, $\mathcal{A} \subset \bbR^{p}$ and $x \in \bbR^{p}$. The euclidean distance between $x$ and the set $\mathcal{A}$ is defined as
%     \begin{equation}
%         d(\mathcal{A}, x) = \min_{y \in \mathcal{A}} \norm{ x - y } 
%     \end{equation}
%     where $\norm{ \cdot }$ is the euclidean norm in $\bbR^{p}$. In particular, if $x \in \mathcal{A}$, then $d(\mathcal{A}, x) = 0$
% \end{definition}

\begin{definition}
  Let $\mathcal{A} \subset \bbR^{p}, \ x \in \bbR^{p}$ and $\alpha > 0$.
  We define a continuous relaxation $\findi{\mathcal{A}}\rel$ with parameter $\alpha$ of the indicator function $\findi{\mathcal{A}}$ as:
  \begin{equation}
    \findi{\mathcal{A}}\rel(x) =
    \left\{
      \begin{aligned}
        & 1 - 2\alpha \times \dis{x}{\mathcal{A}} &&\quad \text{if  } \dis{x}{\mathcal{A}} \leq \frac{1}{2\alpha} \eqsepv\\
        & 0 && \quad \text{if  } \dis{x}{\mathcal{A}} > \frac{1}{2\alpha}\eqsepv\\
      \end{aligned}
    \right.
    % \quad \text{for } x \in \bbR^{p}
  \end{equation}
  where $\dis{x}{\mathcal{A}}$ is the Euclidean distance between $x$ and the set $\mathcal{A}$.
  \label{def:relax_indic}
\end{definition}
The relaxation of the indicator function is illustrated on Fig.~\ref{fig:indic_relax}.
\begin{figure}[htbp]
    \includegraphics[width=0.49\textwidth]{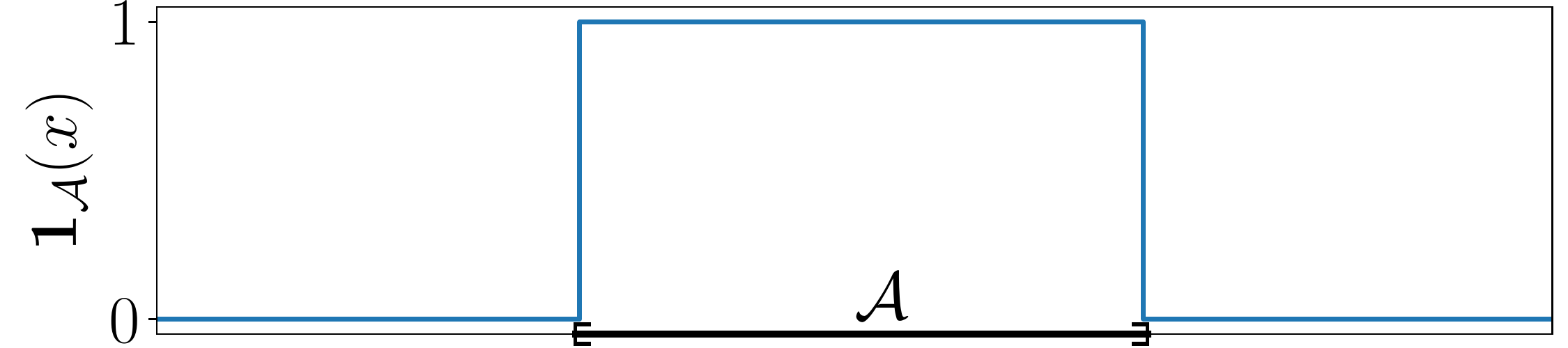}
    \hfill
    \includegraphics[width=0.49\textwidth]{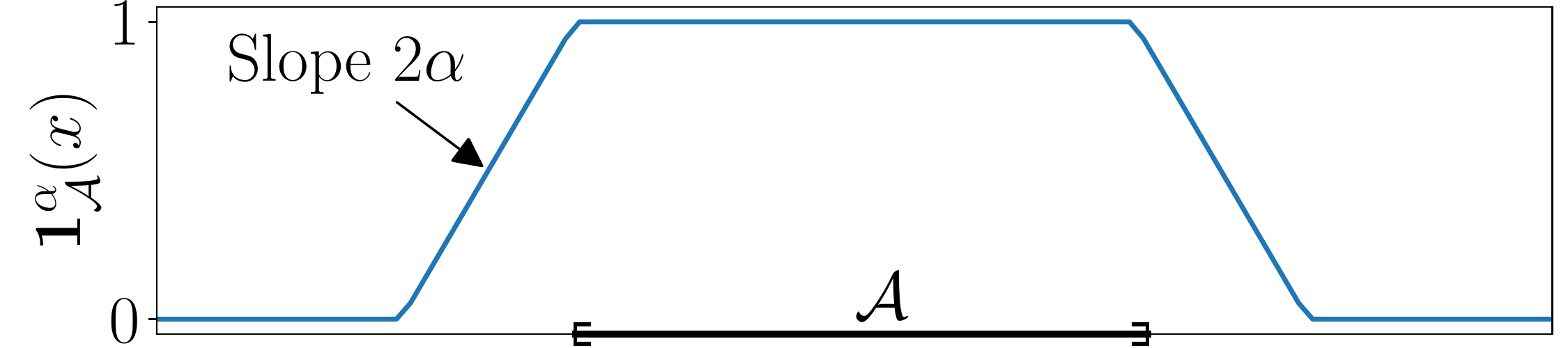}
    \caption{Illustration of the relaxation of the indicator function.}
    \label{fig:indic_relax}
\end{figure}

% % \todo[inline]{Problème d'alignement du 0}

The parameter $\alpha$ quantifies the stiffness of the relaxation. As $\alpha \to +\infty$, the relaxed indicator $\findi{\mathcal{A}}\rel$ converges pointwise towards the original indicator function $\findi{\mathcal{A}}$. As $\alpha \to 0$, the relaxed indicator $\findi{\mathcal{A}}\rel$ converges pointwise towards the constant function $1$. Note that for all $\alpha > 0$, the relaxed indicator $\findi{\mathcal{A}}\rel$ is continuous and differentiable almost everywhere.

A continuous relaxation of the system dynamics of parameter $\alpha$ is obtained by replacing the occurrences of the indicator function by its relaxed version. For instance, the relaxed dynamics of the stock is obtained from~\textup {(\ref{eq:stock})}:
\begin{equation}
  \sto{t+1} = \sto{t} + \sum_{i=1}^{n} \sum_{d=1}^{\delay} \indicrel{\na{\delay-1}}{\spr{i}{t}^{d}} -
  \min \Bp{\sto{t}, \ \sum_{i=1}^{n} \indicrel{\na{0}}{\reg{i}{t}}} \eqfinp
\end{equation}

Additional technical details and an explicit expression for the relaxed dynamics of the components are given in~Appendix~\ref{sec:rel_dyn}. The relaxed dynamics and relaxed auxiliary dynamics are denoted by $\dynrel$ and $\auxdynrel$ respectively.

\subsubsection{Cost relaxation}
\label{subsubsec:relax_cost}

The relaxation of the maintenance cost and forced outage cost is constructed using the same technique as for the dynamics.
Let $i \in \components$, $t \in \timesteps$ and $\alpha > 0$ be given.
\begin{itemize}
    \item The relaxed maintenance cost for component $i$ at time $t$ with parameter $\alpha$ is defined as:
    \begin{equation}
        \left\{
            \begin{aligned}
                \costmrel_{i,t}(\statecomptime{i}{t}, \ctrlcomptime{i}{t}) &= \disc{t} \cpm_{i}\ctrlcomptime{i}{t}^{2} + \disc{t} \ccm_{i}\indicrel{\na{0}}{\reg{i}{t}}\indicrel{\na{0}}{\age{i}{t}}, \quad t \in \timestepsnoend \eqfinv \\
                \costmrel_{i,T}(\statecomptime{i}{T}) &= \disc{T} \ccm_{i}\indicrel{\na{0}}{\reg{i}{T}}\indicrel{\na{0}}{\age{i}{T}} \eqfinp
            \end{aligned}
        \right.
        \label{eq:def_rel_maint_cost}
    \end{equation}
    \item The relaxed forced outage cost at time $t$ with parameter $\alpha$ is defined as:
    \begin{equation}
        \costforel_{t}(\statecomptime{1:n}{t}) = \disc{t}\cfo \min\left\{1,\ \sum_{i=1}^{n} \indicrel{\na{0}}{\reg{i}{t}}\indicrel{\bbRpe}{\age{i}{t}}\right\}
        \label{eq:def_rel_fo_cost}\eqfinp
    \end{equation}
\end{itemize}
Similarly, as in~\textup {(\ref{eq:costm_compact})} and~\textup {(\ref{eq:costfo_compact})}, we set:
\begin{equation}
    \costmrel_{i}(\statecomponent{i}, \ctrlcomp{i}) = \sum_{t=0}^{T-1} \costmrel_{i,t}(\statecomptime{i}{t}, \ctrlcomptime{i}{t}) + \costmrel_{i,T}(\statecomptime{i}{T})
    \quad \text{and} \quad 
    \costforel(\statecomponent{1:n}) =  \sum_{t=0}^{T} \costforel_{t}(\statecomptime{1:n}{t})\eqfinp
\end{equation}

We use the relaxed version of the dynamics $\dynrel$, of the auxiliary dynamics $\auxdynrel$, and of the costs $\costmrel_{i}$ and $\costforel$ in the auxiliary problem~\textup {(\ref{eq:aux_pb_indus})}. Hence, the objective function and the dynamics are continuous and differentiable almost everywhere with respect to $(\stateallcomp, \stoalltime, \ctrlallcomp)$. Note that the choice of the relaxation parameter $\alpha$ plays an important role. We can guess that choosing a low value for $\alpha$ leads to a problem that may be easier to solve numerically than with a larger value of $\alpha$. However, with a low value of $\alpha$ the relaxed dynamics and cost do not represent well the industrial problem whereas with a large $\alpha$ the original and relaxed dynamics are close. 
% Hence with a low $\alpha$ a good solution for the relaxed problem may not be good for the original problem.
More details on the influence of the parameter $\alpha$ and how it is chosen in practice are given in~\S\ref{subsec:param_tuning}.

\subsection{Explicit expression of the subproblems}
\label{subsec:decomp_ap}

% Hence the auxiliary problem that is solved at each iteration uses the relaxed version of the dynamics and the cost:
% \begin{equation}
%     \begin{aligned}
%         \min_{\optimspacerel}
%         &\begin{multlined}[t]
%             \espe \bigg(\sum_{i=1}^{n} \left(\costmrel_{i}(\statecomponent{i}, \ctrlcomp{i}) + \costforel(\statecompbar{1}, \ldots, \statecomponent{i}, \ldots, \statecompbar{n})\right)\\
%                 + \frac{\gamma_{x}}{2} \sqnorm{ X - \stateallcompbar } + \frac{\gamma_{s}}{2} \sqnorm{ \stoalltime - \stoalltimebar } + \frac{\gamma_{u}}{2} \sqnorm{\ctrlallcomp- \ctrlallcompbar }\\
%             + \proscal{ \multallcompbar}{ (\dyn^{\alpha'}(\prevopt, \noiseallcomp) - \auxdyn^{\alpha'}(\prevopt, \noiseallcomp))\cdot (\stateallcomp, \stoalltime, \ctrlallcomp) } \bigg)
%         \end{multlined}\\
%         \sucht &\auxdynrel(\stateallcomp, \stoalltime, \ctrlallcomp, \noiseallcomp) = 0
%     \end{aligned}
%     \label{eq:aux_pb_indus_rel}
% \end{equation}
% For the sake of readability, we drop the superscript $\alpha$ in the following.
By construction, the auxiliary problem~\textup {(\ref{eq:aux_pb_indus})} can be decomposed into $n$ independent subproblems on the components and a subproblem on the stock. In this section, we provide the explicit expression of these subproblems. The APP fixed-point algorithm will be applied on the relaxed system, so we use the relaxed dynamics and cost in the writing of the subproblems. Formally, the gradients of $\dynrel$ and $\auxdynrel$ are not defined everywhere. By abuse of notation, the subproblems are given as if $\dynrel$ and $\auxdynrel$ were differentiable. Appendix~\ref{sec:indic_rel} gives details on how we handle the points where the relaxed indicator is not differentiable. 

The subproblem on component $i \in \components$ is:
\begin{equation}
    \begin{aligned}
        \min_{(\statecomponent{i}, \ctrlcomp{i}) \in \primdecomp{i}}
        &\ \begin{multlined}[t]
            \espe \bigg( \costmrel_{i}(\statecomponent{i}, \ctrlcomp{i}) + \costforel(\statecompbar{1:i-1}, \statecomponent{i}, \statecompbar{i+1:n}) 
            \\+ \frac{\gamma_{x}}{2}\sqnorm{ \statecomponent{i} - \statecompbar{i}} + \frac{\gamma_{u}}{2}\sqnorm{ \ctrlcomp{i} - \ctrlcompbar{i}} \\
            + \proscal{ \multcompbar{\stoalltime}}{ \partial_{\statecomponent{i}} \dyncomprel{\stoalltime}(\statecompbar{1:n}, \stoalltimebar)\cdot \statecomponent{i} }\\
            + \sum_{j=i+1}^{n}\proscal{ \multcompbar{j}}{ \partial_{\statecomponent{i}} \dyncomprel{j}(\statecompbar{1:j}, \stoalltimebar, \ctrlcompbar{j}, \noisecomp{j})\cdot \statecomponent{i} } \bigg)
        \end{multlined}\\
        \sucht &\ \auxdyncomprel{i}(\statecomponent{i}, \ctrlcomp{i}, \noisecomp{i}) = 0\eqfinp
    \end{aligned}
    \label{eq:aux_pb_i}
\end{equation}

The subproblem on the stock is:
\begin{equation}
    \begin{aligned}
        \min_{\stoalltime \in \espacef{S}} \
        &\ \espe \left(\frac{\gamma_{s}}{2} \sqnorm{ \stoalltime - \stoalltimebar} + \sum_{i=1}^{n} \proscal{\multcompbar{i}}{ \partial_{\stoalltime} \dyncomprel{i}(\statecompbar{1:i}, \stoalltimebar, \ctrlcompbar{i}, \noisecomp{i}) \cdot \stoalltime } \right)\\
        \sucht &\ \auxdyncomprel{\stoalltime}(\stoalltime) = 0\eqfinp
    \end{aligned}
    \label{eq:aux_pb_sto}
\end{equation}

% Detailed calculations to retrieve the subproblems from the auxiliary problem~\textup {(\ref{eq:aux_pb_indus})} are given in~\cref{subsec:app_subpb_i,subsec:app_subpb_sto}.
% We managed to build a decomposable auxiliary problem for an industrial system of $n$ components and a stock of spare part. It can naturally be used in a  fixed-point algorithm of the kind of~Algorithm~\ref{alg:fix_pt_cst} to get a solution of the original maintenance problem. Some comments on the application of the fixed-point algorithm to the industrial case are made in the next part.

\subsection{The APP fixed-point algorithm for the industrial system}
\label{subsec:fix_pt_alg}

% % \todo[inline]{Ne pas oublier de dire que $\alpha$ et $\gamma$ varient à chq iter}

We can now solve the original maintenance optimization problem~\eqref{eq:optim_pb_no_time} using the fixed-point algorithm~\ref{alg:fix_pt_cst}. In this section, we give details on the practical implementation of the algorithm and present an efficient mixed parallel and sequential strategy for the resolution of the subproblems, that is tailored for the industrial optimization problem.
% \begin{algorithm}[htbp]
%     \begin{algorithmic}[1]
%         \myState{Start with $(\prevopt) = (\stateallcomp^{0}, \stoalltime^{0}, u^{0})$, $\multallcompbar=\multallcomp^{0}$}
%         \For{$k = 0, \ldots, K$}
%             \For{$i=1, \ldots, n$} 
%             \myState{Solve the subproblem~\textup {(\ref{eq:aux_pb_i})} on component $i$}
%             \myState{Let $(\statecomponent{i}^{k+1}, \ctrlcomp{i}^{k+1})$ be a solution and $\multcomp{i}^{k+1}$ be an optimal multiplier.}
%             \EndFor
%             \myState {Solve the subproblem~\textup {(\ref{eq:aux_pb_sto})} on the stock}
%             \myState{Let $\stoalltime^{k+1}$ be a solution and $\multcomp{\stoalltime}^{k+1}$ be an optimal multiplier.}
%             \myState{Set $(\prevopt) = ((\statecomponent{1}^{k+1}, \ldots, \statecomponent{n}^{k+1}), \stoalltime^{k+1}, (\ctrlcomp{1}^{k+1}, \ldots, \ctrlcomp{n}^{k+1}))$}
%             \myState{Set $\multallcompbar = (\multcomp{1}^{k+1}, \ldots, \multcomp{n}^{k+1}, \multcomp{\stoalltime}^{k+1})$}
%         \EndFor
%         \myState{\Return the maintenance strategy $u^{K}$ for the industrial system}
%     \end{algorithmic}
%     \caption{Fixed-point algorithm with a fully parallel strategy}
%     \label{alg:fix_pt_indus_par}
% \end{algorithm}

At each iteration, we solve the auxiliary problem~\eqref{eq:aux_pb_indus}. In a fully parallel version of the APP fixed-point algorithm, solving the auxiliary problem~\eqref{eq:aux_pb_indus} boils down to the parallel resolution of the $n+1$ independent subproblems defined by~\eqref{eq:aux_pb_i}--\eqref{eq:aux_pb_sto}.
\begin{enumerate}
    \item The subproblems on the components~\eqref{eq:aux_pb_i} are solved with the blackbox algorithm MADS \cite{audet_mesh_2006}. At iteration~$k$, we solve subproblem $i \in \components$ with:
    \begin{align}
        (\stateallcompbar, \stoalltimebar, \ctrlallcompbar) = (\stateallcomp^{k}, \stoalltime^{k}, \ctrlallcomp^{k}) \quad \text{and} \quad \multallcompbar = \multallcomp^{k} \eqfinp
    \end{align}
    MADS outputs a primal solution $(\statecomponent{i}^{k+1}, \ctrlcomp{i}^{k+1})$. The optimal multiplier $\multcomp{i}^{k+1}$ is computed afterwards using the adjoint state. The full derivation of the backward recursion for the multipliers in the industrial case is carried out in~Appendix~\ref{sec:opt_mult}.
    \item The subproblem on the stock~\eqref{eq:aux_pb_sto} is very easy to solve numerically. At iteration $k$ of the fully parallel APP fixed-point algorithm, we use:
    \begin{align}
        (\stateallcompbar, \stoalltimebar, \ctrlallcompbar) = (\stateallcomp^{k}, \stoalltime^{k}, \ctrlallcomp^{k}) \quad \text{and} \quad \multallcompbar = \multallcomp^{k} \eqfinp
    \end{align}
    The constraint $\auxdyncomprel{\stoalltime}(\stoalltime) = 0$ represents the dynamics of the stock with $\stateallcompbar = \stateallcomp^{k}$ being fixed. The value of $\stateallcompbar$ completely determines the dynamics of the stock. Hence, solving the subproblem on the stock just boils down to simulate its dynamics, we get a primal solution $\stoalltime^{k+1}$. The optimal multiplier $\multcomp{\stoalltime}^{k+1}$ is also computed using the adjoint state, see Appendix~\ref{sec:opt_mult}.
\end{enumerate}

% The parallel strategy for solving the subproblem can be opposed to a sequential strategy. In the sequential case each subproblem is solved one after the other. Hence the subproblem on component $i$ can use
% \begin{align}
%     (\statecompbar{1}, \ldots, \statecompbar{n}) &= (\statecomponent{1}^{k+1}, \ldots, \statecomponent{i-1}^{k+1}, \statecomponent{i}^{k}, \ldots, \statecomponent{n}^{k})\\
%     (\ctrlcompbar{1}, \ldots, \ctrlcompbar{n}) &= (\ctrlcomp{1}^{k+1}, \ldots, \ctrlcomp{i-1}^{k+1}, \ctrlcomp{i}^{k}, \ldots, \ctrlcomp{n}^{k})\\
%     (\multcompbar{1}, \ldots,\multcompbar{n}) &= (\multcomp{1}^{k+1}, \ldots, \multcomp{i-1}^{k+1}, \multcomp{i}^{k}, \ldots, \multcomp{n}^{k})
% \end{align}
% This means that the information obtained after the resolution of subproblem $i$ is immediately used for the resolution of a subproblem $j$ with $j>i$. \cite{cohen_coordination_1976} shows that the sequential version converges in fewer iterations than the parallel one. However with multiple processors a parallel strategy is more time efficient as all the subproblems are solved at the same time.

The features of the subproblem on the stock suggest to change the fully parallel strategy into a mixed parallel/sequential strategy. 
\begin{enumerate}
    \item At iteration $k$, the $n$ subproblems on the components are still solved in parallel using:
    \begin{align}
        (\stateallcompbar, \stoalltimebar, \ctrlallcompbar) = (\stateallcomp^{k}, \stoalltime^{k}, \ctrlallcomp^{k}) \quad \text{and} \quad \multallcompbar = \multallcomp^{k} \eqfinp
    \end{align}
    This yields a solution $(\statecomponent{i}^{k+1}, \ctrlcomp{i}^{k+1}, \multcomp{i}^{k+1})$ for each subproblem $i\in \components$.
    \item The difference arises for the subproblem on the stock. At iteration $k$, we immediately use the output of the subproblems on the components. This means that we set before solving the subproblem on the stock at iteration $k$:
    % \begin{align}
    %     (\statecompbar{1}, \ldots, \statecompbar{n}) &= (\statecomponent{1}^{l+1}, \ldots, \statecomponent{n}^{l+1})\\
    %     (\multcompbar{1}, \ldots,\multcompbar{n}) &= (\multcomp{1}^{l+1}, \ldots, \multcomp{n}^{l+1})
    % \end{align}
    \begin{align}
        % (\statecompbar{1}, \ldots, \statecompbar{n}) &= (\statecomponent{1}^{l+1}, \ldots, \statecomponent{n}^{l+1}) \eqfinv\\
        % (\ctrlcompbar{1}, \ldots, \ctrlcompbar{n}) &= (\ctrlcomp{1}^{l+1}, \ldots, \ctrlcomp{n}^{l+1}) \eqfinv\\
        % (\multcompbar{1}, \ldots,\multcompbar{n}) &= (\multcomp{1}^{l+1}, \ldots, \multcomp{n}^{l+1})\eqfinp
        (\stateallcompbar, \ctrlallcompbar) = (\stateallcomp^{k+1}, \ctrlallcomp^{k+1}) \quad \text{and} \quad (\multcompbar{1}, \ldots,\multcompbar{n}) &= (\multcomp{1}^{k+1}, \ldots, \multcomp{n}^{k+1})\eqfinp
    \end{align}
    This implies that the subproblem on the stock at iteration $k$ can be solved only after all subproblems on the components have been solved.
\end{enumerate}
With this strategy, we see experimentally that the number of iterations for convergence is reduced without penalizing the computation time per iteration as the subproblem on the stock can be solved in negligible time. This results in an overall speed up of the algorithm. The APP fixed-point algorithm with the mixed parallel/sequential strategy is presented  in~Algorithm~\ref{alg:fix_pt_indus_mix}. The termination criterion is a maximum number of iterations $M \in \bbN$. This is the version that is used for the numerical experiments.

\begin{algorithm}[htbp]
    \begin{algorithmic}[1]
        \myState{Start with $(\prevopt) = (\stateallcomp^{0}, \stoalltime^{0}, u^{0})$, $\multallcompbar=\multallcomp^{0}$}
        \For{$k = 0, \ldots, M-1$}
            \ParFor{$i \in \na{1, \ldots, n}$}
                \myState{Solve the subproblem~\eqref{eq:aux_pb_i} on component $i$.}
                \myState{Let $(\statecomponent{i}^{k+1}, \ctrlcomp{i}^{k+1})$ be a solution and $\multcomp{i}^{k+1}$ be an optimal multiplier.}
            \EndParFor
            \myState{Set $(\stateallcompbar, \ctrlallcompbar) = ((\statecomponent{1}^{k+1}, \ldots, \statecomponent{n}^{k+1}), (\ctrlcomp{1}^{k+1}, \ldots, \ctrlcomp{n}^{k+1}))$}
            \myState{Set $(\multcompbar{1}, \ldots,\multcompbar{n}) = (\multcomp{1}^{k+1}, \ldots, \multcomp{n}^{k+1})$}
            \myState {Solve the subproblem~\eqref{eq:aux_pb_sto} on the stock.}
            \myState{Let $\stoalltime^{k+1}$ be a solution and $\multcomp{\stoalltime}^{k+1}$ be an optimal multiplier.}
            \myState{Set $\stoalltimebar= \stoalltime^{k+1}$ and $\multcompbar{\stoalltime}= \multcomp{\stoalltime}^{k+1}$}
        \EndFor
        \myState{\Return the maintenance strategy $u^{M}$}
    \end{algorithmic}
    \caption{APP fixed-point algorithm with a mixed parallel/sequential strategy}
    \label{alg:fix_pt_indus_mix}
\end{algorithm}

% We are now almost able to implement the fixed-point algorithm for the industrial problem. The only point left to be explained is how to compute an optimal multiplier for the subproblems. This is done with the adjoint state equation.

%%% Local Variables:
%%% mode: latex
%%% TeX-master: "optim_maint_sched_v2"
%%% End:

%% file: numerical.tex
% !TEX root = optim_maint_sched_v2.tex

\section{Numerical results}
\label{sec:num_res}

In this part, we present the results of the decomposition methodology applied to Problem~\eqref{eq:optim_pb_no_time}. The expectation in~\eqref{eq:optim_pb_no_time} cannot be evaluated exactly, so we solve a Monte-Carlo approximation of the problem with $Q$ fixed failure scenarios $\omega_{1}, \ldots, \omega_{Q} \in \Omega$:
\begin{equation}
  \begin{aligned}
    \min_{\optimspace}\ & \frac{1}{Q}\sum_{q=1}^{Q}\left(\sum_{i=1}^{n} j_{i}\left(\statecomponent{i}(\omega_{q}\right), \ctrlcomp{i}) + \costfo\left(\statecomponent{1:n}(\omega_{q})\right)\right) \\
    \sucht &\dyn\left(\stateallcomp(\omega_{q}), \stoalltime(\omega_{q}), \ctrlallcomp, \noiseallcomp(\omega_{q})\right) = 0, \quad q \in \{1, \ldots, Q \}
    \eqfinp
  \end{aligned}
\label{eq:optim_pb_saa}
\end{equation}

The aim of this work is to show that the decomposition by prediction can efficiently solve a maintenance problem with a large number of components. We consider two test cases with the characteristics given in~Table~\ref{tab:syst_char}. As a maintenance strategy is parametrized by the vector $u \in \espacea{U}$ given in~\eqref{eq:adm_u}, the optimization problem for the $80$-component case involves $nT = 80 \times 40 = 3200$ variables.

The synthetic systems described in Table~\ref{tab:syst_char} do not exactly represent a real situation but are meant to be a proxy for the real cases presented at the beginning of \S\ref{subsec:gen_syst}:
\begin{enumerate}
    \item The situation where we consider turbines of the same production unit is simpler than the synthetic systems as it has fewer components (up to $10$ for the real case and $80$ in the synthetic case).
    \item On the other hand, the situation where we consider turbines across several production units is slightly more complicated than the synthetic systems: in our model all components belong to the same unit, therefore the forced outage cost is simpler than in reality. However, the decomposition methodology can be applied to this real case exactly as for the synthetic cases.
    \item For the case of a system with several families of components with a stock of spare parts for each family, the decomposition methodology is still applicable and will result in an additional subproblem for each stock. However, we have seen in~\S\ref{subsec:fix_pt_alg} that these subproblems are solved in negligible time as they just amount to the simulation of the dynamics of the corresponding stock. Hence, these additional subproblems do not slow down the method compared to the synthetic cases. The expression of the forced outage cost may also be slightly more complicated than for the synthetic cases.
\end{enumerate}
The costs and the failure distributions used in Table~\ref{tab:syst_char} are of the same order of magnitude as what is encountered in practice. In the first test case, the failure distributions are chosen so that the mean time to failure of the components is in the lower limit of the real cases, which induces a large number of failures in the system. This is the situation where the PM strategy has the most significant influence on the economic performance of the system. In the second test case, we consider components with a longer lifespan but the stock has fewer parts than in the first case. This induces a strong coupling between the stock and the components. These cases can be considered to be stress tests for the optimization algorithms.

% The reference algorithm is the blackbox algorithm MADS applied directly on the original optimization problem~\textup {(\ref{eq:optim_pb_saa})}. When running the optimization, the reference algorithm uses the original dynamics as it does not use gradient information. The APP fixed-point algorithm uses the relaxed dynamics. The maintenance strategies given by the two algorithms are then evaluated on a set of $10^{5}$ failure scenarios, distinct from those used for the optimization. For the two strategies, the evaluation is done with the original dynamics of the system in order to ensure a fair comparison.
% % MADS uses the original dynamics of the system as it does not need any gradient information to perform the optimization. 
% % % We run it with a budget of $8 \times 10^{5}$ evaluations of the objective function. 
% % It is then compared to the fixed-point algorithm where the relaxed system has to be used during the optimization.
% % % The number of fixed-point iterations is set to $50$. At each iteration, each subproblem is solved by MADS with a budget of $1000$ evaluations. 
% For the numerical experiments, we consider a system with the characteristics given in~Table~\ref{tab:syst_char}. The time to failure distributions of the components are Weibull distributions denoted by $\mathrm{Weib}(\beta, \lambda)$, where $\beta > 0$ is the shape parameter and $\lambda > 0$ is the scale parameter.
\begin{table}[htbp]
    \centering
    \begin{tabularx}{0.8\textwidth}{X c c}
        \toprule
        \textbf{Parameter} & \textbf{Case 1} & \textbf{Case 2}\\
        % \textbf{Technical characteristics}\\
        \midrule
        Number of components $n$ & $80$ & $80$\\
        Initial number of spare parts $\sto{0}$ & $\mathbf{16}$ & $\mathbf{5}$\\
        Horizon $T$ & $40$ years & $40$ years\\
        Time step $\Delta t$ & $1$ year & $1$ year\\
        Number of time steps for supply $D$ & $2$ & $2$\\
        % \midrule
        % \textbf{Economic characteristics}\\
        % \midrule
        Discount rate $\tau$ & $0.08$ & $0.08$\\
        Maintenance threshold $\nu$ & $0.9$ & $0.9$\\
        Yearly forced outage cost $\cfo$ & $10000$ k\euro / year & $10000$ k\euro / year\\
        Number of failure scenarios $Q$ & $\mathbf{100}$ & $\mathbf{300}$ \\
        \midrule
        \textbf{Components characteristics} \\
        \midrule
        PM cost $\cpm$ & $50$ k\euro & $50$ k\euro\\
        CM cost $\ccm$ & $200$ k\euro & $200$ k\euro\\
        Failure distribution & $\mathbf{\mathrm{\textbf{Weib}}(3, 10)}$ & $\mathbf{\mathrm{\textbf{Weib}}(3, 20)}$\\
        Mean time to failure & $\mathbf{8.93}$ \textbf{years} & $\mathbf{17.86}$ \textbf{years}\\ 
        \bottomrule
    \end{tabularx}
    \caption{Characteristics of the test cases. Bold data represent the features that differ between the two cases.}
    \label{tab:syst_char}
\end{table}

% \begin{remark}
%     It is possible to get a more accurate computation of the cost with respect to the real cost by taking a smaller step size. If we stick to a yearly maintenance decision this does not increase the complexity of the problem. However the evaluation of the cost function becomes more time consuming resulting in a higher computation time for the same number of iterations of the optimization algorithm.
% \end{remark}

\subsection{Parameter tuning}
\label{subsec:param_tuning}

Several parameters have to be tuned in order to apply the APP fixed-point algorithm. The parameters $\gamma_{x}, \gamma_{s}, \gamma_{u}$ appear in the auxiliary problem and $\alpha$ characterizes the relaxation of the system. 
These parameters are used in the APP fixed-point algorithm but not in the reference algorithm as the latter uses the original dynamics of the system to perform the optimization. 

As the maintenance threshold $\nu$ is used both in the reference algorithm and the decomposition method, it is fixed to a value giving good performance with the reference algorithm. The same value of $\nu$ is used in the APP fixed-point algorithm. We do not consider changing the value of $\nu$ as it would mean that the reference algorithm changes for each different value of $\nu$, making a fair comparison harder and the results less clear to analyze.

\subsubsection{Description of the parameters}

In the auxiliary problem~\textup {(\ref{eq:aux_pb_indus})}, the value of $\gamma = (\gamma_{x}, \gamma_{s}, \gamma_{u})$ influences the numerical behavior of the algorithm. We choose to increase the values of $\gamma_{x}, \gamma_{s}$ and $\gamma_{u}$ at each iteration of the APP fixed-point algorithm. The insight is that we can use low values of $\gamma$ in the first iterations to get close to a good solution. Then, we use high values of $\gamma$ to avoid oscillations of the solution of the auxiliary problem. Indeed, with a large value of $\gamma$, the solution $(\stateallcomp^{k+1}, \stoalltime^{k+1}, \ctrlallcomp^{k+1})$ of the auxiliary problem at iteration $k+1$ is close to the previous solution $(\stateallcomp^{k}, \stoalltime^{k}, \ctrlallcomp^{k})$. The value of $\gamma_{u}$ evolves from iteration $k$ to $k+1$ of the APP fixed-point algorithm with an additive step $\Delta \gamma > 0$, so that:
\begin{equation}
    \gamma^{k+1}_{u} = \gamma^{k}_{u} + \Delta \gamma \eqfinp
\end{equation}
Then, $\gamma_{x}$ and $\gamma_{s}$ are chosen to be proportional to $\gamma_{u}$ with ratios $r_{x} > 0$ and $r_{s} > 0$ respectively, so that:
\begin{align}
    \gamma_{x}^{k+1} = \gamma_{u}^{k+1}/r_{x} \ \text{ and } \ \gamma_{s}^{k+1} = \gamma_{u}^{k+1}/r_{s} \eqfinp
\end{align}
The motivation for this choice is that the vectors $\stateallcomp$, $\stoalltime$ and $\ctrlallcomp$ need some rescaling so that their norms are of the same order of magnitude. 
% Typically we have $\norm{\ctrlallcomp} \ll \norm{\stoalltime} \ll \norm{\stateallcomp}$. In the case where $\gamma_{x} = \gamma_{s} = \gamma_{u}$ the contribution of the terms $\gamma_{s} \sqnorm{\stoalltime^{k} - \stoalltime}$ and $\gamma_{u} \sqnorm{\ctrlallcomp^{k} - \ctrlallcomp}$ in the auxiliary cost would be negligible compared to $\gamma_{x} \sqnorm{\stateallcomp^{k} - \stateallcomp}$. The proximal term penalizes the distance between the previous iterate $(\stateallcomp^{k}, \stoalltime^{k}, \ctrlallcomp^{k})$ and the current evaluated point $(\stateallcomp, \stoalltime, \ctrlallcomp)$ and we see that only the distance between $\stateallcomp$ and $\stateallcomp^{k}$ would significantly influence the cost function. Numerically this could result in difficulties for the variable $\ctrlallcomp$ to converge during the fixed-point algorithm as its variation between two iterations is insignificantly penalized. We hope that choosing different proximal parameters $\gamma_{x}, \gamma_{s}$ and $\gamma_{u}$ can help the numerical convergence of the algorithm.
The parameters $\gamma_{u}^{0}, r_{x}, r_{s}, \Delta \gamma$ have to be tuned.

The other parameter that requires attention is the relaxation parameter $\alpha$. Similarly as for~$\gamma$, we choose to increase the value of $\alpha$ at each iteration. A low value of $\alpha$ makes the problem easier to solve numerically but does not represent well the real problem. As $\alpha$ increases, the relaxed problem is closer and closer to the real one but becomes harder to solve. To ease the resolution, we use the solution of the auxiliary problem at iteration $k$ as a warm start in MADS for iteration $k+1$. The value of $\alpha$ varies from iteration $k$ to $k+1$ of the APP fixed-point algorithm with a step $\Delta \alpha > 0$ so that:
\begin{equation}
    \alpha^{k+1} = \alpha^{k} + \Delta \alpha \eqfinp
\end{equation}
The values of $\alpha^{0}$ and $\Delta \alpha$ have to be tuned.
We denote by:
\begin{equation}
    p = (\gamma_{u}^{0}, r_{x}, r_{s}, \Delta \gamma, \alpha^{0}, \Delta \alpha) \in \bbR^{6} \eqfinv
\end{equation}
the vector of parameters that have to be adjusted for the algorithm.

\subsubsection{Tuning methodology}

Choosing a good value for $p$ is difficult in practice. We cannot afford to test the fixed-point algorithm with many different values of $p$ directly on the $80$-component case, as one run of the optimization takes more than a day. In order to find an appropriate value for $p$, we rely on a smaller system than the $80$-component case. This \emph{small system} is designed so that the runs of the fixed-point algorithm take a reasonable amount of time, $4$ hours in our case. The small system consists of $10$ components with $2$ spare parts initially. All the other characteristics of this small system are the same as Case 1 in Table~\ref{tab:syst_char}. In a way, the $10$-component system can be seen as a downscaling of the $80$-component case of interest. 

The idea of the tuning procedure is to run the decomposition method on the $10$-component system several times, but with a different value for $p$ at each run. To do so, we start by defining bounds on the value of the parameters, they are given in~Table~\ref{tab:distrib_param}. These bounds are chosen to be wide in order to ensure that they contain good values of the parameters.
\begin{table}[htbp]
    \centering
    \begin{tabular}{ccccccc}
        \toprule
        Parameter & $\gamma_{u}^{0}$ & $r_{x}$ & $r_{s}$ & $\Delta\gamma$ & $\alpha^{0}$ & $\Delta\alpha$ \\
        \midrule
        Bounds & $[1, 100]$ & $[1, 10^{4}]$ & $[1, 10^{3}]$ & $[0, 100]$ & $[2, 200]$ & $[0, 200]$\\
        \bottomrule
    \end{tabular}
    \caption{Bounds for the parameters of the of the fixed-point algorithm.
    }
    \label{tab:distrib_param}
\end{table}
% In order to assess the effect of the choice of $p$ on the maintenance optimization, we run the fixed-point algorithm $200$ times with different values of $p$.
% The $200$ values of $p$ are sampled with an optimized Latin Hypercube Sampling~\cite{damblin_numerical_2013} using the distributions given in~Table~\ref{tab:distrib_param}.
Then, $200$ samples of $p$ are drawn with an optimized Latin Hypercube Sampling~\cite{damblin_numerical_2013}. We run the fixed-point algorithm with each of the sampled values on the $10$-component system.
% The results of this computation are not shown because of space limitation. 
The value of $p$ which gives the best result on the small system is used for the case with $80$ components. The chosen value is:
\begin{equation}
    p = (17.32,\ 7434,\ 815.3,\ 1.360 \times 10^{-1},\ 46.51,\ 135.5)\eqfinp
    \label{eq:best_param}
\end{equation}
We give the value of $p$ with $4$ significant digits to emphasize that the performance of the APP fixed-point algorithm is very sensitive to this value. A sensitivity analysis has been performed using a Morris method~\cite{campolongo_effective_2007} but no clear pattern for a good choice of $p$ has been identified. More details on the methodology and results of the sensitivity analysis can be found in~\cite[\S10.7]{bittar_stochastic_2021-1}.

\subsection{Comparison of the decomposition method with MADS}

Now that all the parameters of the fixed-point algorithm~\ref{alg:fix_pt_indus_mix} are set, we can solve Problem~\eqref{eq:optim_pb_saa}. The performance of the fixed-point algorithm is compared with a reference algorithm, which is the blackbox algorithm MADS applied directly on Problem~\eqref{eq:optim_pb_saa}. We consider the systems of $80$ components described in~Table~\ref{tab:syst_char}. Parameters of the computation are given in~Table~\ref{tab:comp_set}. When running the optimization, the reference algorithm uses the original dynamics as it does not use gradient information. The APP fixed-point algorithm runs with the value of $p$ in~\eqref{eq:best_param} that parametrizes the auxiliary problem and the relaxation of the system.

\begin{table}[htbp]
    \begin{tabularx}{\textwidth}{c>{\centering\arraybackslash}X>{\centering\arraybackslash}X}
        & Decomposition & MADS \\
        \toprule
        Fixed-point iterations & $50$ & /\\
        Cost function calls & $10^{3}$/subproblem/iteration & $10^{6}$ \\
        Cost and dynamics & Relaxed & Original\\
        \midrule
        Processor model & \multicolumn{2}{c}{Intel\textregistered \ Xeon\textregistered \ Processor E5-2680 v4, $2.4$ GHz} \\
        % Number of cores & $80$ (using $3$ processors) & $1$\\
        % Computation time & $18$h$24$min & $22$h$30$min\\
        \bottomrule
    \end{tabularx}
    \caption{Parameters of the computation for the two algorithms.}
    \label{tab:comp_set}
\end{table}

\begin{remark}
    The APP fixed-point algorithm solves a decomposable auxiliary problem at each iteration, this algorithm is designed to be parallelized. It runs on $80$ processors so that the subproblems on the components are solved in parallel. The reference algorithm MADS runs only on one processor. Note that it is also possible to parallelize MADS~\cite{audet_parallel_2008}, although the implementation is not as straightforward as for the decomposition method. The parallel version of MADS has not been tested.  
\end{remark}

The maintenance strategies returned by the two algorithms are then evaluated on a common set of $10^{5}$ failure scenarios, distinct from those used for the optimization. For the two strategies, the evaluation is done with the original dynamics of the system in order to ensure a fair comparison. The two algorithms return a maintenance strategy with $\ctrlcomptime{i}{t} \in [0,1]$ for $(i,t) \in \components \times \timestepsnoend$. From the operational perspective, PMs make the components as good as new. Hence, for the evaluation of the strategy, the controls are projected on $\{0,1\}$: we consider that if $\ctrlcomptime{i}{t} \geq \nu$, then the PM makes the component as good as new, otherwise no PM is performed. The comparison between the two maintenance strategies is fair as we use the same procedure for their evaluation.

\subsubsection{Results for Case 1}

The mean cost is $14509$ k\euro \ with MADS and $12855$ k\euro \ with the decomposition which represents a gain of $12\%$. The values of some quantiles are gathered in~Table~\ref{tab:quant} and the distribution of the cost is represented on~Fig.~\ref{fig:comp_distrib_cost}.
\begin{table}[htbp]
    \begin{tabularx}{\textwidth}{Xccccccc}
        & $1\%$ & $5\%$ & $25\%$ & $50\%$ & $75\%$ & $95\%$ & $99\%$\\
        \toprule
        Decomposition & $11641$ & $11981$ & $12469$ & $12838$ & $13221$ & $13788$ & $14182$\\
        MADS & $13339$ & $13663$ & $14138$ & $14494$ & $14846$ & $15406$ & $15817$\\
       \bottomrule 
    \end{tabularx}
    \caption{Quantiles of the cost of the two maintenance strategies (k\euro).}
    \label{tab:quant}
\end{table}
\begin{figure}[htbp]
    \centering
    \begin{minipage}[t]{.49\textwidth}
        \includegraphics[width=\textwidth]{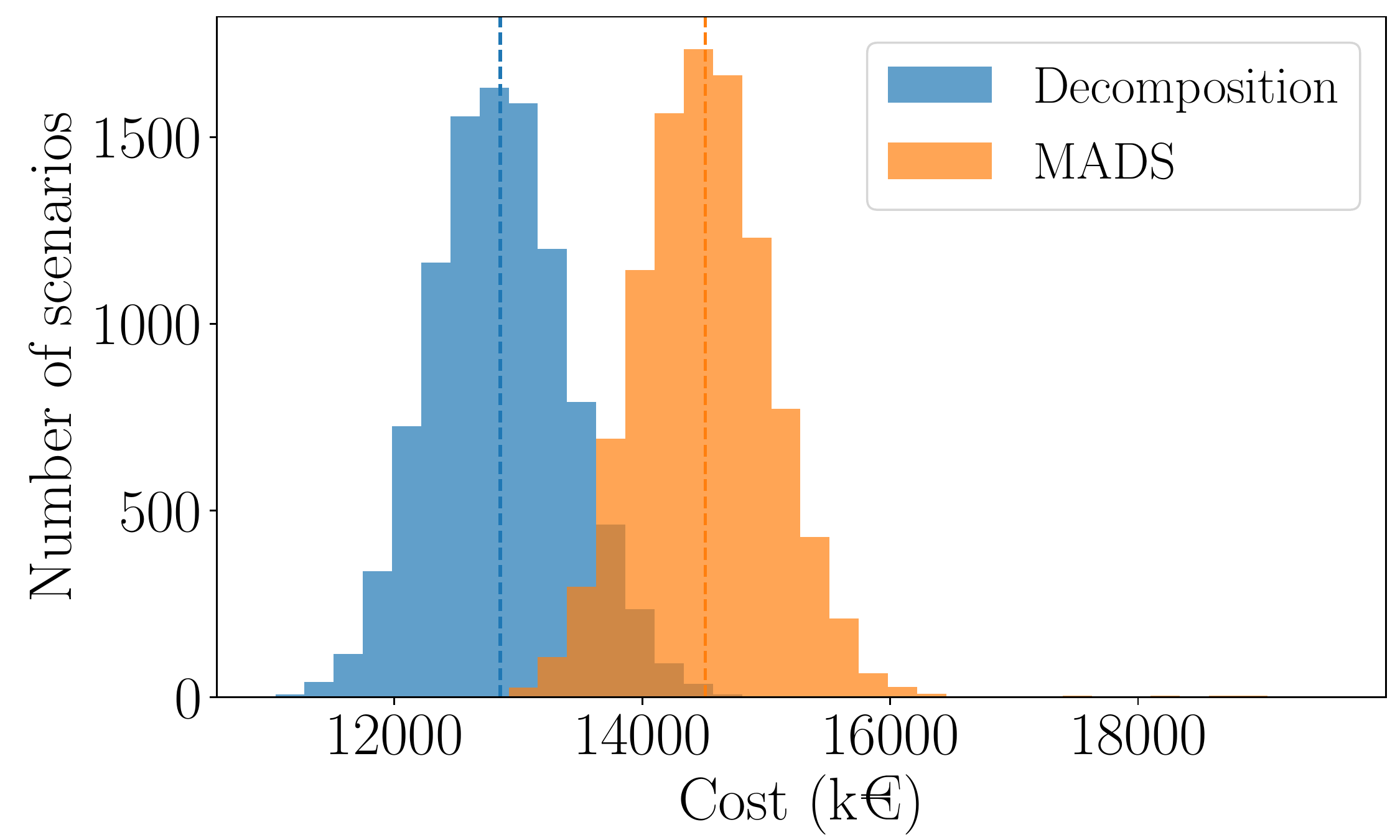}
        \caption{Distribution of the cost for the two maintenance strategies. The dashed lines represent the expected cost for both strategies.}
        \label{fig:comp_distrib_cost}
    \end{minipage}
    \hfill
    \begin{minipage}[t]{.49\textwidth}
        \includegraphics[width=\textwidth]{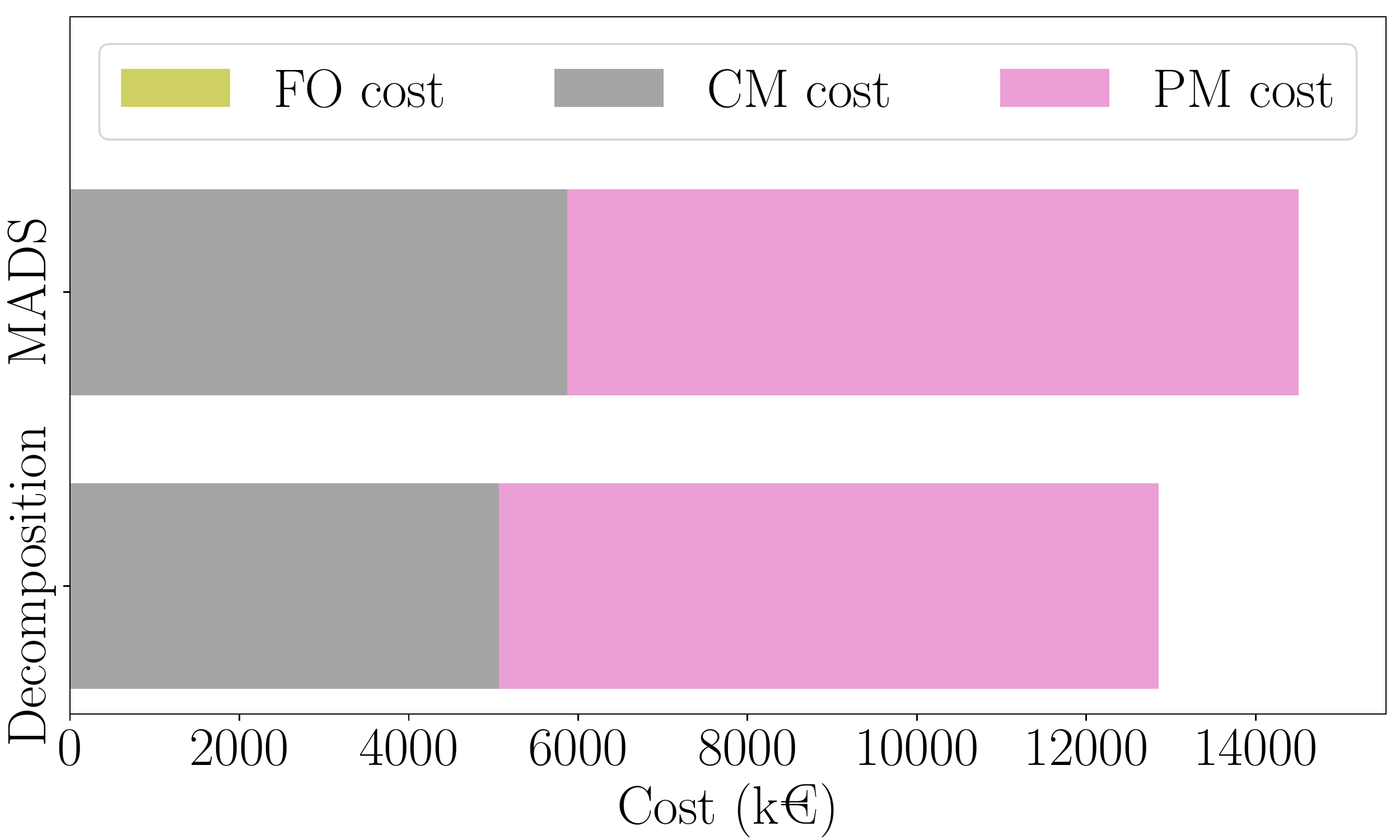}
        \caption{Part of the PM, CM and forced outage cost in the total expected cost.}
        \label{fig:repart_cost}
    \end{minipage}
\end{figure}
Figure~\ref{fig:repart_cost} outlines that both the CM cost and the PM cost are lower with the decomposition strategy. In fact the decomposition strategy performs more PMs than MADS strategy (Table~\ref{tab:nb_pm}) but these PMs occur later in time (Fig.~\ref{fig:pm_cumul}). Because of the discount factor, the decomposition strategy is cheaper even if more PMs are performed. 
% The counterpart is that failures and forced outages occur more often with the decomposition strategy (Table~\ref{tab:nb_pm}). 
The forced outage cost is not visible on~Fig.~\ref{fig:repart_cost} as it represents $8.7$~k\euro \ for MADS strategy and $2.7$~k\euro \ for the decomposition strategy: there are very few scenarios where a forced outage occurs, again with an advantage for the decomposition strategy ($6$ occurrences in $10^{5}$ failure scenarios for the decomposition versus $36$ for MADS). 
% but they almost all occur in the last two time steps of the study horizon. Therefore the cost of forced outages is low because of the discount factor.
Overall, the decomposition strategy performs PMs more effectively than MADS, both in terms of cost and in terms of number of failures and forced outages of the system.

Table~\ref{tab:nb_pm} also provides the computation time for both methods. As we design a long term maintenance schedule, the focus is on the performance of the maintenance strategy and not on the running time of the methods that are acceptable for our purpose. The parameters of the algorithms have been tuned to allow similar numbers of function calls for both methods. The decomposition is slower as it uses the relaxed dynamics which is slower to compute than the original dynamics used by MADS.
\begin{table}[htbp]
    \begin{tabularx}{\textwidth}{l>{\centering\arraybackslash}c>{\centering\arraybackslash}X}
        & Decomposition & MADS\\
        \toprule
        Total number of PMs & $669$ & $595$\\
        % Number of PMs for component $1$ & $1$ & $6$\\
        % Number of PMs for component $2$ & $2$ & $5$\\
        Mean number of PMs / component & $8.4$ & $7.4$\\
        % Mean time between PMs & $4.3$ years & $4.8$ years\\
        Mean number of failures / component & $1.18$ & $1.48$\\
        Number of forced outages / Number of scenarios  & $6/10000$ & $36/10000$\\
        % Computation time & $48$h$53$min & $31$h$37$min\\
        Computation time & $2$ days, $1$ hour & $1$ day, $7$ hours\\
        \bottomrule
    \end{tabularx}
    \caption{Number of PMs, failures and forced outages, and computation time for each strategy.}
    \label{tab:nb_pm}
\end{table}

The cumulative number of PMs can be visualized on~Fig.~\ref{fig:pm_cumul}. 
% As already noticed there are fewer PMs with the decomposition strategy. 
A striking feature with the decomposition strategy is that there are almost no PM in the first three years. This exploits the fact that the components are new. The reference algorithm MADS applied directly on the original problem does not detect this feature. In fact, the region of the space corresponding to not doing any PM in first three years jointly for all components is a very small subset of the admissible space of the original problem and is not explored by MADS. On the other hand, the subproblems in the APP fixed-point algorithm act on an individual component, it is then easier to figure out that doing no PM in the first three years is profitable.

On the second half of the period of study, the decomposition strategy performs more PMs than MADS. Because of the discount factor, these PMs are cheaper than PMs performed in the first years. Thus, the decomposition strategy manages to avoid more failures than MADS strategy in the second half of the horizon. We see that the decomposition strategy has a time-dependent behavior which is not the case for MADS strategy that always performs PMs at the same rate. The design of an effective time-dependent strategy is only possible with the decomposition as it solves subproblems of dimension $40$ on individual components. The algorithm MADS applied on the whole problem cannot detect these features as the search space of dimension $3200$ is too large.

% There is also a significant reduction of the number of PMs in the last five years of the study horizon. It is indeed useless to invest money to repair a component for the last few years, as the payback period could exceed the remaining life of the plant. Moreover, the discount factor at the end of the horizon greatly reduces the incurred cost so that a forced outage is not too penalizing. This is why some forced outages occur with the decomposition strategy at the end of the study period.

\begin{figure}[htbp]
    \centering
    \begin{minipage}[t]{.49\textwidth}
        \includegraphics[width = \textwidth]{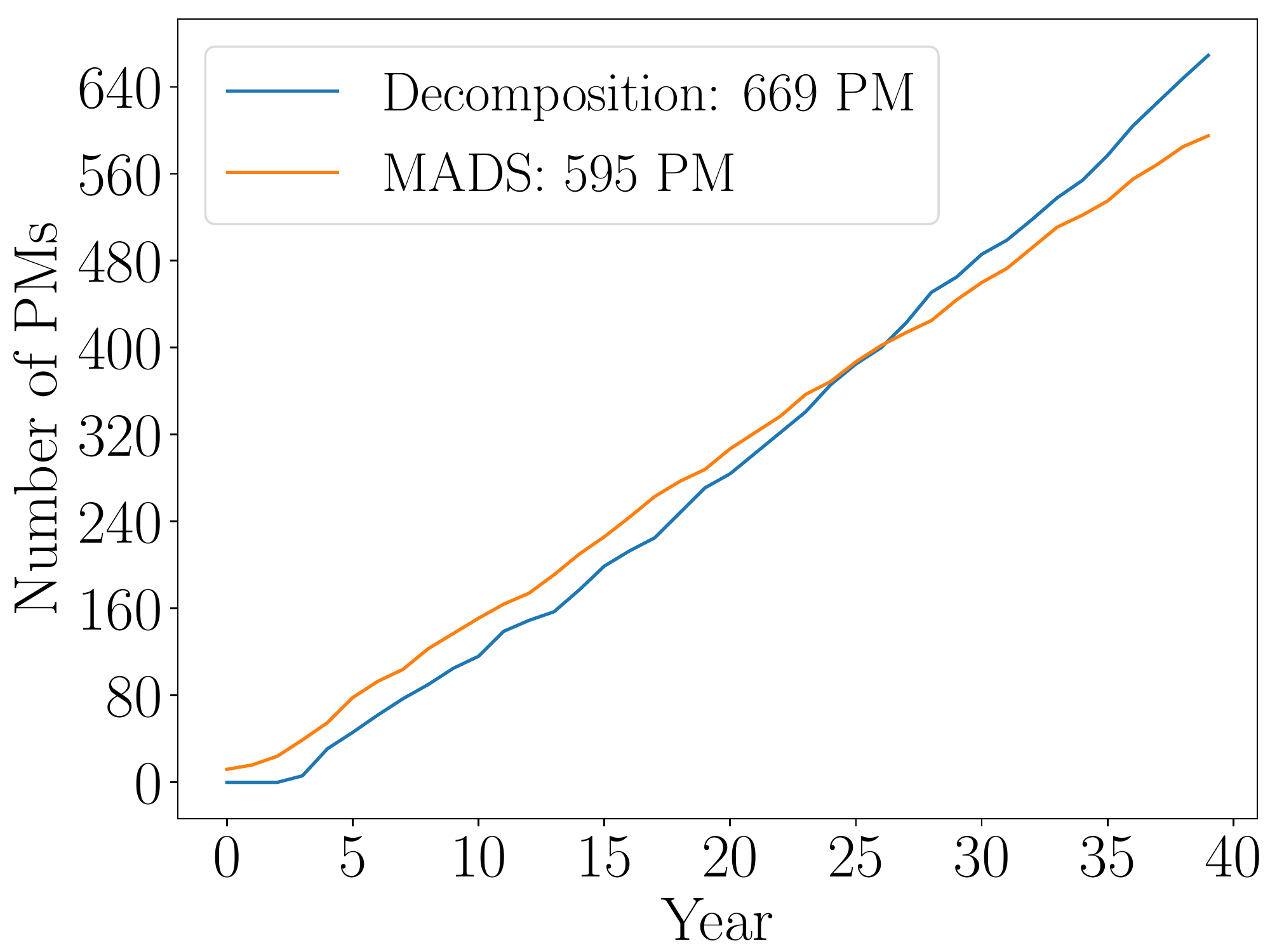}
        \caption{Cumulative number of PMs.}
        \label{fig:pm_cumul}
    \end{minipage}
    \hfill
    \begin{minipage}[t]{.49\textwidth}
        \includegraphics[width=\textwidth]{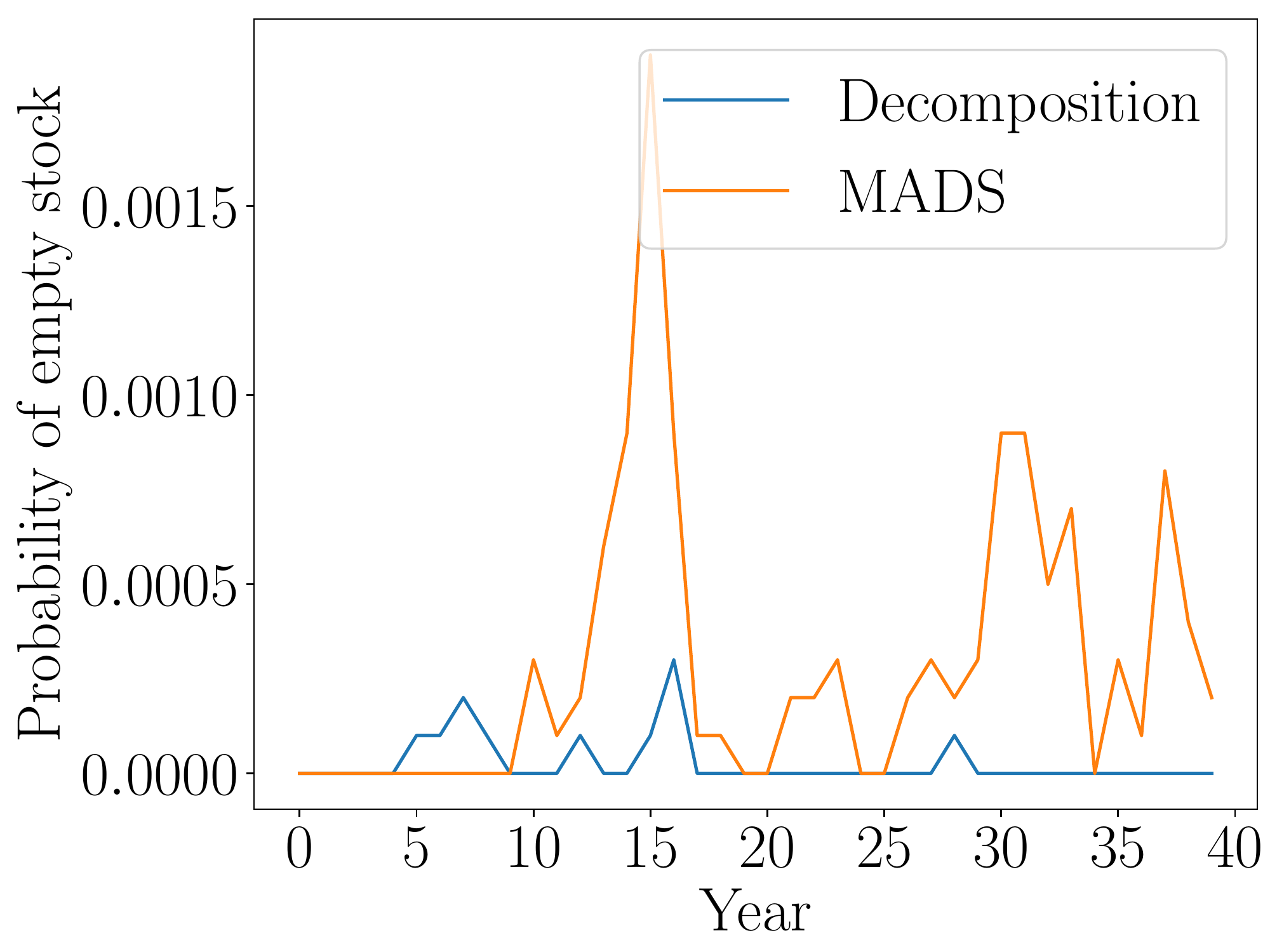}
        \caption{Evolution of the probability of having an empty stock.}
        \label{fig:low_sto}
    \end{minipage}
\end{figure}

% Note that with both strategies the mean time between PMs in Table~\ref{tab:nb_pm} is lower that the mean time to failure given in~Table~\ref{tab:syst_char}. This is also true if we look at the number of PMs for component $1$, $2$ and $i \geq 3$ separately as they do not have the same failure distribution. 
Another indicator that is monitored by decision makers is the level of stock. A necessary condition for the occurrence of a forced outage is that the stock is empty. Hence, we look at the probability of having an empty stock. The higher this probability, the higher the probability of forced outage.
% The mean levels of stock (\cref{fig:mean_sto}) are similar for both strategies in the first $20$ years and is then lower for the decomposition strategy. This is a consequence of the higher number of failures for the decomposition strategy. 
The probability of having an empty stock is very low on the whole horizon for the decomposition strategy (Fig.~\ref{fig:low_sto}). For MADS strategy this probability is often higher in particular around year $15$ and after year $30$. This feature again supports that the the decomposition strategy is more robust than MADS strategy.

% Again, because of the discount factor, forced outages in the last few years do not have heavy financial consequences. It is then more profitable to do fewer PMs and allow for a higher risk of failure. This is what the decomposition strategy does.

% \begin{figure}[htbp]
%     \centering
%     \begin{minipage}[t]{.49\textwidth}
%         \includegraphics[width = \textwidth]{pm_year_no_tit}
%         \caption{Number of PMs per year}
%         \label{fig:pm_year}
%     \end{minipage}
%     %
%     \hfill
%     \begin{minipage}[t]{.49\textwidth}
%         \includegraphics[width = \textwidth]{pm_cumul_no_tit}
%         \caption{Cumulative number of PMs}
%         \label{fig:pm_cumul}
%     \end{minipage}
% \end{figure}

% \begin{figure}[htbp]
%     \centering
%     \begin{minipage}[c]{.49\textwidth}
%         \includegraphics[width=\textwidth]{stock_mean_10000_no_tit}
%         \caption{Mean level of stock}
%         \label{fig:mean_sto}
%     \end{minipage}%
%     \hfill
%     \begin{minipage}[c]{.49\textwidth}
%         \includegraphics[width=\textwidth]{stock_proba_low_10000_no_tit}
%         \caption{Evolution of the probability of having an empty stock}
%         \label{fig:low_sto}
%     \end{minipage}
% \end{figure}

Overall on this case, the strategy obtained by decomposition is more cost effective than MADS strategy. 
% For a decision maker the decomposition strategy requires less investment as we do fewer PMs. It also has the best expected cost.
In the case of extreme events, the decomposition strategy is more robust than MADS strategy, as shown by the $99\%$ quantile in~Table~\ref{tab:quant}, the number of failures and of forced outages in Table~\ref{tab:nb_pm} and the probability of having an empty stock in Fig.~\ref{fig:low_sto}.
% Indeed the forced outages may occur only at the end of the horizon. 

\subsubsection{Results for Case 2}

In the second case presented in Table~\ref{tab:syst_char}, the components have a longer lifespan but the stock has fewer parts than in the first case. The coupling between the components and the stock is then more important than in Case 1. The mean cost of the maintenance strategy returned by MADS is $10815$ k\euro \ and it is $9749$ k\euro \ for the strategy returned by the decomposition, which represents a gain of $10\%$. The values of some quantiles are gathered in~Table~\ref{tab:quant2} and the distribution of the cost is represented on~Fig.~\ref{fig:comp_distrib_cost2}. The situation is quite different from the first case as the decomposition strategy has a more important variance than MADS strategy and presents several modes.
% \rouge{Corresponding to the nb of forced outages in a given scenario ?} 
\begin{table}[htbp]
    \begin{tabularx}{\textwidth}{Xccccccc}
        & $1\%$ & $5\%$ & $25\%$ & $50\%$ & $75\%$ & $95\%$ & $99\%$\\
        \toprule
        Decomposition & $5381$ & $5601$ & $5999$ & $9054$ & $13014$ & $18421$ & $23069$\\
        MADS & $10221$ & $10357$ & $10568$ & $10733$ & $10916$ & $11266$ & $13650$\\
       \bottomrule 
    \end{tabularx}
    \caption{Quantiles of the cost of the two maintenance strategies (k\euro).}
    \label{tab:quant2}
\end{table}
\begin{figure}[htbp]
    \centering
    \begin{minipage}[t]{.49\textwidth}
        \includegraphics[width=\textwidth]{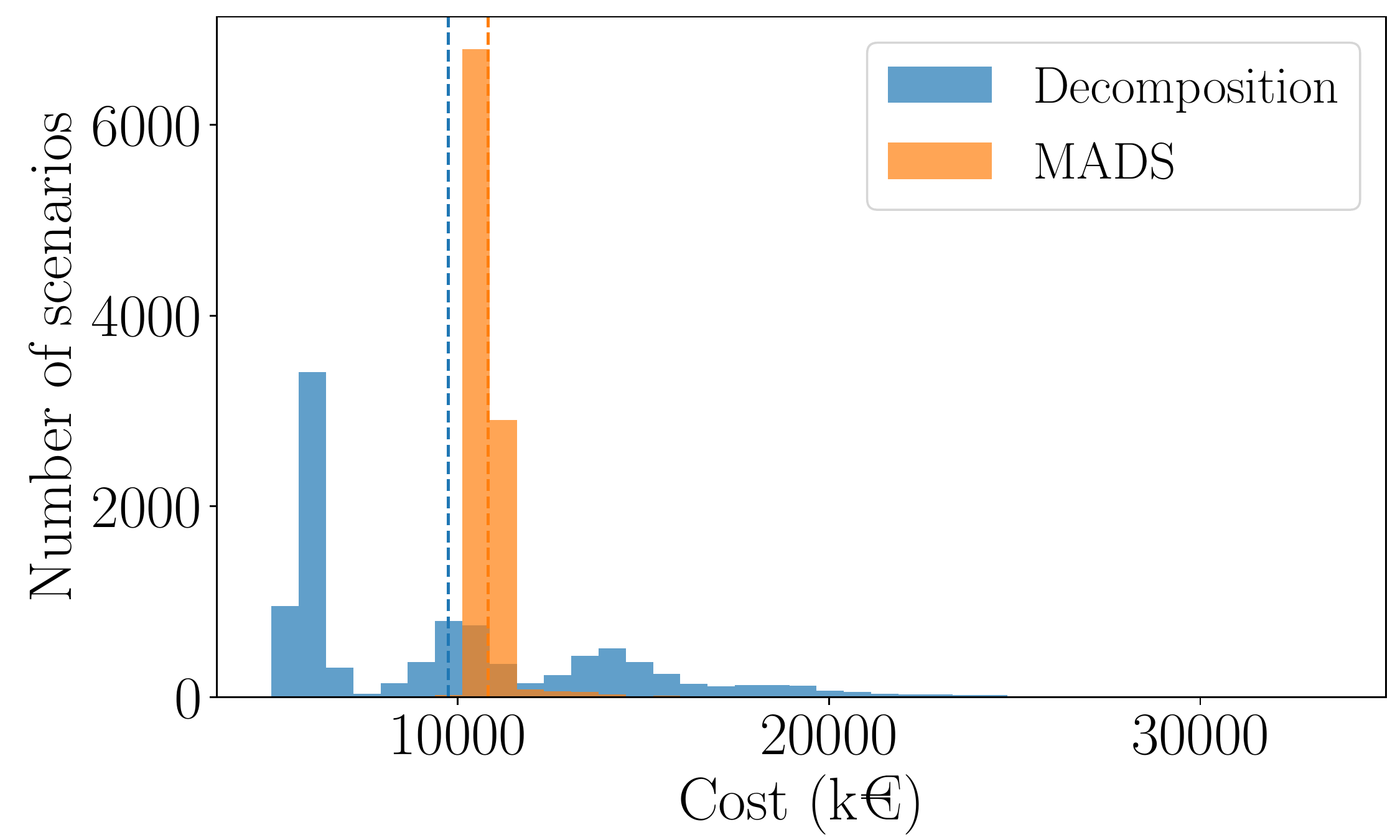}
        \caption{Distribution of the cost for the two maintenance strategies. The dashed lines represent the expected cost for both strategies.}
        \label{fig:comp_distrib_cost2}
    \end{minipage}
    \hfill
    \begin{minipage}[t]{.49\textwidth}
        \includegraphics[width=\textwidth]{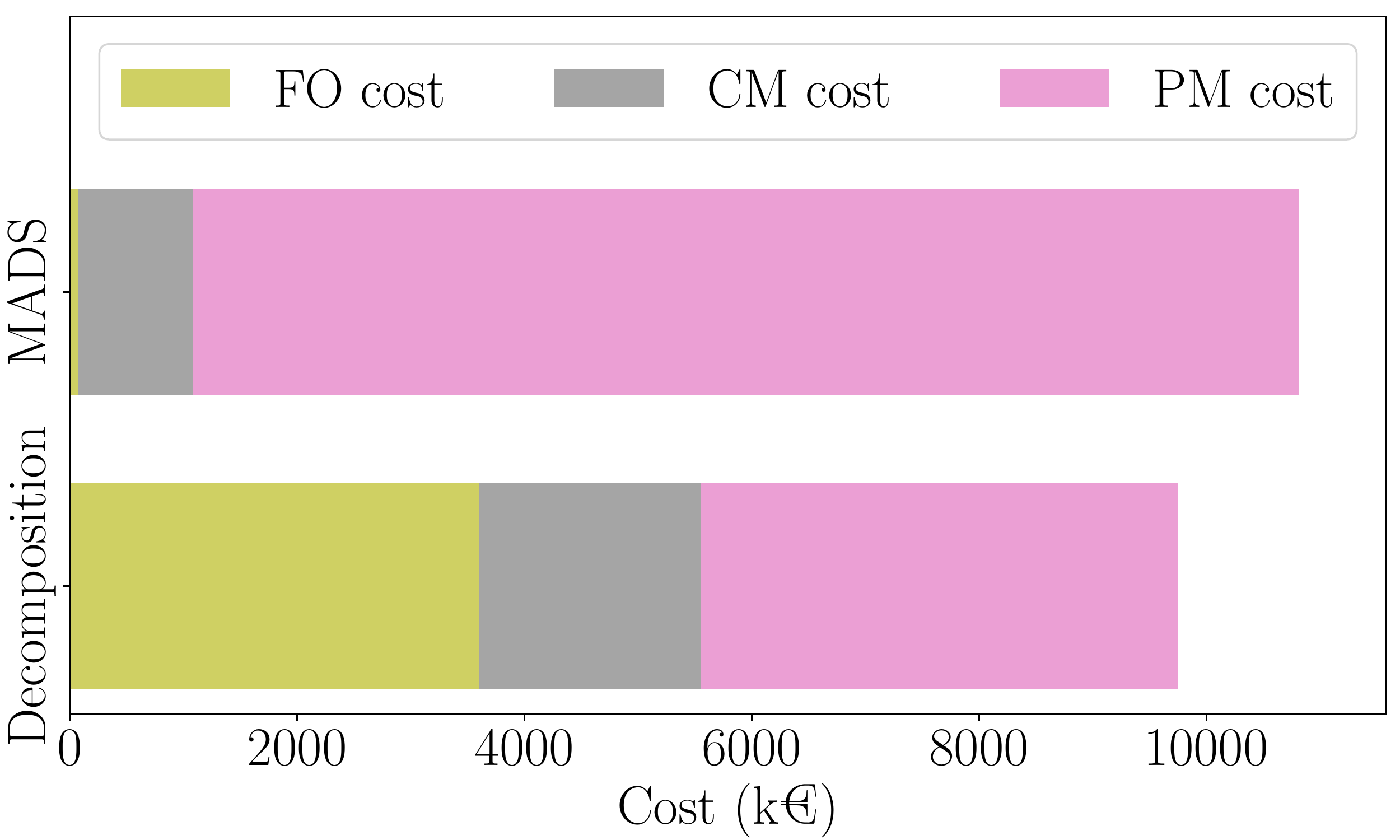}
        \caption{Part of the PM, CM and forced outage cost in the total expected cost.}
        \label{fig:repart_cost2}
    \end{minipage}
\end{figure}

The repartition of the cost in Fig.~\ref{fig:repart_cost2} confirms that the two maintenance strategies have a different approach. The strategy returned by MADS is similar than in the first case with a high number of PMs, few failures and forced outages, see Table~\ref{tab:nb_pm2}. On the other hand, with the decomposition strategy the PM cost is much lower, as fewer PMs are performed (Table~\ref{tab:nb_pm2}). The counterpart is a higher CM cost and a notable contribution of the forced outage cost. We see in Table~\ref{tab:nb_pm2} that the number of forced outages is indeed significantly higher with the decomposition strategy, with almost one forced outage per scenario in average. Thus, the approaches of the two maintenance strategies are completely different: the MADS strategy is conservative with a lot of PMs (similarly as in the first case) whereas the decomposition strategy is more risky with fewer PMs and more forced outages.

\begin{table}[htbp]
    \begin{tabularx}{\textwidth}{l>{\centering\arraybackslash}c>{\centering\arraybackslash}X}
        & Decomposition & MADS\\
        \toprule
        Total number of PMs & $503$ & $644$\\
        Mean number of PMs / component & $6.3$ & $8.1$\\
        Mean number of failures / component & $0.38$ & $0.26$\\
        Number of forced outages / scenario  & $0.94$ & $0.04$\\
        % Computation time & $132$h$53$min & $84$h$05$min\\
        Computation time & $5$ days, $12$ hours & $3$ days, $12$ hours\\
        \bottomrule
    \end{tabularx}
    \caption{Number of PMs, failures and forced outages, and computation time for each strategy.}
    \label{tab:nb_pm2}
\end{table}

The cumulative number of PMs can be visualized on~Fig.~\ref{fig:pm_cumul2}. 
Similarly as in the first case, we see a time-dependent behavior of the decomposition strategy with almost no maintenance in the first seven years. In Case 1, the decomposition strategy does not perform PMs in the first three years only but the components have a shorter average lifespan (around $9$ years in Case 1 and $18$ years in Case 2). In the first half of the time horizon, the decomposition strategy performs fewer PMs than MADS strategy, whereas in the second half of the horizon the decomposition performs slightly more PMs than MADS. We see in Fig.~\ref{fig:low_sto2}, that the decomposition strategy results in a high probability of having an empty stock between years $10$ and $15$, where most of the forced outages occur. In the first years and in the second half of the horizon, both strategies present a low risk of having an empty stock.

In Table~\ref{tab:nb_pm2}, we see that the computation time are around 3 times larger than in Case 1. This is due to the fact that we have used $Q = 300$ scenarios here against $Q = 100$ scenarios in Case 1 for the estimation of the expected cost (see Table~\ref{tab:syst_char}). As the decomposition designs a risky strategy, if $Q$ is kept too low, the maintenance strategy may not perform well when evaluated on the $10^{5}$ validation scenarios as extreme events are not taken into account during the optimization.

\begin{figure}[htbp]
    \centering
    \begin{minipage}[t]{.49\textwidth}
        \includegraphics[width = \textwidth]{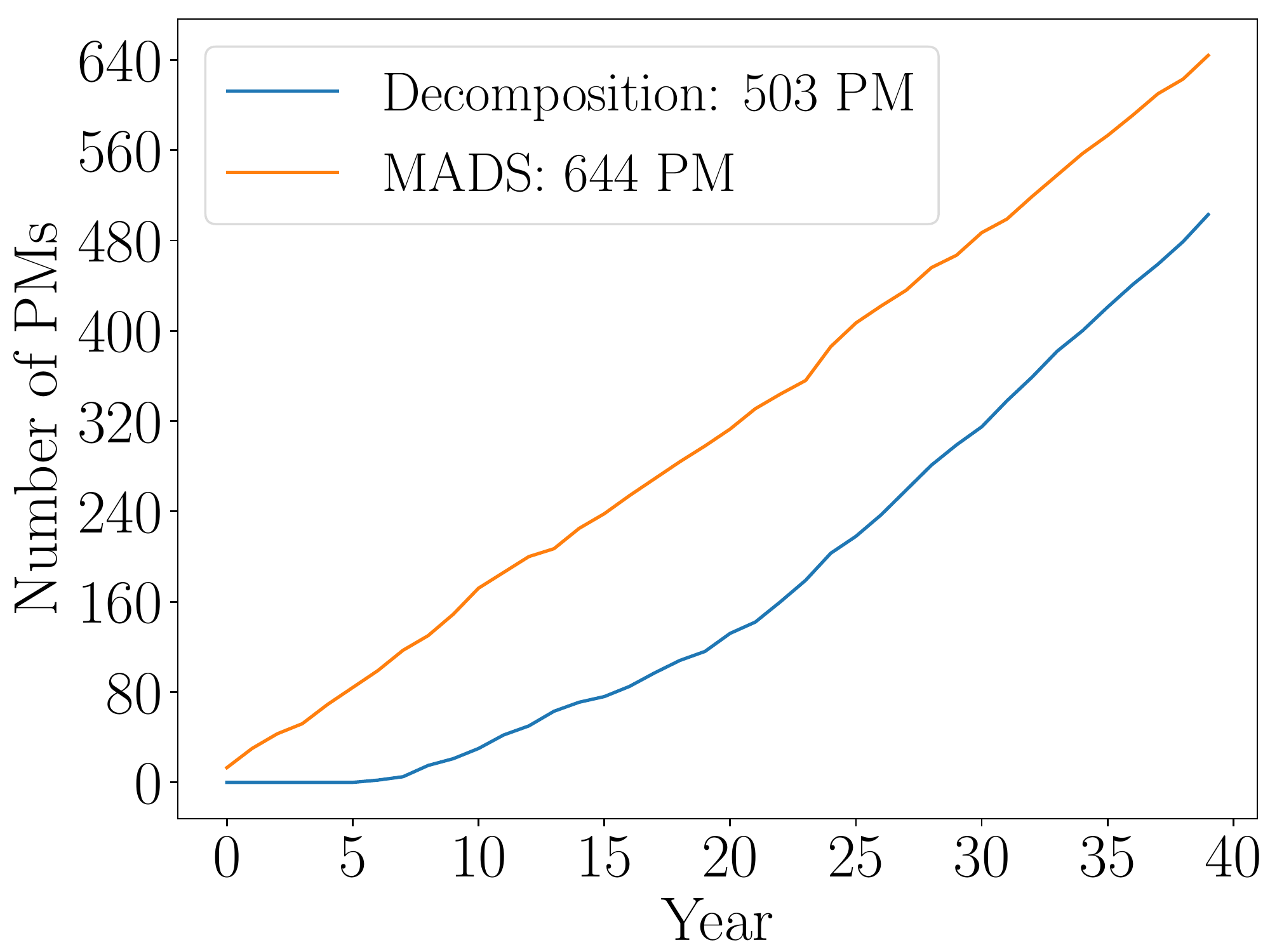}
        \caption{Cumulative number of PMs.}
        \label{fig:pm_cumul2}
    \end{minipage}
    \hfill
    \begin{minipage}[t]{.49\textwidth}
        \includegraphics[width=\textwidth]{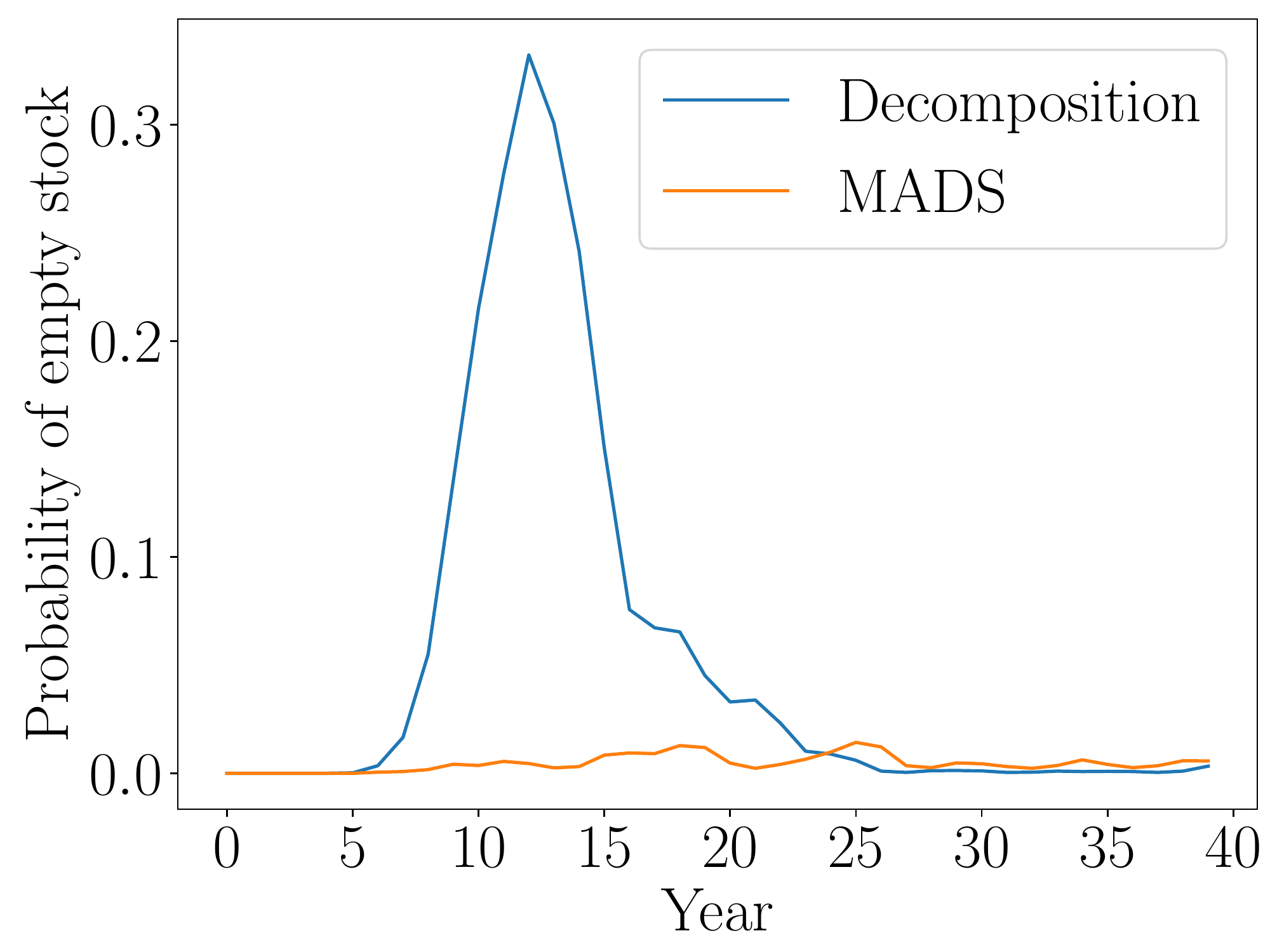}
        \caption{Evolution of the probability of having an empty stock.}
        \label{fig:low_sto2}
    \end{minipage}
\end{figure}

Overall for this case, the decomposition strategy performs better in expectation, which is the criterion that the problem has been designed to optimize. However, the decomposition strategy has a higher variance than MADS strategy with an important risk of forced outage. Thus, from an operational point of view, a decision maker will probably choose the MADS strategy. Our work is meant to be a decision support tool and is not design to take the final operational decision. The analysis of Fig.~\ref{fig:pm_cumul2} and~\ref{fig:low_sto2} suggest to investigate a maintenance strategy that is similar to the decomposition strategy in the first years and in the second half of the horizon but with more PMs between years $10$ and $15$ to reduce the number of forced outages. We have the insight that this new strategy could improve the robustness of the decomposition strategy and perform similarly or even better in expectation.

%%% Local Variables:
%%% mode: latex
%%% TeX-master: "optim_maint_sched_v2"
%%% End:

%% file: conclusion.tex
% !TEX root = optim_maint_sched_v2.tex

\section{Conclusion}
\label{sec:ccl}

In this work we study a maintenance scheduling optimization problem for hydropower plants management. 
% The problem modelled in this paper is a simplified version of an industrial problem. 
We set up a decomposition method to find a deterministic preventive maintenance strategy for a system of physical components sharing a common stock of spare parts. The decomposition relies on the Auxiliary Problem Principle. We construct a sequence of auxiliary problems that are solved iteratively. 
% The sequence of auxiliary solutions converge to the solution of the original problem under convexity and regularity assumptions. 
The auxiliary problems are decomposable into independent subproblems of smaller dimension that are solved in parallel. Each subproblem involves only one component of the system or the stock. A relaxation of the system is necessary as the gradients of the dynamics and the cost are required to compute the coordination terms. Then, the fixed-point algorithm is implemented with an efficient mixed parallel/sequential strategy. This implementation is particularly adapted to the structure of the industrial problem, for which the subproblem on the stock is easy to solve numerically.

On the two industrial test cases, the decomposition method outperforms in expectation the blackbox algorithm MADS applied directly on the full problem. The reason is that when MADS is used to solve subproblems of small dimension within the decomposition algorithm, it manages to detect features to design an efficient time-dependent PM strategy. These features -- such as not replacing components in the first years of the horizon -- cannot be detected when MADS is applied directly on the full problem as the corresponding strategies only represent a very small subset of the search space which is not explored by the algorithm.

The approach of the decomposition strategy differs between the two test cases. In the first case the strategy is conservative and robust to extreme events whereas in the second test case, it is more risky as more forced outages can occur. Our work is meant to be a decision support tool and even if it is unlikely that a decision maker chooses a risky maintenance strategy, the results on the second case can give insights for the design of an efficient robust maintenance strategy.

This work proves the interest of the modeling effort needed to apply the decomposition method. Some challenges still remain for an application in an operational context. In order to force the design of a robust maintenance strategy, some risk criteria should be added in the mathematical formulation of the problem, either in the cost function or in the constraint. Another point that must be noticed is that the dynamics is simulated with a time step of one year. A smaller time step must be used for an accurate evaluation of the costs. This will not increase the complexity of the problem as maintenance decisions are always made on a yearly basis, so that the space of admissible maintenance strategies is still of the same dimension. However, the time needed for the evaluation of the cost function will increase. 
% To keep a reasonable execution time, we could make fewer iterations of the fixed point algorithm and fewer evaluations in the subproblem resolution. 
% This seems possible as we see that a convergence of the strategy after few fixed-point iterations.
It is also possible to model more complex systems, by adding a control on the time of the order of spare parts or dependence between the failures of the components for instance. We could also consider imperfect preventive maintenance. A balance must be found between the simplicity of the model and its adequation to reality given the industrial application in mind.

%%% Local Variables:
%%% mode: latex
%%% TeX-master: "optim_maint_sched_v2"
%%% End:

%% file: appendix.tex
% !TEX root = optim_maint_sched_v2.tex

\section{Explicit expression of the dynamics of the original system}
\label{sec:rew_syst_dyn}
In this part, we give an explicit expression for the dynamics $f_{i}$ of component $i \in \components$ that appears in~\textup {(\ref{eq:dyn_comp_i})}. We can write:
\begin{equation}
    \statecomptime{i}{t+1} =
    \begin{pmatrix}
        \reg{i}{t+1}\\
        \age{i}{t+1}\\
        \spr{i}{t+1}
    \end{pmatrix}
    =
    \begin{pmatrix}
        f_{i,\va{E}}(\statecomptime{1:i}{t}, \sto{t}, \ctrlcomptime{i}{t}, \noise{i}{t+1})\\
        f_{i,\va{A}}(\statecomptime{1:i}{t}, \sto{t}, \ctrlcomptime{i}{t}, \noise{i}{t+1})\\
        f_{i,\va{P}}(\statecomptime{1:i}{t}, \sto{t}, \ctrlcomptime{i}{t}, \noise{i}{t+1})\\
    \end{pmatrix} \eqfinv
\end{equation}
so that $f_{i} = (f_{i,\va{E}}, f_{i,\va{A}}, f_{i,\va{P}})$. We give an explicit formula for $f_{i,\va{E}}, f_{i,\va{A}}$ and $f_{i,\va{P}}$.

\subsection{Dynamics of the regime \texorpdfstring{$\boldsymbol{E}_{i,t}$}{E}}

Using~Fig.~\ref{fig:comp_dyn}, we can write:
\begin{equation}
    \begin{aligned}
        \reg{i}{t+1} &= f_{i,\va{E}}(\statecomptime{1:i}{t}, \sto{t}, \ctrlcomptime{i}{t}, \noise{i}{t+1})\\
        % &=\begin{multlined}[t]
            % \underbrace{\indic{\bbRp}{\sto{t} - \sum_{j=1}^{i}\indic{\na{0}}{\reg{i}{t}}}}_{\text{Enough spare parts}}\indic{\na{0}}{\reg{i}{t}} \\
        %   + \big(\underbrace{\indic{\na{1}}{\ctrlcomptime{i}{t}}}_{\text{PM}} + \underbrace{\indic{\bbRp}{\noise{i}{t+1} - p_{i}(\age{i}{t})}}_{\text{No failure}}\underbrace{\indic{\na{0}}{\ctrlcomptime{i}{t}}}_{\text{No PM}} \big) \indic{\na{1}}{\reg{i}{t}}
        % \end{multlined}
        &=\begin{multlined}[t]
            \indic{\bbRp}{\sto{t} - \sum_{j=1}^{i}\indic{\na{0}}{\reg{j}{t}}}\indic{\na{0}}{\reg{i}{t}} \\
            + \big(\indic{\bbRp}{\ctrlcomptime{i}{t} - \nu} + \indic{\bbRp}{\noise{i}{t+1} - p_{i}(\age{i}{t})}\indic{\bbRpe}{\nu - \ctrlcomptime{i}{t}} \big) \indic{\na{1}}{\reg{i}{t}}\eqfinp
        \end{multlined}
    \end{aligned}
    \label{eq:phys_dyn}
\end{equation}
% \begin{multline}
%     \reg{i}{t+1} =  \indic{\bbRp}{\sto{t} - \sum_{j=1}^{i}(1 - \reg{j}{t})}\indic{\na{0}}{\reg{i}{t}} \\
%     + \left(\indic{\bbRp}{\ctrlcomptime{i}{t}-\nu} + \indic{\bbRp}{\noise{i}{t+1} - p_{i}(\va{A}_{i,t})}\indic{\bbRpe}{\nu - \ctrlcomptime{i}{t}} \right) \indic{\na{1}}{\reg{i}{t}}
% \label{eq:phys_dyn}
% \end{multline}

The first part of~\textup {(\ref{eq:phys_dyn})} means that if the component is broken at $t$ and we have enough spares to repair it, it is then functioning at $t+1$. The second part means that if the component is functioning at $t$ and we do a PM, it is still functioning at $t+1$. Finally, if we do not do a PM, the regime depends on the occurrence of a failure between $t$ and $t+1$.

\subsection{Dynamics of the age \texorpdfstring{$\boldsymbol{A}_{i,t}$}{A}}

Again using~Fig.~\ref{fig:comp_dyn}, we can write:
\begin{equation}
    \begin{aligned}
        \age{i}{t+1} &= f_{i,\va{A}}(\statecomptime{1:i}{t}, \sto{t}, \ctrlcomptime{i}{t}, \noise{i}{t+1})\\
        % &=\begin{multlined}[t]
        %     (\age{i}{t} + 1)\bigg[ \underbrace{\indic{\bbRpe}{\sum_{j=1}^{i}\indic{\na{0}}{\reg{j}{t}} - \sto{t}}}_{\text{Stays broken}}\indic{\na{0}}{\reg{i}{t}}\\
        %     + \big(\underbrace{\indic{\bbRp}{\noise{i}{t+1} - p_{i}(\age{i}{t})}\indic{\na{0}}{\ctrlcomptime{i}{t}}}_{\text{No PM and no failure}}\big) \indic{\na{1}}{\reg{i}{t}}\bigg]
        % \end{multlined}
        &=\begin{multlined}[t]
            (\age{i}{t} + 1)\bigg( \indic{\bbRpe}{\sum_{j=1}^{i}\indic{\na{0}}{\reg{j}{t}} - \sto{t}}\indic{\na{0}}{\reg{i}{t}}\\
            + \indic{\bbRp}{\noise{i}{t+1} - p_{i}(\age{i}{t})}\indic{\bbRpe}{\nu - \ctrlcomptime{i}{t}} \indic{\na{1}}{\reg{i}{t}}\bigg)\\
            + \indic{\bbRp}{\sto{t} - \sum_{j=1}^{i}\indic{\na{0}}{\reg{j}{t}}}\indic{\na{0}}{\reg{i}{t}} + \left((1-\ctrlcomptime{i}{t})\age{i}{t} + 1\right)\indic{\bbRp}{\ctrlcomptime{i}{t} - \nu}\indic{\na{1}}{\reg{i}{t}}\eqfinp
        \end{multlined}
    \end{aligned} 
    \label{eq:age_dyn}
\end{equation}

% \begin{equation}
%     \begin{aligned}
%         \age{i}{t}{i}{t+1} =&
%         \begin{multlined}[t]
%             (\va{A}_{i,t} + 1)\bigg[ \indic{\bbRpe}{\sum_{j=1}^{i}(1-\reg{j}{t}) - \sto{t}}\indic{\na{0}}{\reg{i}{t}}\\
%             + \left((1-\ctrlcomptime{i}{t})\indic{\bbRp}{\ctrlcomptime{i}{t}- \nu} + \indic{\bbRp}{\noise{i}{t+1} - p_{i}(\va{A}_{i,t})}\indic{\bbRpe}{\nu - \ctrlcomptime{i}{t}}\right) \indic{\na{1}}{\reg{i}{t}}\bigg] 
%         \label{eq:age_dyn}
%         \end{multlined}
%     \end{aligned}
% \end{equation}

If the component is broken at $t$, it stays broken if there are not enough spares in the stock (first line of~\eqref{eq:age_dyn}). In this case the time $\age{i}{t}$ increases by $1$. If the component is healthy at $t$, it ages if no PM is done and no failure occurs (second line of~\eqref{eq:age_dyn}). When the component is broken and that there are enough spares (first term in the third line of~\eqref{eq:age_dyn}), a CM is performed and we have $\age{i}{t+1} = 1$. In the case of a PM (second term in the third line of~\eqref{eq:age_dyn}), the component is rejuvenated and its age becomes $\age{i}{t+1} = (1-\ctrlcomptime{i}{t})\age{i}{t} + 1$. If there is a failure, we have $\age{i}{t+1} = 0$.

\subsection{Dynamics of the vector of times since last failures \texorpdfstring{$\boldsymbol{P}_{i,t}$}{P}}

The expression of the dynamics of $\spr{i}{t} = (\spr{i}{t}^{1}, \ldots, \spr{i}{t}^{\delay})$ is more complex. We write:
\begin{align}
    \spr{i}{t+1} &=
    \begin{pmatrix}
        \spr{i}{t+1}^{1}\\
        \vdots\\
        \spr{i}{t+1}^{\delay}
    \end{pmatrix}
    =
    \begin{pmatrix}
        f_{i,\va{P}}^{1}(\statecomptime{1:i}{t}, \sto{t}, \ctrlcomptime{i}{t}, \noise{i}{t+1})\\
        \vdots\\
        f_{i,\va{P}}^{\delay}(\statecomptime{1:i}{t}, \sto{t}, \ctrlcomptime{i}{t}, \noise{i}{t+1})
    \end{pmatrix}
    = f_{i,\va{P}}(\statecomptime{1:i}{t}, \sto{t}, \ctrlcomptime{i}{t}, \noise{i}{t+1})\eqfinv \\
    \intertext{so that $f_{i,\va{P}} = (f_{i,\va{P}}^{1}, \ldots, f_{i,\va{P}}^{\delay})$. We give the expression of $f_{i,\va{P}}^{d}$ for $d \in \{ 1, \ldots, \delay \}$:}
    \spr{i}{t+1}^{d} &=
    \begin{multlined}[t]
        \Big((\spr{i}{t}^{d} + 1) \indic{\bbRp}{\spr{i}{t}^{d}} + \sprnofail \indic{\na{\sprnofail}}{\spr{i}{t}^{d}} \Big)\Big(1- \indic{\na{1}}{\reg{i}{t}} \indic{\na{0}}{\reg{i}{t+1}}\Big)\\
        + \Big((\spr{i}{t}^{d} + 1)\indic{\bbRp}{\spr{i}{t}^{d}}\indic{\na{\sprnofail}}{\spr{i}{t}^{\delay}} + \sprnofail \indic{\na{\sprnofail}}{\spr{i}{t}^{d-1}}\indic{[2, \delay]}{d}\\
        + (\spr{i}{t}^{d+1} + 1)\indic{\bbRp}{\spr{i}{t}^{\delay}}\indic{[1, \delay-1]}{d} \Big) \indic{\na{1}}{\reg{i}{t}} \indic{\na{0}}{\reg{i}{t+1}}\eqfinp
    \end{multlined}
    \label{eq:fip}
\end{align}
The first line represents the case where there is no failure. Then $\spr{i}{t}$ increases by one if it is different from $\sprnofail$, otherwise it keeps the value $\sprnofail$. When there is a failure, if component $i$ has undergone fewer than $\delay$ failures, the evolution of $\spr{i}{t}$ is described by~\textup {(\ref{eq:nb_fail_fail})}. This case is represented by the second line of~\textup {(\ref{eq:fip})}. When the component has already undergone $\delay$ failures, the evolution of $\spr{i}{t}$ is described by~\textup {(\ref{eq:nb_fail_leqk_fail})}. This case is represented by the third line of~\textup {(\ref{eq:fip})}. Note that this expression of $\spr{i}{t+1}^{d}$ depends on $\reg{i}{t+1}$. It is possible to express $\spr{i}{t+1}^{d}$ only with variables describing component $i$ at time $t$, this can be done by replacing $\reg{i}{t+1}$ by its expression~\textup {(\ref{eq:phys_dyn})}.

% This concludes the derivation of an explicit expression of the dynamics of a component. 

\subsection{Dynamics of the stock \texorpdfstring{$\boldsymbol{S}_{t}$}{S}}

We recall the explicit dynamics of the stock that is already given in~\textup {(\ref{eq:stock})}:

\begin{equation}
    \sto{t+1} = \sto{t} + \sum_{i=1}^{n} \sum_{d=1}^{\delay} \indic{\na{\delay-1}}{\spr{i}{t}^{d}} - \min\left\{\sto{t}, \ \sum_{i=1}^{n} \indic{\na{0}}{\reg{i}{t}} \right\}\eqfinp
    \label{eq:stock_appendix}
\end{equation}
The dynamics of the whole system has now been explicitly described.

\section{Explicit expression of the dynamics of the relaxed system}
\label{sec:rel_dyn}
The expression of the relaxed dynamics of parameter $\alpha$ is obtained from Equations~\textup {(\ref{eq:phys_dyn})}, \textup {(\ref{eq:age_dyn})}, \textup {(\ref{eq:fip})} and~\textup {(\ref{eq:stock_appendix})} by replacing the indicator function with its relaxed version.

We do not always substitute directly the indicator with its relaxation. The dynamics often involves conditions on complementary events. For example, the condition \textit{if the component is broken} is represented by $\indic{\na{0}}{\reg{i}{t}}$. On the other hand, the condition \textit{if the component is healthy} is represented by $\indic{\na{1}}{\reg{i}{t}}$. For the original dynamics as $\reg{i}{t} \in \{0,1\}$ we always have
\begin{equation}
    \indic{\na{0}}{\reg{i}{t}} + \indic{\na{1}}{\reg{i}{t}} = 1\eqfinp
\end{equation}
This relation is not true anymore using directly the relaxed version of the indicator with the relaxed variables. Take for example $\alpha = 2$, and suppose $\reg{i}{t} = \frac{1}{2}$, then
\begin{equation}
    \indicrel{\na{0}}{\reg{i}{t}} + \indicrel{\na{1}}{\reg{i}{t}} = 0\eqfinp
\end{equation}
If we replace directly all indicator functions by their relaxation, the consequence would be in this case that $\reg{i}{t+1} = 0$ no matter the control $\ctrlcomptime{i}{t}$. This means that even if we do a PM with $\ctrlcomptime{i}{t} = 1$, the component is down at $t+1$. This does not represent the dynamics of the system as we would expect. To design a coherent relaxed dynamics, complementary conditions $\findi{\mathcal{A}}$ and $\findi{\mathcal{A}\compl}$ are represented  using the relaxed version $\findi{\mathcal{A}}\rel$ of the indicator function for the first condition and the function $1 - \findi{\mathcal{A}}\rel$ for the complementary condition.

\subsection{Relaxed dynamics of the regime \texorpdfstring{$\boldsymbol{E}_{i,t}$}{E}}

The relaxed dynamics of the regime of parameter $\alpha > 0$ is given by:
\begin{equation}
    \begin{aligned}
        \reg{i}{t+1} &= f_{i,\va{E}}\rel(\statecomptime{1:i}{t}, \sto{t}, \ctrlcomptime{i}{t}, \noise{i}{t+1})\\
        &=
        \begin{multlined}[t]
            \indicrel{\bbRp}{\sto{t} - \sum_{j=1}^{i}\indicrel{\na{0}}{\reg{j}{t}}}\indicrel{\na{0}}{\reg{i}{t}}
            + \Big(\indicrel{\bbRp}{\ctrlcomptime{i}{t} - \nu}\\ 
            + \indicrel{\bbRp}{\noise{i}{t+1} - p_{i}(\age{i}{t})}(1-\indicrel{\bbRp}{\ctrlcomptime{i}{t} - \nu}) \Big) \left(1-\indicrel{\na{0}}{\reg{i}{t}}\right)\eqfinp
        \end{multlined}
    \end{aligned}
\end{equation}
We use the relaxed version of the indicator for $\indic{\na{0}}{\reg{i}{t}}$ and $\indic{\bbRp}{\ctrlcomptime{i}{t} - \nu}$. We relax $\indic{\na{1}}{\reg{i}{t}}$ and $\indic{\bbRpe}{\nu - \ctrlcomptime{i}{t}}$ as $1-\indicrel{\na{0}}{\reg{i}{t}}$ and $1-\indicrel{\bbRp}{\ctrlcomptime{i}{t} - \nu}$ respectively.

\subsection{Relaxed dynamics of the age \texorpdfstring{$\boldsymbol{A}_{i,t}$}{A}}

The relaxed dynamics of the age of parameter $\alpha > 0$ is given by:
\begin{equation}
    \begin{aligned}
        \age{i}{t+1} &= f_{i,\va{A}}\rel(\statecomptime{1:i}{t}, \sto{t}, \ctrlcomptime{i}{t}, \noise{i}{t+1})\\
        &=
        \begin{multlined}[t]
            (\age{i}{t} + 1)\bigg[ 
                \indicrel{\bbRpe}{\sum_{j=1}^{i}\indicrel{\na{0}}{\reg{j}{t}} - \sto{t}}\indicrel{\na{0}}{\reg{i}{t}} \\
                + \indicrel{\bbRp}{\noise{i}{t+1} - p_{i}(\age{i}{t})}(1-\indicrel{\bbRp}{\ctrlcomptime{i}{t} - \nu}) \big(1-\indicrel{\na{0}}{\reg{i}{t}}\big)
                \bigg]\\
            + \big(1-\indicrel{\bbRpe}{\sum_{j=1}^{i}\indicrel{\na{0}}{\reg{j}{t}} - \sto{t}}\big)\indicrel{\na{0}}{\reg{i}{t}}\\
            + \left((1-\ctrlcomptime{i}{t})\age{i}{t} + 1\right)\indicrel{\bbRp}{\ctrlcomptime{i}{t} - \nu}\big(1-\indicrel{\na{0}}{\reg{i}{t}}\big)\eqfinp
        \end{multlined}
    \end{aligned}
\end{equation}

\subsection{Relaxed dynamics of the vector of last failures \texorpdfstring{$\boldsymbol{P}_{i,t}$}{P}}

The relaxed dynamics of parameter $\alpha > 0$ of the $d$-th element of the vector of last failures is given by:
\begin{equation}
    \begin{aligned}
        \spr{i}{t+1}^{d} &= f_{i,\va{P}}^{d, \alpha}(\statecomptime{1:i}{t}, \sto{t}, \ctrlcomptime{i}{t}, \noise{i}{t+1})\\
        &=
        \begin{multlined}[t]
            \Big((\spr{i}{t}^{d} + 1) (1-\indicrel{\na{\sprnofail}}{\spr{i}{t}^{d}}) + \sprnofail\indicrel{\na{\sprnofail}}{\spr{i}{t}^{d}} \Big)\Big(1- \indicrel{\na{1}}{\reg{i}{t}} \indicrel{\na{0}}{\reg{i}{t+1}}\Big)\\
            + \Big((\spr{i}{t}^{d} + 1)(1-\indicrel{\na{\sprnofail}}{\spr{i}{t}^{d}})\indicrel{\na{\sprnofail}}{\spr{i}{t}^{\delay}} + \sprnofail \indicrel{\na{\sprnofail}}{\spr{i}{t}^{d-1}}\indic{[2, \delay]}{d}\\
            + (\spr{i}{t}^{d+1} + 1)(1-\indicrel{\na{\sprnofail}}{\spr{i}{t}^{\delay}})\indic{[1, \delay-1]}{d} \Big) \indicrel{\na{1}}{\reg{i}{t}} \indicrel{\na{0}}{\reg{i}{t+1}}\eqfinp
        \end{multlined}
    \end{aligned}
    \label{eq:relax_dyn_spr}
\end{equation}

We do not relax $\indic{[2, D]}{d}$ and $\indic{[1, \delay-1]}{d}$. The reason is that these indicator functions do not arise from a discontinuity in the original dynamics. They are just used to take into account in the same equation the cases of $\spr{i}{t}^{1}$ and $\spr{i}{t}^{\delay}$ that have a slightly different expression than $\spr{i}{t}^{d}$ for $1 < d < \delay$.

\subsection{Relaxed dynamics of the stock \texorpdfstring{$\boldsymbol{S}_{t}$}{S}}

The relaxed dynamics of the stock of parameter $\alpha > 0$ is given by:
\begin{equation}
    \sto{t+1} = \sto{t} + \sum_{i=1}^{n} \sum_{d=1}^{\delay} \indicrel{\na{\delay-1}}{\spr{i}{t}^{d}} - \min\left\{\sto{t}, \ \sum_{i=1}^{n} \indicrel{\na{0}}{\reg{i}{t}} \right\}\eqfinp
\end{equation}

\section{Computation of optimal multipliers}
\label{sec:opt_mult}
At iteration $k$ of the APP fixed-point algorithm, the subproblem on component $i \in \components$ is solved with the blackbox algorithm MADS~\cite{audet_mesh_2006}. MADS directly solves the constrained problem and outputs a primal solution $(\statecomponent{i}^{k+1}, \ctrlcomp{i}^{k+1})$. Finding the primal solution $\stoalltime^{k+1}$ of the subproblem on the stock just requires a simulation of the dynamics. For each subproblem we also have to compute optimal multipliers $\multcomp{1}^{k+1}, \ldots, \multcomp{n}^{k+1}, \multcomp{\stoalltime}^{k+1}$ to update the coordination term at the end of each iteration. 

Suppose that the optimal solution and optimal multiplier of the auxiliary problem~\textup {(\ref{eq:aux_pb_indus})} are uniquely defined. 
% As the constraints in the subproblems only result from the dynamics of the system, the optimal multiplier corresponds to the adjoint state. 
% It can be computed using the discrete-time maximum principle. We derive the computation of the optimal multiplier from the Lagrangian of the auxiliary problem. 
As we know the primal solution, we can compute the optimal multiplier using the stationarity of the Lagrangian. In the following computation, we use the relaxed cost and dynamics to be able to compute the different gradients that appear, however for the sake of readability we drop the superscript $\alpha$. 

The Lagrangian $\lag$ of the auxiliary problem~\textup {(\ref{eq:aux_pb_indus})} is given by:
\begin{align}
    \lag(\stateallcomp, \stoalltime, \ctrlallcomp, \multallcomp) = &
    \begin{multlined}[t]
        \ \espe \Big(\sum_{i=1}^{n} \left(j_{i}(\statecomponent{i}, \ctrlcomp{i}) + \costfo(\statecompbar{1:i-1}, \statecomponent{i}, \statecompbar{i+1:n})\right)\\
        + \frac{\gamma_{x}}{2} \sqnorm{ \stateallcomp - \stateallcompbar } + \frac{\gamma_{s}}{2} \sqnorm{ \stoalltime - \stoalltimebar } + \frac{\gamma_{u}}{2} \sqnorm{\ctrlallcomp- \ctrlallcompbar }\\
        \ + \proscal{ \multallcompbar}{ (\dyn'(\prevopt, \noiseallcomp) - \auxdyn'(\prevopt, \noiseallcomp))\cdot (\stateallcomp, \stoalltime, \ctrlallcomp) }\\ 
        +\proscal{ \multallcomp}{ \auxdyn(\stateallcomp, \stoalltime, \ctrlallcomp, \noiseallcomp) } \Big)\eqfinp
    \end{multlined}
\end{align}
% % \todo[inline]{Utiliser un ou le pour pt selle et multiplicateur optimal?}
At the saddle point $(\stateallcompopt, \stoalltimeopt, \ctrlallcompopt, \multallcompopt)$ of $\lag$ we have: 
\begin{equation}
    \nabla \lag(\stateallcompopt, \stoalltimeopt, \ctrlallcompopt, \multallcompopt) = 0\eqfinp
    \label{eq:lag_opt}
\end{equation}
Recall that for $i \in \{1, \ldots, n, \stoalltime\}$, we have:
\begin{equation}
    \multcomp{i}^{k+1} = (\multcomptime{i}{0}^{k+1}, \ldots, \multcomptime{i}{T}^{k+1})\eqfinp
\end{equation}
Using~\textup {(\ref{eq:lag_opt})} and knowing the solution $(\stateallcompopt, \stoalltimeopt, \ctrlallcompopt)$ of the auxiliary problem, we can update the multiplier $\multallcompopt$ with a backward recursion.

% \multcomptime{i}{0} does not play any role in the optimization as the constraint is never violated at t=0, neither for the components nor the stock, that is why it does not appear in the code where $\multallcomp$ has only dimension $T$ compared to $T+1$ in the report.

\begin{proposition}
    Let $i \in \components$. For the dynamics of component $i$ in the auxiliary problem~\textup {(\ref{eq:aux_pb_indus})}, the optimal multiplier $\multcompopt{i} = (\multcomptimeopt{i}{0}, \ldots, \multcomptimeopt{i}{T})$ can be computed with the following backward recursion for $t\in \timesteps$:
    \begin{equation}
        \begin{aligned}
            \multcomptimeopt{i}{T} &=
            \begin{multlined}[t]
                - \nabla_{\statecomptime{i}{T}} j_{i,T}(\statecomptimeopt{i}{T}, \ctrlcomptimeopt{i}{T}) - \nabla_{\statecomptime{i}{T}}\costfo_{T}(\statecompbar{1:i-1,T}, \statecomptimeopt{i}{T}, \statecompbar{i+1:n,T})\\
                - \gamma_{x} (\statecomptimeopt{i}{T} - \statecompbar{i,T})\eqfinv
            \end{multlined}\\
            \multcomptimeopt{i}{t} &=
            \begin{multlined}[t]
                - \nabla_{\statecomptime{i}{t}} j_{i,t}(\statecomptimeopt{i}{t}, \ctrlcomptimeopt{i}{t}) - \nabla_{\statecomptime{i}{t}}\costfo_{t}(\statecompbar{1:i-1,t}, \statecomptimeopt{i}{t}, \statecompbar{i+1:n,t})\\ 
                - \gamma_{x} (\statecomptimeopt{i}{t} - \statecompbar{i,t}) - \sum_{j=i+1}^{n} \partial_{\statecomptime{i}{t}} \dyncomptime{j}{t+1}(\statecompbar{1:j}, \stoalltimebar, \ctrlcompbar{j}, \noisecomp{j})\transpos \cdot \multcomptimebar{j}{t+1} \\ 
                - \partial_{\statecomptime{i}{t}} \dyncomptime{\stoalltime}{t+1}(\statecompbar{1:n}, \stoalltimebar)\transpos \cdot \multcomptimebar{\stoalltime}{t+1}
                - \partial_{\statecomptime{i}{t}} \auxdyncomptime{i}{t+1}(\statecompopt{i}, \ctrlcompopt{i}, \noisecomp{i})\transpos \cdot\multcomptimeopt{i}{t+1}\eqfinp
            \end{multlined}
        \end{aligned}
        \label{eq:back_comp}
    \end{equation}
    \label{prop:back_comp}
\end{proposition}

\begin{proof}
    The gradient of the Lagrangian $\lag$ with respect to $\statecomptime{i}{t}, \ t\in \timesteps$ is given by:
    \begin{equation}
        \begin{aligned}
            \nabla_{\statecomptime{i}{T}} \lag(\stateallcompopt, \stoalltimeopt, \ctrlallcompopt, \multallcompopt) &=  
            \begin{multlined}[t]
                \nabla_{\statecomptime{i}{T}} j_{i,T}(\statecomptimeopt{i}{T}, \ctrlcomptimeopt{i}{T})\\
                + \nabla_{\statecomptime{i}{T}}\costfo_{T}(\statecomptimebar{1:i-1}{T}, \statecomptimeopt{i}{T}, \statecomptimebar{i+1:n}{T})\\
                + \gamma_{x} (\statecomptimeopt{i}{T} - \statecompbar{i,T})
                + \partial_{\statecomptime{i}{T}} \auxdyncomptime{i}{T}(\statecompopt{i}, \ctrlcompopt{i}, \noisecomp{i}) \transpos \cdot\multcomptimeopt{i}{T}\eqfinv
            \end{multlined}\\
            \nabla_{\statecomptime{i}{t}} \lag(\stateallcompopt, \stoalltimeopt, \ctrlallcompopt, \multallcompopt) &=
            \begin{multlined}[t]
                \nabla_{\statecomptime{i}{t}} j_{i,t}(\statecomptimeopt{i}{t}, \ctrlcomptimeopt{i}{t})
                + \nabla_{\statecomptime{i}{t}}\costfo_{t}(\statecomptimebar{1:i-1}{t}, \statecomptimeopt{i}{t}, \statecomptimebar{i+1:n}{t})\\
                + \gamma_{x} (\statecomptimeopt{i}{t} - \statecompbar{i,t})
                + \partial_{\statecomptime{i}{t}} \dyncomptime{\stoalltime}{t+1}(\statecompbar{1:n}, \stoalltimebar)\transpos \cdot \multcomptimebar{\stoalltime}{t+1}\\
                + \sum_{j=i+1}^{n} \partial_{\statecomptime{i}{t}} \dyncomptime{j}{t+1}(\statecompbar{1:j}, \stoalltimebar, \ctrlcompbar{j}, \noisecomp{j})\transpos\cdot\multcomptimebar{j}{t+1}\\ 
                + \partial_{\statecomptime{i}{t}} \auxdyncomptime{i}{t}(\statecompopt{i}, \ctrlcompopt{i}, \noisecomp{i})\transpos\cdot\multcomptimeopt{i}{t}\\
                + \partial_{\statecomptime{i}{t}} \auxdyncomptime{i}{t+1}(\statecompopt{i}, \ctrlcompopt{i}, \noisecomp{i})\transpos\cdot\multcomptimeopt{i}{t+1}\eqfinp
            \end{multlined}
        \end{aligned}
    \label{eq:grad_lag_sta}
    \end{equation}
    Using
    \begin{align}
        \nabla_{\statecomptime{i}{t}} \lag(\stateallcompopt, \stoalltimeopt, \ctrlallcompopt, \multallcompopt) &= 0, \quad t \in \timesteps\eqfinv\\
        \partial_{\statecomptime{i}{t}} \auxdyncomptime{i}{t}(\statecompopt{i}, \ctrlcompopt{i}, \noisecomp{i}) &= \mathrm{I}, \quad t \in \timesteps\eqfinv
    \end{align}
    where $\mathrm{I}$ is the identity matrix of appropriate size, we get the formula~\textup {(\ref{eq:back_comp})}.
    \skipfinpreuve
\end{proof}

\begin{proposition}
    The optimal multiplier $\multcompopt{\stoalltime} = (\multcomptimeopt{\stoalltime}{0}, \ldots, \multcomptimeopt{\stoalltime}{T})$ associated to the dynamics of the stock in the auxiliary problem~\textup {(\ref{eq:aux_pb_indus})} can be computed with the following backward recursion for $t\in \timesteps$:
    \begin{equation}
        \begin{aligned}
            \multcomptimeopt{\stoalltime}{T} &= - \gamma_{s} (\stoopt{T} - \stobar{T})\eqfinv\\
            \multcomptimeopt{\stoalltime}{t} &=
            \begin{multlined}[t]
                - \gamma_{s} (\stoopt{t} - \stobar{t}) - \sum_{i=1}^{n} \partial_{\sto{t}} \dyncomptime{i}{t+1}(\statecompbar{1:i}, \stoalltimebar, \ctrlcompbar{i}, \noisecomp{i})\transpos \cdot \multcomptimebar{i}{t+1}\\ 
                - \partial_{\sto{t}} \auxdyncomptime{\stoalltime}{t+1}(\stoalltimeopt)\transpos \cdot \multcomptimeopt{\stoalltime}{t+1}\eqfinp
            \end{multlined}
        \end{aligned}
        \label{eq:back_stock}
    \end{equation}
    \label{prop:back_stock}
\end{proposition}

\begin{proof}
    The gradient of the Lagrangian $\lag$ with respect to $\sto{t}, \ t\in \timesteps$ is given by:
    \begin{equation}
        \begin{aligned}
            \nabla_{\sto{T}} \lag(\stateallcompopt, \stoalltimeopt, \ctrlallcompopt, \multallcompopt) &= \gamma (\stoopt{T} - \stobar{T}) + \partial_{\sto{T}} \auxdyncomptime{\stoalltime}{T}(\stoalltimeopt)\transpos \cdot \multcomptimeopt{\stoalltime}{T}\eqfinv\\
            \dersto{t} \lag(\stateallcompopt, \stoalltimeopt, \ctrlallcompopt, \multallcompopt) &=
            \begin{multlined}[t]
                \gamma (\stoopt{t} - \stobar{t})\\ 
                + \sum_{i=1}^{n} \partial_{\sto{t}} \dyncomptime{i}{t+1}(\statecompbar{1:i}, \stoalltimebar, \ctrlcompbar{i}, \noisecomp{i})\transpos \cdot \multcomptimebar{i}{t+1}\\
                + \partial_{\sto{t}} \auxdyncomptime{\stoalltime}{t+1}(\stoalltimeopt)\transpos\cdot \multcomptimeopt{\stoalltime}{t+1} + \partial_{\sto{t}} \auxdyncomptime{\stoalltime}{t}(\stoalltimeopt)\transpos\cdot \multcomptimeopt{\stoalltime}{t}\eqfinp
            \end{multlined}
        \end{aligned}
        \label{eq:grad_lag_sto}
    \end{equation}
    Using
    \begin{align}
        \nabla_{\sto{t}} \lag(\stateallcompopt, \stoalltimeopt, \ctrlallcompopt, \multallcompopt) &= 0, \quad t \in \timesteps\eqfinv\\
        \partial_{\sto{t}} \auxdyncomptime{\stoalltime}{t}(\stoalltimeopt) &= 1, \quad t \in \timesteps\eqfinv
    \end{align}
    we get the backward recursion~\textup {(\ref{eq:back_stock})}.
    \skipfinpreuve
\end{proof}

% \begin{remark}
%     We have used the non-adapted version of the adjoint state as $\multallcompopt$ is not measurable with respect to the natural filtration $(\trib_{t})$ of the problem. It is possible to derive an adapted version $\va{\Pi}^{\sharp}$ of the adjoint state by defining for all $t \in \timesteps$
%     \begin{equation}
%         \va{\Pi}^{\sharp}_{i, t} = \espe(\multcomptimeopt{i}{t} \vert \trib_{t})
%     \end{equation}
%     Refer to~\cite[Chapter II]{dallagi_methodes_2007} for more details.
% \end{remark}

% The proofs of~\crefrange{prop:back_comp}{prop:back_stock} are given in~\cref{sec:proofs_app}.

% All the methods and techniques for the implementation of the APP fixed-point algorithm have now been presented. For a practical use of the APP fixed-point algorithm, we have to compute the gradient of the relaxed dynamics $\dynrel$ and $\auxdynrel$ as well as the gradient of the costs $\costmrel_{i}$ and $\costforel$.

\section{Derivative of the relaxed indicator function}
\label{sec:indic_rel}

We give some details about the derivative of the relaxed indicator function. 
% The gradient of the relaxed cost and dynamics can then be derived using their explicit expression given in~\S\ref{subsubsec:relax_cost} and~Appendix~\ref{sec:rel_dyn} respectively.
The relaxed indicator function $\findi{\mathcal{A}}\rel$, where $\alpha > 0$, appears in the dynamics and cost with three main cases for the set $\mathcal{A} \subset \bbR$. Note that $\findi{\mathcal{A}}\rel$ is not differentiable at $x \in \bbR$ if $d(\mathcal{A}, x) = \frac{1}{2\alpha}$. At such point, the derivative is taken to be $0$. The following situations occur:
\begin{enumerate}
    \item $\mathcal{A}$ is a singleton $\na{a}$, then for $x \in \bbR$:
    \begin{equation}
        \indicrel{\na{a}}{x} =
        \left\{
        \begin{aligned}
            1 - 2\alpha \nabs{x - a} & \quad \text{if } \nabs{x - a} \leq \frac{1}{2\alpha}\eqfinv\\
            0 & \quad \text{if } \nabs{x - a} > \frac{1}{2\alpha}\eqfinp\\
        \end{aligned}
        \right.
    \end{equation}
    Hence the derivative $\findi{\na{a}}^{'\alpha}$ is given by:
    \begin{equation}
        \indicrelder{\na{a}}{x} =
        \left\{
        \begin{aligned}
            2\alpha & \quad \text{if } a - \frac{1}{2\alpha} < x < a\eqfinv\\
            - 2\alpha & \quad \text{if } a < x < a + \frac{1}{2\alpha}\eqfinv\\
            0 & \quad \text{otherwise }\eqfinp\\
        \end{aligned}
        \right.
    \end{equation}
    \item $\mathcal{A} = \bbRp$ then for $x \in \bbR$ we have:
    \begin{equation}
        \indicrel{\bbRp}{x} =
        \left\{
        \begin{aligned}
            2\alpha x + 1 & \quad \text{if } -\frac{1}{2\alpha} < x < 0\eqfinv\\
            1 & \quad \text{if } x \geq 0\eqfinv\\
            0 & \quad \text{if } x \leq -\frac{1}{2\alpha}\eqfinp\\
        \end{aligned}
        \right.
        \quad
        \indicrelder{\bbRp}{x} =
        \left\{
        \begin{aligned}
            2\alpha & \quad \text{if } -\frac{1}{2\alpha} < x < 0\eqfinv\\
            0 & \quad \text{otherwise}\eqfinp\\
        \end{aligned}
        \right.
    \end{equation}
    \item $\mathcal{A} = \bbRpe$: if we strictly apply~Definition~\ref{def:relax_indic}, we would have $\findi{\bbRpe}\rel = \findi{\bbRp}\rel$. However with this definition we would not have pointwise convergence of $\findi{\bbRpe}\rel$ towards $\findi{\bbRpe}$ as $\alpha$ goes to $0$. Indeed, for all $\alpha > 0$ we would have $\indicrel{\bbRpe}{0} = 1$ but $\indic{\bbRpe}{0} = 0$. To overcome this issue we define $\findi{\bbRpe}\rel$ as follows:
    \begin{equation}
        \indicrel{\bbRpe}{x} =
        \left\{
        \begin{aligned}
            2\alpha x & \quad \text{if } 0 < x < \frac{1}{2\alpha}\eqfinv\\
            1 & \quad \text{if } x \geq \frac{1}{2\alpha}\eqfinv\\
            0 & \quad \text{if } x \leq 0\eqfinp\\
        \end{aligned}
        \right.
        \quad 
        \indicrelder{\bbRpe}{x} =
        \left\{
        \begin{aligned}
            2\alpha & \quad \text{if } 0 < x < \frac{1}{2\alpha} \eqfinv\\
            0 & \quad \text{otherwise}\eqfinp\\
        \end{aligned}
        \right.
    \end{equation}
\end{enumerate}
Using these formulas for the derivative of the relaxed indicator function and the explicit expressions of the relaxed cost function and relaxed dynamics given in~\S\ref{subsubsec:relax_cost} and~Appendix~\ref{sec:rel_dyn} respectively, all the gradients that appear either in the objective function of the subproblems or in the backward recursion for the multiplier update can be computed.

% The exhaustive computation is done in the supplementary materials.

%%% Local Variables:
%%% mode: latex
%%% TeX-master: "optim_maint_sched_v2"
%%% End: